\documentclass[acmtog,timestamp]{acmart}

\usepackage{booktabs} % For formal tables
\usepackage{mystyle}
\usepackage{tabulary}

\acmISBN{123-4567-24-567/08/06}

\setcopyright{acmlicensed}
\acmJournal{TOG}
\acmYear{2018}\acmVolume{37}\acmNumber{6}\acmArticle{250}\acmMonth{11} \acmDOI{10.1145/3272127.3275064}

\begin{document}
\title{Dynamical Optimal Transport on Discrete Surfaces}
%An optimal transport-based Riemannian structure on the space of probability distributions on a discrete surface}

\author{Hugo Lavenant}
\affiliation{%
  \institution{Universit\'e Paris-Sud}
  %\streetaddress{P.O. Box 1212}
  \city{Orsay, France}
  %\postcode{91400}
}
\author{Sebastian Claici}
\affiliation{
 \institution{Massachusetts Institute of Technology}
 \city{Cambridge, MA}
}
\author{Edward Chien}
\affiliation{
 \institution{Massachusetts Institute of Technology}
 \city{Cambridge, MA}
}
\author{Justin Solomon}
\affiliation{
 \institution{Massachusetts Institute of Technology}
 \city{Cambridge, MA}
}

% The default list of authors is too long for headers.

\begin{abstract}
We propose a technique for interpolating between probability distributions on discrete surfaces, based on the theory of optimal transport. Unlike previous attempts that use linear programming, our method is based on a dynamical formulation of quadratic optimal transport proposed for flat domains by Benamou and Brenier \shortcite{Benamou2000}, adapted to discrete surfaces. Our structure-preserving construction yields a Riemannian metric on the (finite-dimensional) space of probability distributions on a discrete surface, which translates the so-called Otto calculus to discrete language. From a practical perspective, our technique provides a smooth interpolation between distributions on discrete surfaces with less diffusion than state-of-the-art algorithms involving entropic regularization. Beyond interpolation, we show how our discrete notion of optimal transport extends to other tasks, such as distribution-valued Dirichlet problems and time integration of gradient flows.
\end{abstract}

%
% The code below should be generated by the tool at
% http://dl.acm.org/ccs.cfm
% Please copy and paste the code instead of the example below.
%



\begin{teaserfigure}
\centering
\begin{tabular}{cccccccc}
	\includegraphics[width=.095\textwidth]{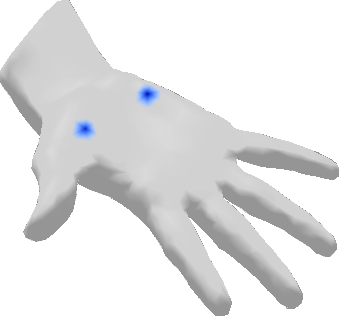} &
	\includegraphics[width=.095\textwidth]{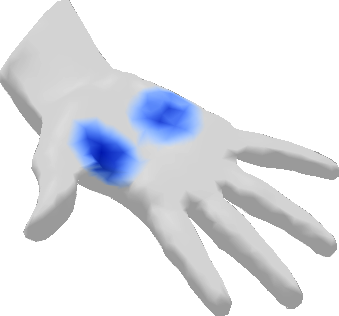} &
	\includegraphics[width=.095\textwidth]{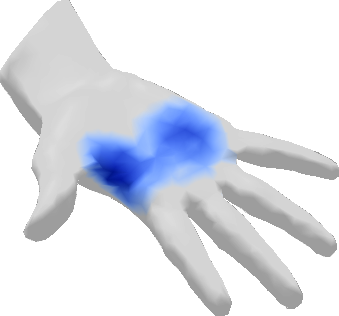} &
	\includegraphics[width=.095\textwidth]{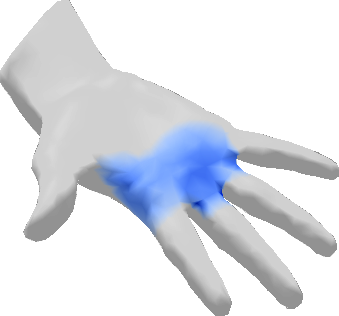} &
	\includegraphics[width=.095\textwidth]{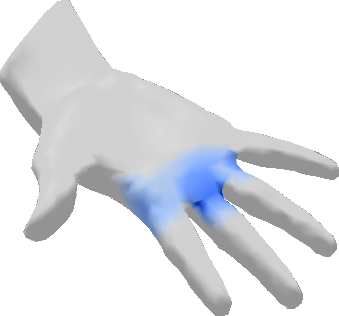} &
	\includegraphics[width=.095\textwidth]{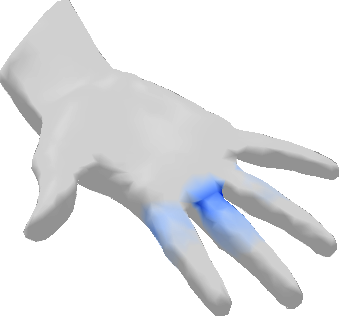} &
	\includegraphics[width=.095\textwidth]{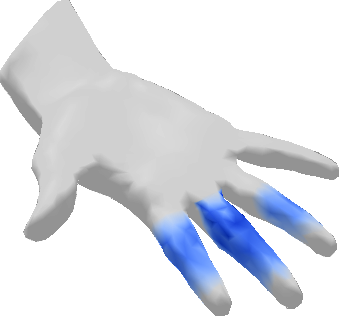} &
	\includegraphics[width=.095\textwidth]{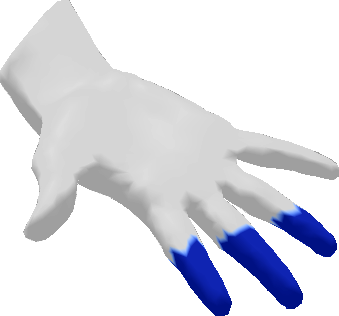} \\
	\includegraphics[width=.095\textwidth]{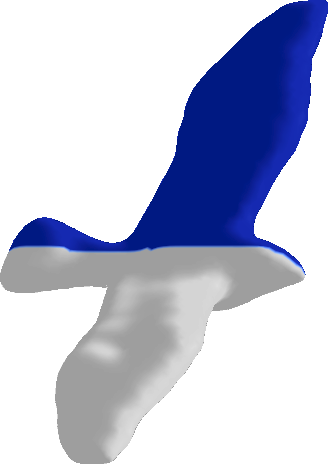} &
	\includegraphics[width=.095\textwidth]{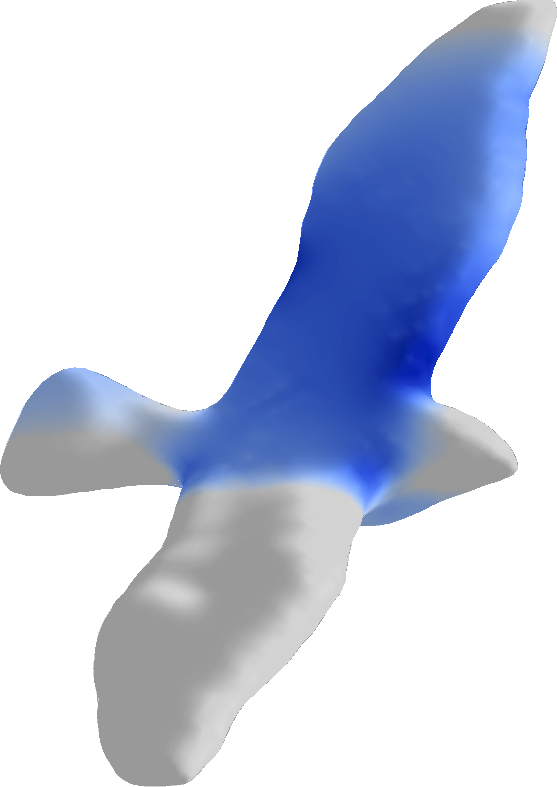} &
	\includegraphics[width=.095\textwidth]{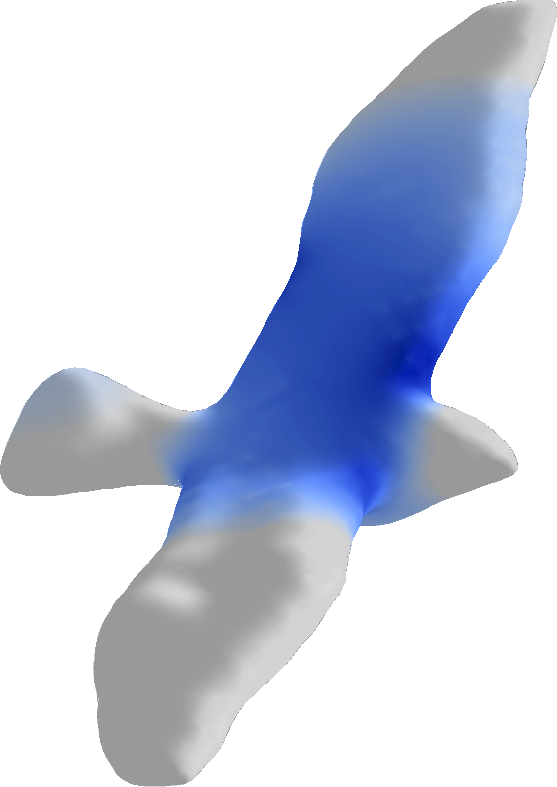} &
	\includegraphics[width=.095\textwidth]{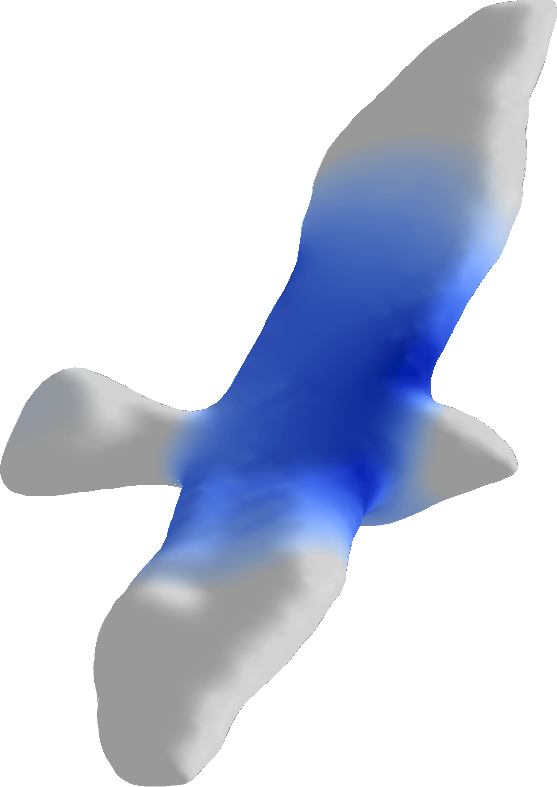} &
	\includegraphics[width=.095\textwidth]{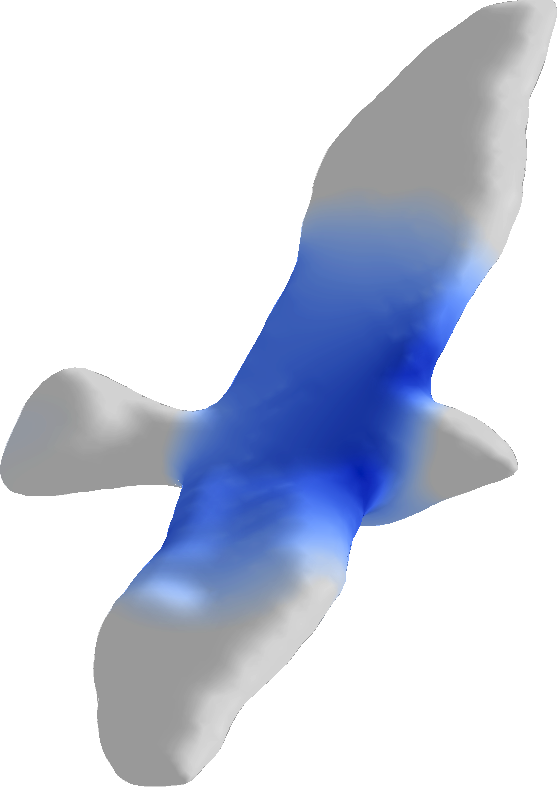} &
	\includegraphics[width=.095\textwidth]{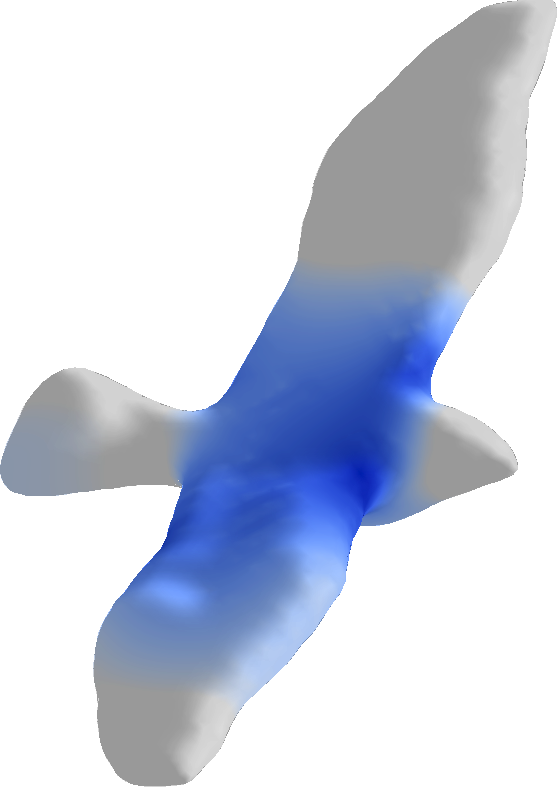} &
	\includegraphics[width=.095\textwidth]{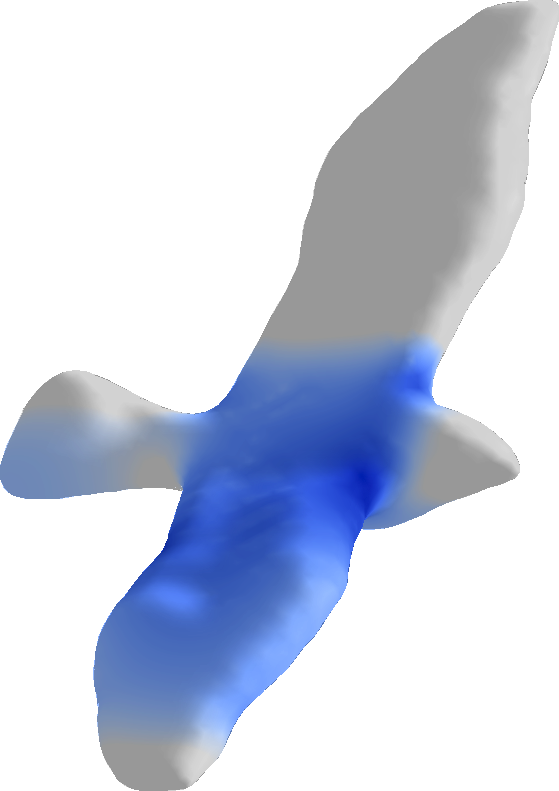} &
	\includegraphics[width=.095\textwidth]{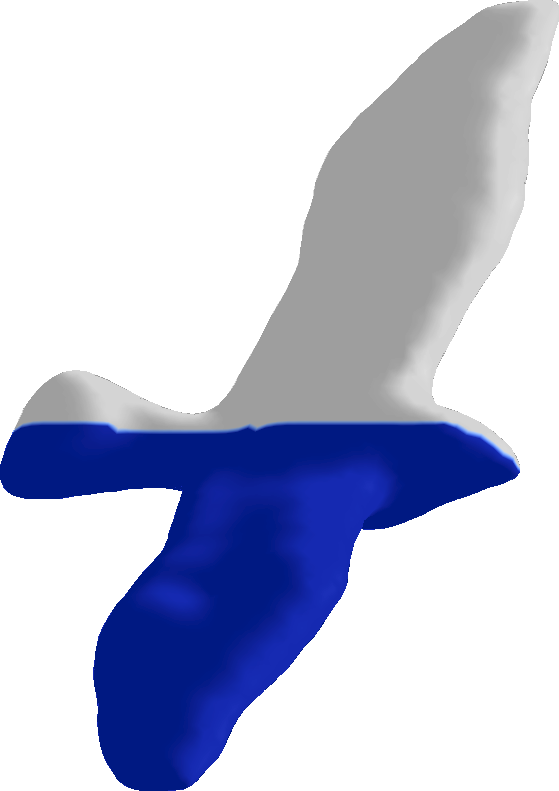}
\end{tabular}
\label{figure_teaser}
\caption{
\review{Given two probability distributions over a discrete surface (left and right), our algorithm generates an interpolation that takes the geometric structure of the surface into account.}
}
%The left and right figures represent given probability distributions (our input), the output is the interpolation in the middle.}
\end{teaserfigure}


\maketitle

\section{Introduction}

Probability distributions are key objects in geometry processing that can encode a variety of quantities, including uncertain feature locations on a surface, color histograms, %the intensity of a measurement, %<- wasn't sure what this means
and physical measurements like the density of a fluid.
A central problem related to distributions is that of \emph{interpolation}: Given two probability distributions over a fixed domain, how can one transition smoothly from the first to the second?

\emph{Optimal transport} gives one potential solution. This theory lifts the geometric structure of a surface to a Riemannian structure on the space of probability distributions over the surface, the latter being endowed with the so-called \emph{Wasserstein} metric; the set of distributions equipped with this metric is sometimes called Wasserstein space. To interpolate  between two probability distributions, one computes a \emph{geodesic} in  Wasserstein space between the two. This definition is sometimes referred to as McCann's displacement interpolation \shortcite{Mccann1997}, applied to graphics e.g.\ in \cite{bonneel2011displacement}.

Even though optimal transport theory is now well-understood \cite{Villani2003,Villani2008,Santambrogio2015}, the interpolation problem remains challenging numerically. Related problems, like the computation of Wasserstein distances or barycenters in Wasserstein space, can be tackled by fast and scalable algorithms like entropic regularization or semi-discrete methods, developed only a few years ago. Most of these methods, however, fail to reproduce the Riemannian structure of Wasserstein space and/or are prone to diffusion: The interpolation between two peaked probability distributions is more diffuse in the midpoint than optimal transport theory suggests. This drawback can inhibit application of transport in computer graphics practice, in which blurry interpolants are often undesirable.

As an alternative, we define a Riemannian structure on the space of probability distributions over a discrete surface, designed to mimic that of the Wasserstein distance between distributions over a smooth manifold.  Our construction is inspired by the Benamou--Brenier formula~\shortcite{Benamou2000}, previously discretized only on flat grids without structure preservation. This Riemannian structure automatically defines geodesics and distances between probability distributions. In particular, the geodesic problem can be recast as a convex problem and be tackled by iterative methods %. These iterative methods can be
phrased using local operators familiar in geometry processing and finite elements (gradients, divergence and Laplacian on the surface).  Our method does not require precomputation of pairwise distances between points on the surface. 
%With the help of this new Riemannian structure, we can also study gradient flows in the Wasserstein space as well as the computation of harmonic mappings valued in the Wasserstein space.

Compared to other methods, our interpolation can be rephrased as a geodesic problem and \review{numerically exhibits} less diffusion when interpolating between peaked distributions.  In cases where the sharpness captured by our method and predicted by optimal transport theory is undesirable visually, we provide a quadratic regularizer that \emph{controllably} reduces congestion of the computed interpolant; unlike entropically-regularized transport, however, our optimization problem does not degenerate or become harder to solve when the regularization term vanishes. \review{Although the computation of interpolants remains quite slow for meshes with more than a few thousand vertices and improving the scalability of numerical routines used to optimize our convex objective remains a challenging task for future work, we demonstrate} application to tasks derived from transport, e.g.\ computation of harmonic mappings into Wasserstein space and integration of gradient flows.

In addition to our algorithmic contributions, we regard our work as a key theoretical step toward making optimal transport compatible with the language of discrete differential geometry (DDG). Our Riemannian metric induces a \emph{true} geodesic distance---with a triangle inequality---on the space of distributions over a triangulated surface expressed using one value per \review{vertex. % (interpreted either as a primal 0-form or a dual 2-form).  % <-- removed since reviewers asked to strip out references to DEC...
Inspired} by an analogous construction on graphs~\cite{Maas2011}, we leverage a non-obvious observation that a strong contender for structure-preserving discrete transport on meshes actually involves a real-valued external time variable, rather than discretizing transport as a linear program as in most previous work. The resulting geodesic problem naturally preserves convexity and other key properties from the theoretical case while suggesting an effective computational technique.

% !TeX root = main.tex
\section{Related work}

%\justin{Typically this is written in present rather than past tense:  ``Benamou and Brenier propose a rephrasing...'' (instead of ``proposed''); currently draft isn't consistent.}

\subsection{Linear programming and regularization}

%\justin{Currently this starts ``in medias res'' --- you're discussing different ways to pose ``the optimal transport problem'' without saying what it is --- be sure to define it, at least in words.  For example, you're using phrases like ``both measures'' without saying that transport has anything to do with measure theory.  Basically, an intro paragraph is needed here}

Landmark work  by Kantorovich~\shortcite{Kantorovich1942} showed that optimal transport can be phrased as a linear programming problem. If both probability distributions have finite support, we end up with a finite-dimensional linear program solvable using standard convex programming techniques. A variety of solvers has been designed to tackle this linear program, which exploit the particular structure of the objective functional~\cite{edmonds1972theoretical,klein1967primal,orlin1997polynomial}. These methods, however, usually require as input the \emph{pairwise} distance matrix, a dense matrix that scales quadratically in the size of the support and is difficult to evaluate if the points are on a curved space.

A landmark paper by Cuturi~\shortcite{Cuturi2013} reinvigorated interest in numerical transport by proposing adding an \emph{entropic regularizer} to the problem, leading to the efficient \emph{Sinkhorn} (or \emph{matrix rebalancing}) algorithm. This algorithm, which involves iteratively rescaling the rows and columns of a kernel in the cost matrix, is highly parallelizable and well-suited to GPU architectures.
When the cost matrix involves squared geodesic distances along a discrete surface, Solomon et al.~\shortcite{Solomon2015} showed that Sinkhorn iterations can be written in terms of heat diffusion operators, eliminating the need to store the cost matrix explicitly.
While they are efficient, these entropically-regularized techniques suffer from diffusion, making them less relevant to problems in which measures are sharp or peaked.  They also do not define true distances on the space of distributions over mesh vertices.

When the transport cost is equal to geodesic distance, i.e.\ the $1$-Wasserstein distance, optimal transport is equivalent to the \emph{Beckmann problem} \cite[Chapter 4]{Santambrogio2015}, for which specific and efficient algorithms can be designed~\cite{Solomon2014,Li2018}. These methods cannot be applied to the quadratic Wasserstein distance, which is needed to make transport-based interpolation nontrivial, namely to recover McCann's displacement interpolation \shortcite{Mccann1997}. In particular, the optimal transport problem defining the $1$-Wasserstein distance does not come with a time dependency allowing to define a smooth interpolation and suffers from non-uniqueness coming from the lack of strict convexity.

%\ed{Upon reading this, I found myself asking why one needs McCann's displacement interpolation. Why isn't the 1-Wasserstein interpolation enough? Maybe could be good to justify this in a sentence or two. It's because of non-uniqueness of the optimal transport, right? Though aren't there ways to pull out unique solutions by asking for additional properties or adding additional optimization terms?}

%\hugo{I've added a sentence. Indeed, in $W_1$ you don't have uniqueness, and there is not really a dynamic problem (analogue to Benamou-Brenier) defining $W_1$. In some sense, in the $W_1$ problem is not encoded the speed at which the particles (in the Lagrangian representation) should move, and even if you add an additional term to your optimization I think it's artificial and does not behave well numerically.}

\subsection{Semi-discrete optimal transport}

When one of the \review{distributions} has a density w.r.t.\ Lebesgue while the other one is discrete, the transport problem can be reduced to a finite-dimensional convex problem whose number of unknowns
scales with the cardinality of the support of the discrete \review{distribution}. Leveraging tools from computational geometry, this \emph{semi-discrete} problem can be solved efficiently up to fairly large scale when the cost is Euclidean \review{\cite{Aurenhammer1998,Merigot2011,de2012blue,Levy2015, Kitagawa2016}}.

Semi-discrete transport has been used to tackle problems for which the precise structure of the optimal transportation map is relevant, as in fluid dynamics \cite{DeGoes2015fluid, Merigot2016, Gallouet2017}.  It also has been used for approximating barycenters in the stochastic case \cite{Claici2018} and as a measure of proximity for shape reconstruction \cite{DeGoes2011,Digne2014}. \review{Extensions of semi-discrete transport to curved spaces can be found in \cite{DeGoes2014, Merigot2018}}. Although they can be fast and give explicit transport maps, these methods are not suited for the application we have in mind: They rely on the computation of transport maps between two probability distributions that are not of the same nature (one is discrete, the other has a density) and hence cannot be used to implement a distance or interpolation \review{cleanly}.% in the space of probability distributions.

\subsection{Fluid dynamic formulations}

By switching from Lagrangian to Eulerian descriptions of transport, Benamou and Brenier~\shortcite{Benamou2000} proved that optimal transport could be rephrased using fluid dynamics: Instead of computing a coupling, they show that transport with quadratic costs is equivalent to \review{finding} a \emph{time-varying} sequence of distributions smoothly interpolating between the two measures. The problem that they obtain is convex and solved via the Alternating Direction Method of Multipliers (ADMM)~\cite{Boyd2011}. Proof of the convergence of ADMM in the infinite-dimensional setting (i.e.\ when neither time nor the geometric domain is discretized) is provided in \cite{Guittet2003,Hug2015}. Papadakis et al.~\shortcite{Papadakis2014} reread the ADMM iterations as a proximal splitting scheme and show how one can build different algorithms to solve the convex problem. This fluid dynamic formulation also appears in mean field games \cite{Benamou2015}.

In all of the above work, however, the authors work in a flat space and use finite difference discretizations of the densities and velocity fields. Hence their work does not contain a clear indication about how to handle the problem on a discrete curved space, and theoretical properties of their models \emph{after} discretization remain unverified. 

\review{The algorithm for approximating 1-Wasserstein distances presented by Solomon et al.~\shortcite{Solomon2014} achieves some of the objectives mentioned above.  Their vector field formulation is in some sense dynamical, and their distance satisfies properties like the triangle inequality after discretization.  As mentioned above, however, their optimization problem lacks strict convexity and is not suitable for interpolation.}

\subsection{Dynamical transport on graphs and meshes}

Maas~\shortcite{Maas2011} defines a Wasserstein distance between probability distributions over the vertices of a graph. The (finite-dimensional) space of distributions in this case inherits a Riemannian metric with some structure preserved from the infinite-dimensional definition; for instance, the gradient flow of entropy corresponds to a notion of heat flow along the graph. A similar structure is proposed by Chow et al.~\shortcite{Chow2016}, but they recover a different heat flow. Erbar et al.~\shortcite{Erbar2017} propose a numerical algorithm for approximating the discrete Wasserstein distance introduced by Maas, but the \review{distributions} they produce have a tendency to diffuse along the graph. This flaw is not related to their numerical method but rather comes from the very definition of their optimal transport distance. It is also not obvious \review{what} is the best way to adapt their construction to discrete surfaces rather than graphs.

%\ed{Could we compare to this? I don't know how much work that would be.}

%\hugo{Their implementation is (quoting Sebastian) "Not a walk in the park". You can also have a look at their paper, their computations are mainly on very simple graphs and not actual surfaces, their paper is more a theoretical one (proving that the scheme converges for instance).}

%\justin{I don't think it's worth comparing.  I added a sentence mentioning that it's not clear how to extend to a surface rather than a graph.}

\subsection{Interpolation and geodesics}

Optimal transport is not the only way to interpolate between probability distributions; for instance, Azencot et al.~\shortcite{Azencot2016} use a time-independent velocity field to advect functions and match them. Their method, however, cannot be understood as a geodesic curve in the space of distributions. In another direction, Heeren et al.~\shortcite{Heeren2012} have provided an efficient way to discretize in time geodesics in a high-dimensional space of thin shells. Their formulation is not well-suited for optimal transport where direct discretization of the Benamou--Brenier formula is possible.  Finally, methods like~\cite{panozzo2013weighted} provide a means of averaging points on discrete surfaces, although it is not clear how to extend them to the more general distribution case.

%\ed{Maybe Daniele's stuff could be mentioned? Quite a different problem though, so maybe not...}

%\hugo{Their work is more about averaing points rather than interpolating probability distributons, isn't it?}

%\justin{It'd be nice to compare but isn't 100\% critical}

% \subsection{Discrete surfaces}

% \justin{What's the point of this section?  I don't think it's needed.  While your readers don't know much about OT, they're experts on triangles.}

% In practice, one works with a triangulated surface, i.e. a set of vertices, linked by edges, but also a set of faces, each face containing 3 edges (this set of faces is precisely what makes the difference between the graph strucure and the mesh structure). Rather than to be thought as the discretization of a smooth surface, but as a discrete project which carries an intrisic structure. The notions of gradient, divergence, Laplacian, etc., can be defined (even though there is no right choice but many different ones) such that the equivalent of some cohomology results (like the Hodge decomposition) hold true on the discrete surface \cite{Crane2013,DeGoes2015vector}.

% Our goal is precisely to design a fluid dynamics formulation of optimal transport on a discrete surface which can be expressed with the tools of the geometry of discrete surfaces.

% !TeX root = main.tex

\section{Optimal transport on a discrete surface}
\label{section_discrete_OT}

\begin{figure} 

\begin{tabular}{ccccc}
     \includegraphics[width=.07\textwidth]{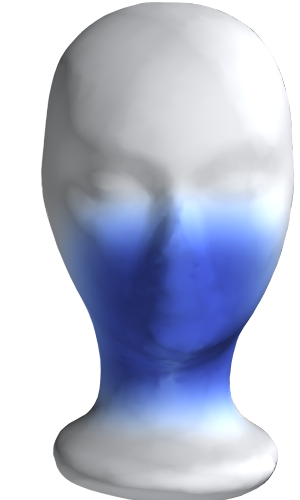} &
     \includegraphics[width=.07\textwidth]{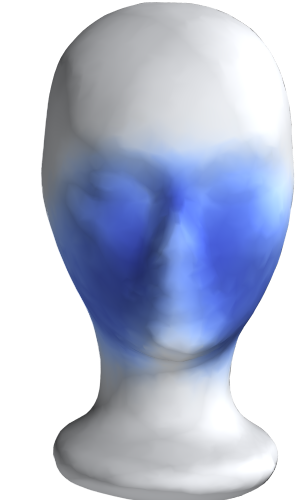} &
     \includegraphics[width=.07\textwidth]{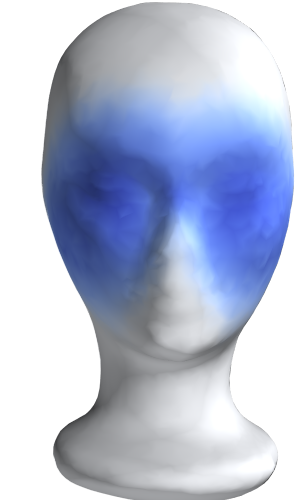} &
     \includegraphics[width=.07\textwidth]{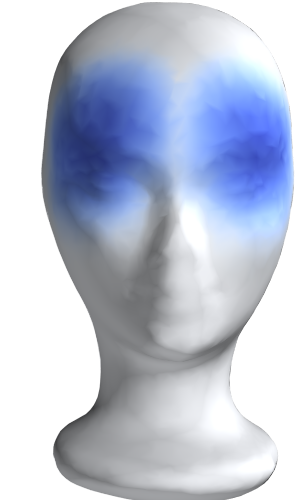} &
     \includegraphics[width=.07\textwidth]{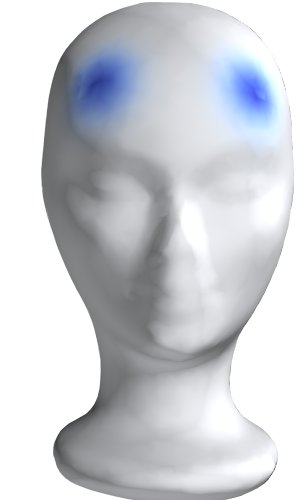} \\
      
     \includegraphics[width=.07\textwidth]{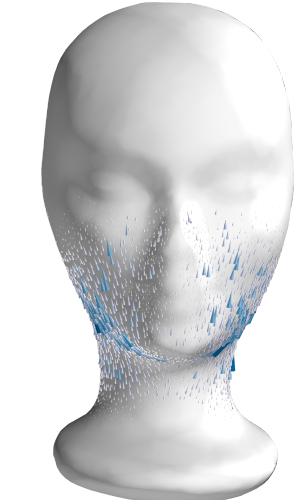} &
     \includegraphics[width=.07\textwidth]{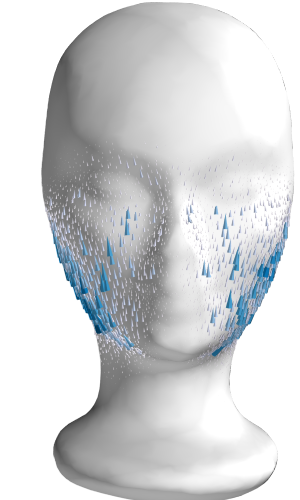} &
     \includegraphics[width=.07\textwidth]{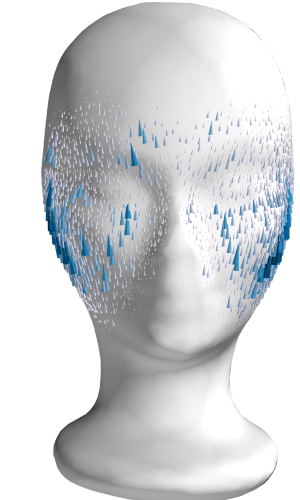} &
     \includegraphics[width=.07\textwidth]{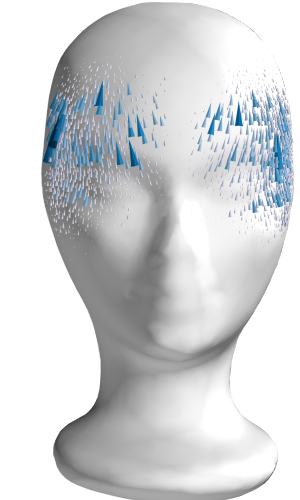} &
     \includegraphics[width=.07\textwidth]{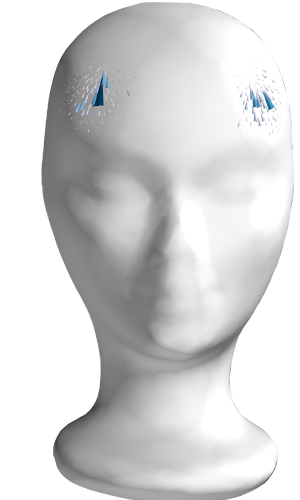} \\
     
     \hline 
      
     $t=0$ & $t=\nicefrac{1}{4}$ & $t = \nicefrac{1}{2}$ & $t=\nicefrac{3}{4}$ & $t=1$ 
      
\end{tabular}

\caption{Top row: Interpolation of probability distributions. The left and right distributions are data and the middle ones are the output of our algorithm. Bottom row: Display of the momentum $\mbf = \mu \vbf$, where $\vbf$ is a time-dependent velocity-field advecting the left distrbution on the right one. We have used the regularization described in Subsection \ref{subsection_regularization} with $\alpha = 0.1$.}
\label{figure_momentum}
\end{figure}

\subsection{Optimal transport on manifolds}

\begin{figure}
\begin{tabular}{cccc}

\multicolumn{4}{c}{
\begin{tikzpicture}[scale = 2]

% mu0
\draw [fill = blue!30, line width = 1pt] (0,0.5) circle (0.3) ;

\draw (-0.5,0.5) node{$\bar{\mu}^0$} ;

% mu1
\draw [fill = blue!10, line width = 1pt] (2,0.2) circle (0.15) ;
\draw [fill = blue!10, line width = 1pt] (2,0.8) circle (0.15) ;

\draw (2.4,0.5) node{$\bar{\mu}^1$} ;

% First point 
% Start
\draw [fill = black] (0.1,0.6) circle (0.02) ;
\draw (0.15,0.65) node[above right]{$x$} ;
% Image 
\draw [fill = black] (2.05,0.85) circle (0.02) ;
\draw (2.15,0.85) node[right]{$y$} ;

% Second point
% Start
\draw [fill = black] (0,0.3) circle (0.02) ;
\draw (-0.05,0.25) node[below left]{$x'$} ;
% Image 
\draw [fill = black] (1.95,0.1) circle (0.02) ;
\draw (1.9,-0.05) node[right]{$y'$} ;
 
% Arrows 
\draw [->, >=latex, dashed] (0.1,0.6) -- (2.05,0.85) ; 
\draw [->, >=latex, dashed] (0,0.3) -- (1.95,0.1) ; 

\draw (1,.9) node{$\dint \pi(x,y)$} ; 

\end{tikzpicture}} \\

\hline

& & & \\

\includegraphics[width=.1\textwidth]{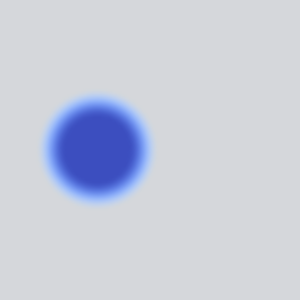} &
\includegraphics[width=.1\textwidth]{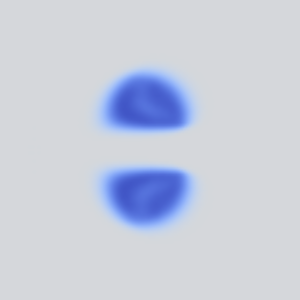} &
\includegraphics[width=.1\textwidth]{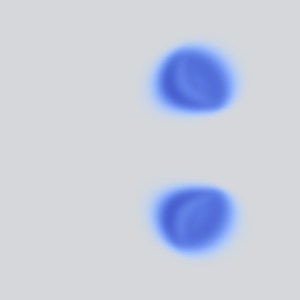} &
\includegraphics[width=.1\textwidth]{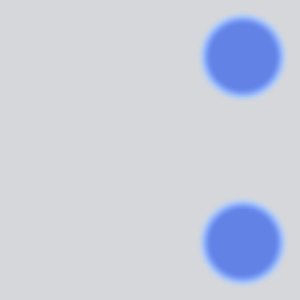} \\

$t= 0$ & $t = \nicefrac{1}{3}$ & $t = \nicefrac{2}{3}$ & $t=1$

\end{tabular}

\caption{\review{Top row: schematic view of the static formulation of optimal transport \eqref{equation_definition_W_coupling}. The initial distribution $\bar{\mu}^0$ is on the left, and the final distribution $\bar{\mu}^1$ is on the right. The quantity $\dint \pi(x,y)$ represents the amount of mass that is transported from $x$ to $y$. The coupling $\pi$ is chosen in such a way that the total cost is minimal. Bottom row: dynamical formulation between the same distributions (computed with the algorithm in Section \ref{section_algorithm}). To go from the top to the bottom row, once one has the optimal $\pi$, a proportion $\dint \pi(x,y)$ of particles follows the geodesic (in this case a straight line) between $x$ and $y$ with constant speed. The macroscopic result of all these motions is a time-varying sequence of distributions, displayed in blue.}}
\label{figure_static_dynamic}
\end{figure}

We begin by introducing briefly optimal transport theory on a smooth space. Let $\M$ be a connected and compact Riemannian manifold with metric $\langle \ , \ \rangle$ and induced norm $\| \ \|$; define $d : \M \times \M \to \R_+$ to be geodesic distance.

Denote by $\P(\M)$ the space of probability measures on $\M$. This space is endowed with the quadratic Wasserstein distance from optimal transport: If $\bar{\mu}^0, \bar{\mu}^1 \in \P(\M)$, then the distance $W_2(\bar{\mu}^0, \bar{\mu}^1)$ between them is defined as
\begin{equation}
\label{equation_definition_W_coupling}
W_2^2(\bar{\mu}^0, \bar{\mu}^1) \eqdef \min_{\pi} \iint_{\M \times \M} \frac{1}{2} d(x,y)^2 \dint \pi(x,y),
\end{equation}
where the minimum is taken over all probability measures $\pi$ on the product space $\M \times \M$ whose first (resp.\ second) marginal is $\bar{\mu}^0$ (resp.\ $\bar{\mu}^1$). 

\review{The problem~\eqref{equation_definition_W_coupling} can be interpreted as follows: $\dint \pi(x,y)$ denotes the quantity of particles located at $x$ that are sent to $y$, and the cost for such a displacement is $d(x,y)^2$. The constraint on the marginals enforces that $\pi$ describes a way of moving the distribution of mass $\bar{\mu}^0$ onto $\bar{\mu}^1$. Thus, the variational problem \eqref{equation_definition_W_coupling} reads: Find the cheapest way $\pi$ to send $\bar{\mu}^0$ onto $\bar{\mu}^1$, and the result (i.e.\ the minimal cost) is defined as the squared Wasserstein distance between $\bar{\mu}^0$ and $\bar{\mu}^1$. In some  generic cases \cite{Brenier1991, Gangbo1996}, the optimal $\pi$ is located on the graph of a map $T : \M \to \M$, which means that a particle $x \in \M$ is sent onto a unique location $y=T(x) \in \M$. 

The space $(\P(\M), W_2)$ is a complete metric space \cite{Santambrogio2015,Villani2003}, and---at least formally---it has the structure of an (infinite-dimensional) Riemannian manifold. Revealing this manifold structure requires some manipulation and rephrasing of the original problem~\eqref{equation_definition_W_coupling}, detailed below.

%More precisely, a
As first noticed by Benamou and Brenier~\shortcite{Benamou2000}, the Wasserstein distance between $\bar{\mu}^0$ and $\bar{\mu}^1$ can be obtained by solving an alternative, physically-motivated optimization problem:}
\begin{equation}
\label{equation_BB_continuous_primal}
W_2^2(\bar{\mu}^0, \bar{\mu}^1) = \begin{cases}
\min_{\mu,\vbf} & \int_0^1 \int_\M \frac{1}{2} \|\vbf^t\|^2 \dint \mu^t \dint t \\
\text{s.t. } &\mu^0 = \bar{\mu}^0, \ \mu^1 = \bar{\mu}^1, \\
&\dr_t \mu + \nabla \cdot( \mu \vbf ) = 0.
\end{cases}
\end{equation}
\review{As we will have to deal with functions and vectors depending both on time and space, here and moving forward we adopt the following convention: Upper indices denote time, and lower indices denote space. Moreover, $t \in [0,1]$ will denote an instant in time, and $f$ will later denote a generic triangle ($f$ for \emph{face}) in a triangulation.} 
In~\eqref{equation_BB_continuous_primal}, the minimum is taken over all curves $\mu : [0,1] \to \P(\M)$ and all time-dependent velocity fields $\vbf : [0,1] \times \M \to T\M$ such that the continuity equation $\dr_t \mu + \nabla \cdot( \mu \vbf ) = 0$ is satisfied in the sense of distributions. The optimal curve $\mu$ is known as McCann's displacement interpolation~\shortcite{Mccann1997}.

%In other words, we transport $\bar{\mu}^0$ onto $\bar{\mu}^1$ by the help of a time-dependent velocity field $\vbf$, and the price paid is the integral over time and space of the kinetic energy $\frac{1}{2} \|\vbf\|^2 \mu$. The optimal curve $\mu$ is known as McCann's displacement interpolation~\shortcite{Mccann1997}.

\review{The physical interpretation of this problem is as follows. Imagine probability distributions as distributions of mass, e.g.\ the density of a fluid. The curve $\mu$ represents an assembly of particles in motion, distributed as $\bar{\mu}^0$ at $t=0$ and $\bar{\mu}^1$ at $t=1$. At time $t$, a particle located at $x \in \M$ moves with velocity $\vbf^t_x$. The continuity equation $\dr_t \mu + \nabla \cdot( \mu \vbf ) = 0$ simply expresses the conservation of mass. For a given time $t$, the cost $\int_\M \frac{1}{2} \|\vbf^t\|^2 \dint \mu^t$ is the total kinetic energy of all the particles. Hence, the cost minimized in \eqref{equation_BB_continuous_primal}, i.e.\ the integral w.r.t.\ time of the kinetic energy, is the \emph{action} of the curve. As there is no congestion cost---that is, the particles do not interact with each other---\eqref{equation_BB_continuous_primal} is the least-action principle for a pressureless gas. 

Formulation \eqref{equation_definition_W_coupling} is static, since it directly determines the target for each particle at $t=1$ given the arrangement at $t=0$.  %which particle is sent where), whereas 
On the other hand, \eqref{equation_BB_continuous_primal} is dynamical, recovering a curve in $\P(\M)$ interpolating smoothly between $\bar{\mu}^0$ and $\bar{\mu}^1$. To convert from the static to the dynamical formulation, one takes an optimal transport plan $\pi$ from \eqref{equation_definition_W_coupling} and an assembly of particles distributed according to $\bar{\mu}^0$. If a particle located at $x \in \M$ at time $t=0$ and is supposed, according to $\pi$, to be sent to $y \in \M$, then this particle follows a constant-speed geodesic along $\M$ from $x$ to $y$. The optimal curve $\mu$ in \eqref{equation_BB_continuous_primal} is exactly the resulting macroscopic motion of all the particles, illustrated in Figure \ref{figure_static_dynamic}.}   

Calling $\mbf = \mu \vbf$ the momentum and using the change of variables $(\mu,\vbf) \leftrightarrow (\mu, \mbf)$, \review{problem} \eqref{equation_BB_continuous_primal} becomes convex, \review{because the mapping $(\mu, \vbf) \to \nicefrac{1}{2} \| \vbf \|^2 \mu$ is not jointly convex while $(\mu, \mbf) \to \nicefrac{1}{2} \| \mbf \|^2/ \mu$ is}. Its dual reads
\begin{equation}
\label{equation_BB_continuous_dual}
W_2^2(\bar{\mu}^0, \bar{\mu}^1) = \begin{cases}
\max_{\varphi} & \int_\M \varphi^1 \dint \bar{\mu}^1 - \int_\M \varphi^0 \dint \bar{\mu}^0\\
\text{s.t. } & \dr_t \varphi + \frac{1}{2} \| \nabla \varphi \|^2 \leqslant 0 \text{ on } [0,1] \times \M,
\end{cases}
\end{equation}
where the maximization is performed over real-valued functions $\varphi : [0,1] \times \M \to \R$ \review{\cite{Villani2008,Santambrogio2015}}. %Moreover, o
%One can show that t
\review{The} relation $\vbf = \nabla \varphi$ holds whenever $\vbf$ (resp.\ $\varphi$) is a minimizer (resp.\ maximizer) of the primal (resp.\ dual) problem. In particular, in \eqref{equation_BB_continuous_primal}, we can restrict ourselves to the set of $\vbf$ such that $\vbf^t = \nabla \varphi^t$ for every $t \in [0,1]$.

Equation~\eqref{equation_BB_continuous_primal}
defines a formal Riemannian structure on $\P(\M)$ \cite{Otto2001}. Given $\mu \in \P(\M)$ with a density bounded from below by a strictly positive constant, the tangent space $T_\mu \P(\M)$ is identified as the set of functions $\delta \mu : \M \to \R$ with $0$-mean: $\delta \mu$ is the partial derivative w.r.t.\ time of a curve whose value at time $0$ is $\mu$. If $\delta \mu \in T_\mu \P(\M)$, we can compute $\varphi : \M \to \R$ the solution (unique up to translation by constants) of the elliptic equation
\begin{equation}
\label{equation_elliptic}
\nabla \cdot (  \mu \nabla \varphi ) = - \delta \mu.
\end{equation}
Then, the norm of $\delta \mu$ is defined as
\begin{equation}
\label{equation_norm_continuous}
\| \delta \mu \|^2_{T_\mu \P(\M)} \eqdef \review{\frac{1}{2}} \int_\M \| \nabla \varphi \|^2 \dint \mu.
\end{equation}
Endowed with this scalar product obtained from the polarization identity $\langle x,y \rangle = \frac{1}{4} (\| x+ y \|^2 - \| x-y \|^2)$, %norm by polarization\justin{not sure what this means} \hugo{it is the identity $\langle x,y \rangle = \frac{1}{4} (\| x+ y \|^2 - \| x-y \|^2)$, how would you call this trick to get the scalar product from the norm?} \ed{I've heard it referred to this way before. Not very natural, but there's a Wikipedia page on "polarization identity" for what it's worth.},
one can check, \review{and the derivation appears in the supplemental material}, that the Wasserstein distance $W_2$ can be interpreted as the geodesic distance induced by \eqref{equation_elliptic} and \eqref{equation_norm_continuous}.  This is precisely the content of the Benamou--Brenier formula \eqref{equation_BB_continuous_primal}.

One needs to assume $\mu \geqslant c > 0$ on $\M$ for the elliptic equation \eqref{equation_elliptic} to be well-posed. Nevertheless, one can still give a meaning to this Riemannian structure using tools from analysis in metric spaces \cite{Ambrosio2008}.

\subsection{Discrete surfaces}

The previous subsection contains only well-understood results.  Let us now start our contribution: to mimic these definitions and properties when the manifold is replaced by a triangulated surface.

Instead of a smooth manifold $\M$, \review{we consider the case where we only have access to a triangulated surface $S = (V,E,T)$}, which consists of a set $V\subset \R^3$ of vertices,
a set $E\subseteq V\times V$ of edges linking vertices, and a set $T\subseteq V\times V\times V$ of triangles containing exactly $3$ vertices linked by $3$ edges. \review{For a given face $f \in T$, we denote by $V_f \subset V$ the set of vertices $v$ such that $v \in f$; for a given vertex $v \in V$, we denote by $T_v \subset T$ the set of faces $f$ such that $v \in f$.} The area of a triangle $f \in T$ is denoted by $|f|$. Each vertex $v$ is associated to a barycentric dual cell (see Figure \ref{figure_meshes}) whose area, equal to $\frac{1}{3} \sum_{f \in T_v} |f|$, is denoted by $|v|$.

Following standard constructions from \review{first-order finite elements} (FEM), % and discrete exterior calculus (DEC),
a scalar function on $\M$ will be seen as having one value per vertex, i.e.\ belonging to $\R^{|V|}$. A \review{distribution} $\mu \in \M$ will be also discretized by one value per vertex representing the density w.r.t.\ the volume measure. In other words, the volume of the dual cell centered \review{at} $v \in V$, measured with $\mu$, is $|v| \mu_v$. We denote by $\P(S)$ the set of probability distributions on the discrete surface:
\begin{equation}
\P(S) \eqdef \left\{ \mu \in \R^{|V|} \text{ s.t. } \mu_v \geqslant 0 \text{ for all } v \in V \text{ and } \sum_{v \in V} |v| \mu_v = 1 \right\}.
\end{equation}
For instance, the volume measure is represented by the vector in $\P(S)$ parallel to \review{$(1,1, \ldots, 1)^\top$.}  %In DEC language, we have chosen to discretize $\P(S)$ using primal 0-forms integrating to 1; applying a Hodge star constructed from barycentric areas $|v|$ would lead to a completely equivalent theory in terms of dual 2-forms summing to 1.

The set $V$ of vertices can be interpreted as a discrete metric space, either by using directly the Eulidean distance on $\R^3$ or by some version of the discrete geodesic distance along $S$. Hence, a natural attempt to discretize the 2-Wasserstein distance would be to use \eqref{equation_definition_W_coupling} and replace $d$ by the distance between vertices. As pointed out in \cite{Maas2011,Gigli2013}, however, this discretization leads to a space without a smooth structure. For instance, there do not exist non-constant smooth (e.g., Lipschitz) curves valued in such a space; \review{whereas in a space with a smooth structure (e.g. a Riemannian manifold), one expects the existence of non-constant Lipschitz curves, namely the (constant-speed) geodesics.}

\review{Let us briefly recall the argument. We take the simplest example of a space consisting of two points. If  $X = \{ x_0, x_1 \}$ contains two points separated by a given distance $\ell$, a probability distribution $\mu$ on $X$ is characterized by a single number %, e.g.\ 
$\mu_{x_0} \in [0,1]$, as $\mu_{x_1} = 1 - \mu_{x_0}$. If $\mu^t$ is a curve valued in $\P(X)$, one can compute $W_2(\mu^t, \mu^s) = \ell \sqrt{|\mu^t_{x_0} - \mu^s_{x_0}|}$. In particular, if $\mu$ is Lipschitz with Lipschitz constant $L$, our expression for $W_2$ implies $|\mu^t_{x_0} - \mu^s_{x_0}| \leqslant \frac{L^2}{\ell^2} |t-s|^2$. There is an exponent $2$ on the r.h.s., but only $1$ on the l.h.s.: it is precisely this discrepancy which is an issue. Indeed, dividing by $|t-s|$ on both side and letting $s \to t$, one sees that $t \mapsto \mu^t_{x_0}$ is differentiable everywhere with derivative $0$, i.e.\ is constant.}

 %\ed{Maybe give more of an example here? I didn't understand this. There's no smooth curves in $\P(S)$ under such a defined metric?} \hugo{This result is the analogue of: there is no smooth curve valued in $\Z$. The space $\P(S)$ with a "true" Wasserstein distance is too "rough".}

For this reason, we prefer to discretize the Benamou--Brenier formulation \eqref{equation_BB_continuous_primal}, as it will automatically give a Riemannian structure on the space $\P(S)$. In this sense, the basic inspiration for our technique is the same as that of Maas~\shortcite{Maas2011}, although on a triangulated surface we enjoy the added structure afforded by an embedded manifold approximation of the domain rather than an abstract graph.

As \eqref{equation_BB_continuous_primal} involves velocity fields, a choice has to be made about their representation \cite{DeGoes2015vector}. To take full advantage of the triangulation, we want to use triangles and not only edges to define our objective functional. The latter choice leads to formulas similar to \cite{Maas2011,Chow2016}, which, as we say above, exhibit strongly diffuse geodesics. %\justin{The previous sentence is false; one value per edge is fine so long as you understand how to interpolate them along triangles (see e.g.\ Whitney basis).  They key issue is to discretize the objective using triangles rather than just edges.}
We prefer to represent vector fields on triangles. More precisely, a (piecewise-constant) velocity field $\vbf$ is represented as an element of $(\R^3)^{|T|}$, i.e.\ as one vector per triangle, with the constraint that $\vbf_f$, which is a vector of $\R^3$, is parallel to the plane spanned by $f$, which means that our velocity fields lie in a subspace of dimension \review{$2 |T|$.}
%\justin{A point of caution here.  The dimension $3|T|$ of this space is incorrect topologically, since you would expect the quotient of exact/co-exact fields to be the Euler characteristic, as discussed in Wardetzky's thesis and Fernando's survey.  In fact, the right dimensionality is one number per edge, just equipped with a ``triangle aware'' inner product.  My bet is that your method implicitly works in this space since you use the FEM gradient operator...otherwise you'd likely be getting noisy vector fields.}
%\hugo{I agree that our method makes eveything live in a space of dimension 2|T| which is the correct dimension for the Hodge decomposition. In the algorithm in itsefl, everything is stored as a vector of $\R^3$ to avoid the use of different bases on different triangles.}
%As with many vector field constructions built using finite elements, we expect that our model could be rephrased using discrete 1-forms in DEC but prefer to manipulate vector fields directly for simplicity. % <--- removed to make reviewers happy

If $\varphi \in \R^{|V|}$ represents a real-valued function, we compute its gradient \review{along the mesh} using the first-order (piecewise-linear) finite element method \cite{brenner2007mathematical}:
For each triangle $f$, we compute $\hat{\varphi}$, the unique affine function defined on $f$ coinciding  with $\varphi$ on the vertices of $f$. Then, the gradient of $\varphi$ in $f$ is simply defined as the gradient of $\hat{\varphi}$ at any point of $f$; as the gradient is constant on each triangle, we need to store only one vector per triangle. %\justin{mention gradient is piecewise constant, which explains why you get dimension $3|T|$}
Since this operator is linear, let us denote by $G \in \R^{ 3|T| \times |V| }$ its matrix representation. %Notice that this way of interpolating $\varphi$ through piecewise affine functions over all triangles is nothing else than the $P_1$ finite element representation of $\varphi$ \cite[Chapter 6]{Crane2013}. %<--- cite something earlier or an FEM book
In particular, the Dirichlet energy of $\varphi \in \R^{|V|}$ is defined as
\begin{equation}
\mathrm{Dir}(\varphi) \eqdef \frac{1}{2} \sum_{f \in T} |f| \| (G \varphi)_f \|^2.
\end{equation}
The sum is weighted by the areas of the triangles to discretize a surface integral. The first variation of this Dirichlet energy can be expressed in matrix form as $(G^\top M_T G) \varphi$, where $M_T \in \R^{3|T| \times 3|T|}$ is a diagonal weight matrix whose elements are the areas of the triangles. The matrix $G^\top M_T G$ is the so-called cotangent Laplace matrix of a triangulated surface \cite{Pinkall1993}.

\subsection{Dual problem on meshes}

Let us introduce our discrete Benamou--Brenier formula by starting from its dual formulation \eqref{equation_BB_continuous_dual}. Since the objective functional is linear, its discrete counterpart is straightforward as both $\mu$ and $\varphi$ are defined on vertices. On the other hand, in the constraint $\dr_t \varphi + \frac{1}{2} \| \nabla \varphi \|^2 \leqslant 0$, we would like to replace $\nabla \varphi$ by $G \varphi$ but then the two terms of the sum do not live on the same space.

The constraint $\dr_t \varphi + \frac{1}{2} \| \nabla \varphi \|^2 \leqslant 0$ is \review{a priori} not coercive.
%
%
% Justin's version
Suppose $\varphi$ satisfies the constraint, and take another function $\psi$ with the property that $\varphi+s\psi$ satisfies the constraint for arbitrarily large $s \geqslant 0$.  Expanding the inequality $\dr_t (\varphi + s \psi) + \frac{1}{2} \| \nabla \varphi + s \nabla\psi \|^2 \leqslant 0$ and taking the limit $s \to + \infty$  shows that $\psi$ satisfies this property if and only if $\|\nabla \psi\|=0$ and $\partial_t \psi\leqslant0$; these two conditions together imply that the objective functional in~\eqref{equation_BB_continuous_dual} is smaller when evaluated at $\varphi+s\psi$ rather than at $\varphi$.
%
%
%Though, if $\varphi$ satisfies it, and if $\psi$ is given, the function $\varphi + s \psi$ satisfies the constraint for $|s|$ arbitrary large if and only if $\| \nabla \psi \| \leqslant 0$ and $ s \dr_t \psi \leqslant 0$ which implies that the objective functional \eqref{equation_BB_continuous_dual} is smaller when evaluated at $\varphi + s \psi$ rather than $\varphi$.
%\hugo{Ok, now I've written exactly what I means, the only thing left to the reader is to expand $\dr_t (\varphi + s \psi) + \frac{1}{2} \| \nabla \varphi + s \psi \|^2 \leqslant 0$ and take the limit $s \to + \infty$}.
%
% the functions that lie in the kernel of the quadratic part of the constraint are constant in space (hence decreasing in time, since the constraint implies $\dr_t \varphi \leqslant 0$); in other words, $\varphi^0$ and $\varphi^1$ are constant on $\M$ \justin{why?} and $\varphi^1 \leqslant \varphi^0$. Plugging this expression in the objective functional of \eqref{equation_BB_continuous_dual}, one gets a negative value. \ed{Not sure I understood this part fully. Why would a negative value for the objective functional be implied?} \hugo{I have detailed, is it understandable now?} \justin{Afraid I still don't follow!  Why should it be constant?  And aren't distances positive rather than negative?}
%
This is a property that we would like to keep at the discrete level. To do so, we enforce a discrete analogue of the constraint at each vertex of the mesh. To go from $\| G \varphi \|^2$, which is defined on triangles, to something defined on vertices, we \emph{first} take the squared norm and \emph{subsequently} average in space:\footnote{If we do the opposite (averaging and then taking the square), there are spurious modes in the kernel of the quadratic part of the constraint, which leads to poor results when working with non-smooth densities.}
%\sebastian{I'm afraid this is still a bit opaque to me. I don't fully follow what this discussion has to do with the following definition.}

\begin{definition}
Let $\bar{\mu}_0, \bar{\mu}_1 \in \P(S)$. The discrete (quadratic) Wasserstein distance $W_d(\bar{\mu}_0, \bar{\mu}_1)$ is defined as the solution of the following convex problem:
\begin{equation}
\label{equation_BB_discrete_dual}
W_d^2(\bar{\mu}_0, \bar{\mu}_1) = \begin{cases}
\sup_{\varphi} &  \sum_{v \in V} |v| \varphi^1_v \bar{\mu}^1_v - \sum_{v \in V} |v| \varphi^0_v \bar{\mu}^0_v  \\
\text{s.t. } & \dr_t \varphi^t_v + \displaystyle{ \frac{1}{2} \frac{\sum_{f \in T_v} |f| \| (G \varphi)^t_f \|^2 }{  3 |v|}}   \leqslant 0 \\ & \text{ for all } (t,v) \in [0,1] \times V,
\end{cases}
\end{equation}
where the unknown is a function $\varphi : [0,1] \times V \to \R$.
\end{definition}

\noindent The denominator $3|v|$ is nothing else, by definition, than $\sum_{f \in T_v} |f|$. In particular, the value $\left( \sum_{f \in T_v } |f| \| (G \varphi)^t_f \|^2 \right)(3 |v|)^{-1}$ is the average, weighted by the areas of the triangles, of $\| (G \varphi)^t_f \|^2$ for \review{$f \in T_v$}. One can check that the same reasoning as above can be performed. \review{Indeed, if $\varphi  : [0,1] \times V \to \R$ satisfies the constraint in \eqref{equation_BB_discrete_dual} and $\varphi + s \psi$ also satisfies it for arbitrarily large $s \geqslant 0$, it implies, taking $s \to + \infty$, that
\begin{equation}
\displaystyle{ \frac{1}{2} \frac{\sum_{f \ \text{s.t.} \  v \in f  } |f| \| (G \psi)^t_f \|^2 }{  3 |v|}} \leqslant 0.
\end{equation}
This inequality must hold for all $(t,v) \in [0,1] \times V$.  Thus, we conclude (and it is for this implication that it is important to average after taking squares) that $G \psi$ is identically $0$. In other words, for all $t \in [0,1]$, the function $\psi^t$ is constant over the discrete surface. Plugging this information back into the constraint in \eqref{equation_BB_discrete_dual} and taking again $s \to + \infty$, we see that $\dr_t \psi \leqslant 0$. Hence, the value $\psi^0$ (which is constant over the surface) is larger than $\psi^1$. With this information ($G \psi = 0$ and $\dr_t \psi \leqslant 0$), %we see easily that 
the value of the objective functional must be smaller for $\varphi + s \psi$ than for $\varphi$ as soon as $s \geqslant 0$. }

 %$G \psi = 0$ and $\dr_t \psi = 0$, which ends up in a decrease of the objective functional. 

\subsection{Riemannian structure of the space of probabilities on a discrete surface}
\label{subsection_Riemannian_discrete}

To recover an equation which looks like the primal formulation of the Benamou--Brenier formula \eqref{equation_BB_continuous_primal}, it is enough to write the dual of the discrete formulation \eqref{equation_BB_discrete_dual}. The latter formulation, as explained above, was important to justify the \emph{choice} of the way we average quantities that do not live on the same grid.

We introduce additional notation to deal with the averaging of the density $\mu$. If $\mu \in \P(S)$, we denote by $\hat{\mu} \in \R^{|T|}$ the vector given by, for any $f \in T$,
\begin{equation}
\label{equation_definition_averaging_mu}
\hat{\mu}_f = \frac{1}{3} \sum_{v \in V_f} \mu_v.
\end{equation}
This is a natural way to average $\mu$ from vertices to
triangles, \review{appearing in the dual formulation given below:} 

\begin{proposition}
\label{proposition_BB_discrete_primal}
The following identity holds:
\begin{equation}
\label{equation_BB_discrete_primal}
W_d^2(\bar{\mu}^0, \bar{\mu}^1) = \begin{cases}
\min_{\mu,\vbf} & \int_0^1 \left( \sum_{f \in T} \frac{1}{2} \| \vbf^t_f \|^2 |f| \hat{\mu}^t_f \right) \dint t \\
\text{s.t. } &\mu^0 = \bar{\mu}^0, \ \mu^1 = \bar{\mu}^1 \\
&\dr_t (M_V \mu^s_v) +  (-G^\top M_T [\hat{\mu}^t \vbf^t])_v = 0 \\
& \text{for all } (t,v) \in [0,1] \times V.
\end{cases}
\end{equation}
\end{proposition}
\noindent %Notice that
Recall that $M_T \in \R^{3|T| \times 3|T|}$ and $M_V \in \R^{|V| \times |V|}$ are diagonal matrices corresponding to multiplication by the area of the triangles and of the dual cells respectively. %\justin{Use a letter other than $D$, which usually indicates ``derivative''}
Then, $-G^\top M_T$ represents a discrete version of the (integrated) divergence operator, suggesting that the constraint can be interpreted as a discrete continuity equation. The derivation of this result, \review{detailed in the supplemental material}, relies on an $\inf$-$\sup$ exchange, similar to the case of a smooth surface $\M$. %) and is not detailed in this article.%<--- no need to say this

Proposition \ref{proposition_BB_discrete_primal}, very similar to \eqref{equation_BB_continuous_primal}, shows that $W_d$ is the geodesic distance for a Riemannian structure on the space $\P(S)$, at least for non-vanishing densities. Let us detail the metric tensor for a density $\mu \in \P(S)$ with $\min_v \mu_v > 0$. As the set $\P(S)$ is a codimension-1 subset of the linear space $\R^{|V|}$, the tangent space at $\mu$ is naturally $\{ x \in \R^{|V|} \text{ s.t. } \sum_{v \in V} |v| x_v = 0 \}$. In analogy to~\eqref{equation_elliptic}, take $\delta \mu \in T_\mu \P(S)$. We call $\varphi \in \R^{|V|}$ a solution of
\begin{equation}
\label{equation_elliptic_discrete}
M_V  \delta \mu = - (G^\top M_T M_{\hat{\mu}} G) \varphi,
\end{equation}
where $M_{\hat{\mu}} \in \R^{3|T| \times 3 |T|}$ is a diagonal matrix corresponding to multiplication on each triangle by $\hat{\mu}$. As $\hat{\mu} > 0$ everywhere on $V$, this equation is well-posed: The kernel of $G^\top M_T M_{\hat{\mu}} G$ is of dimension one (it consists only of the constant functions), and $M_V  \delta \mu$ lies in the image of this operator. When the distribution $\mu$ is uniform, \eqref{equation_elliptic_discrete} boils down to a Poisson equation, as the operator $- (G^\top M_T M_{\hat{\mu}} G)$ is proportional to the cotangent Laplacian.

One can then define the norm of $\delta \mu$ on the tangent space $T_\mu \P(S)$ as
\begin{equation}
\| \delta \mu \|^2_{T_\mu \P(S)} \eqdef \frac{1}{2} \sum_{f \in T} \| (G \varphi)_f \|^2 |f| \hat{\mu}_f.
\end{equation}
The function $\varphi$ is unique up to an additive constant, which lies in the kernel of the matrix $G$, so this norm is well-defined.

To put everything in one formula, the scalar product $\langle \delta \mu, \delta \nu \rangle_{T_\mu \P(S)}$ between two elements of the tangent space at $\mu$ can be expressed as $(\delta \nu)^\top P_\mu (\delta \mu)$, where the matrix $P_\mu$ is expressed as
\begin{equation}
\label{equation_metric_tensor_discrete}
\review{P_\mu = \frac{1}{2} M_V^\top G^{- \top} (M_{\hat{\mu}} M_T)^{-1} G^{-1} M_V.}
\end{equation}
One can check, \review{and the derivation is provided in the supplemental material}, that $W_d$ is exactly the geodesic distance induced by this metric tensor.

\begin{proposition}
The function $W_d : \P(S) \times \P(S)$ is a distance.
\end{proposition}
\begin{proof}
It is a general fact that the geodesic distance on a manifold (defined by minimization over all possible trajectories) is a distance, see for instance \cite[Section 1.4]{Jost2008}.
\end{proof}

A natural question is whether the space $(\P(S), W_d)$ converges to $(\P(\M), W_2)$ as $S$ becomes a finer and finer discretization of a manifold $\M$. For a discrete Wasserstein distance like the one of Maas \shortcite{Maas2011}, based on the graph structure of $S$---which corresponds to the case where velocity fields are discretized by their values on edges and a particular choice of scalar product---the answer is known to be positive in the case where $\M$ is the flat torus \cite{Gigli2013,Trillos2017} in the sense of %, and the convergence is a
Gromov--Hausdorff convergence of metric spaces, \review{while a very recent work by Gladbach, Kopfer and Mass~\shortcite{Gladbach2018} has refined the analysis and exhibits necessary conditions for such a convergence to hold}. The high technicality of the proofs of these results, however, indicates that the question for our particular definition is likely to be challenging and out of the scope of the present article.

% !TeX root = main.tex

\section{Time discretization of the geodesic problem}
\label{section_algorithm}

% Figure for the meshes:

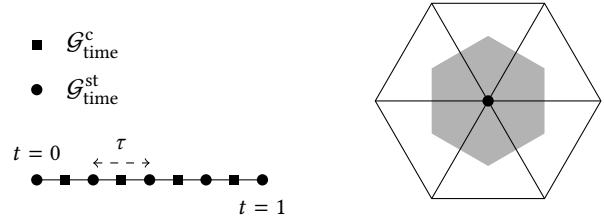
\begin{figure}
\begin{center}
\begin{tikzpicture}[scale = 1.5]

% Temporal grid ---------

\draw (-4,-0.7) -- (-2,-0.7) ;

% Staggerred grid
\foreach \k in {0,1,...,4}
	\fill (-4+0.5*\k,-0.7) circle (0.05) ;
% Centered grid
\foreach \k in {0,1,...,3}
	\fill (-3.75+0.5*\k+0.04,0.04-0.7) -- (-3.75+0.5*\k+0.04,-0.04-0.7) -- (-3.75+0.5*\k-0.04,-0.04-0.7) -- (-3.75+0.5*\k-0.04,0.04-0.7) -- cycle ;

% Temporal bounds
\draw (-4,-0.6) node[above]{$t=0$} ;
\draw (-2,-0.8) node[below]{$t=1$} ;

% Length of the subintervals
\draw[style = <->, dashed] (-3.5,-0.55) -- (-3,-0.55) ;
\draw (-3.25,-0.5) node[above]{$\tau$} ;

% Name of the grid -------

\draw (-3.5,0.5) node{$\Gtimec$};
\fill (-4+0.04,0.5+0.04) --(-4-0.04,0.5+0.04) -- (-4-0.04,0.5-0.04)-- (-4+0.04,0.5-0.04) -- cycle ;

\draw (-3.5,0.1) node{$\Gtimes$};
\fill (-4,0.1) circle (0.05) ;

% Mesh ------------------

% boundary of the dual cell
\fill [color = gray!60] (0,0.577) -- (0.5,0.289) -- (0.5,-0.289) -- (0,-0.577) -- (-0.5,-0.289) -- (-0.5,0.289) -- cycle ;

% Segments center -> boundary
\draw (0,0) -- (1,0) ;
\draw (0,0) -- (0.5,0.866) ;
\draw (0,0) -- (-0.5,0.866) ;
\draw (0,0) -- (-1,0) ;
\draw (0,0) -- (-0.5,-0.866) ;
\draw (0,0) -- (0.5,-0.866) ;

% boundary of the cell
\draw (1,0) -- (0.5,0.866) -- (-0.5,0.866) -- (-1,0) -- (-0.5,-0.866) -- (0.5,-0.866) -- cycle ;

% central node
\fill (0,0) circle (0.05);

\end{tikzpicture}
\end{center}
\caption{Left: temporal grids $\Gtimec$ and $\Gtimes$ for $N = 4$. Right: a vertex ($\bullet$) surrounded by $6$ adjacent triangles, the dual barycentric cell is in gray.}
\label{figure_meshes}
\end{figure}

\subsection{Discrete geodesic}

So far, we have defined a structure-preserving notion of optimal transport on a triangle mesh.  While our model has many properties in common with the continuum version of transport, the resulting optimization problem is infinite-dimensional since the unknown $\mu^t$ is indexed by a time variable $t\in[0,1]$.  Our next step is to derive a time discretization that approximates this interpolant in practice.
Put simply, we want to solve the geodesic problem, i.e., given $\bar{\mu}^0, \bar{\mu}^1 \in \P(S)$, we want to approximate the solution $\mu$ of \eqref{equation_BB_discrete_primal}. To this end, we discretize in time the dual problem \eqref{equation_BB_discrete_dual}.

%Indeed, this problem has a nice structure, namely it is a SOCP (Second Order Cone Programming) problem.
Our infinite-dimensional problem can be classified as a second-order cone program (SOCP) \cite[Section 4.4.2]{Boyd2004};  we choose a discretization that preserves this structure. The main issue is that with a standard finite difference scheme, the derivative $\dr_t \varphi$ ends up on a grid staggered w.r.t.\ the one on which $\varphi$ is defined. Hence, we average to define the constraint on a compatible grid.
We apply the same idea as before: With the term involving $\| G \varphi \|^2$, we average \emph{after} taking the square  to avoid the introduction of any spurious null space.

Let $N$ be the number of discretization points in time. We consider two grids: the \emph{staggered grid} $\Gtimes \eqdef \{ k/N \ : \ k = 0,1, \ldots,N \}$ and the \emph{centered grid} $\Gtimec \eqdef \{ (k+1/2)/N \ : \ k = 0,1, \ldots N-1 \}$, see Figure \ref{figure_meshes}  The staggered grid has $N+1$ elements whereas the centered one has only $N$. We call $\tau \eqdef 1/N$ the time step. The linear operator $D : \R^{\Gtimes} \to \R^{\Gtimec}$ defined by
\begin{equation}
(D \varphi)^t \eqdef \frac{\varphi^{t + \tau / 2} - \varphi^{t - \tau / 2}}{\tau},
\end{equation}
is a  natural discretization of the time derivative.

Next, we discretize $\varphi \in \R^{ \Gtimes \times |V|}$ a function depending both on space and time. The constraint $\dr_t \varphi^t_v +\frac{1}{2} \frac{\sum_{f \in T_v} |f| \| (G \varphi)^t_f \|^2 }{  3 |v|}   \leqslant 0$ will be imposed on the centered grid $\Gtimec$. It is enough to replace $\dr_t \varphi$ by $D \varphi$. On the other hand, the term $\frac{1}{2} \frac{\sum_{f \in T_v} |f| \| (G \varphi)^t_f \|^2 }{  3 |v|}$, which is defined on $\Gtimes$, will be also averaged in time. In other words, the fully discrete problem reads:

\begin{equation}\label{equation_fully_discrete_problem}
\boxed{
\begin{array}{l}
\text{Find } \varphi \in \R^{ \Gtimes \times |V|} \text{ maximizing }  \\
\begin{cases}
&\sum_{v \in V} |v| \varphi^1_v \bar{\mu}^1_v - \sum_{v \in V} |v| \varphi^0_v \bar{\mu}^0_v  \\
&\text{s.t. }  (D \varphi)^t_v + \displaystyle{ \frac{1}{2} \sum_{i\!\in\!\{\!-1,\!1\!\}} \frac{1}{2} \frac{\sum_{f \in T_v} |f| \| (G \varphi)^{t + i \tau /2}_\review{f} \|^2 }{  3 |v|}}   \leqslant 0 \\ & \text{ for all } (t,v) \in \Gtimec \times V,
\end{cases}
\end{array}
}
\end{equation}
%\justin{Should there be a 1/2 inside \emph{and} outside the sum in the constraint above?} \hugo{Yes! The first one is because we are looking at $\frac{1}{2} \| \nabla \varphi \|^2$, and the second one is because we average in time. In other words, the second $2$ is the number of neighbors of a point in the centered grid.}
The constraint still stays quadratic, and hence the fully-discrete problem is still a SOCP.

\subsection{Algorithm}
\label{subsection_algorithm}

To tackle \eqref{equation_fully_discrete_problem} algorithmically, we follow Benamou and Brenier~\shortcite{Benamou2000} by building an augmented Lagrangian and using the Alternating Direction Method of Multipliers (ADMM). The main issue is that the constraint is nonlocal---since it involves discrete derivatives---and nonlinear.  We construct a splitting of the problem that decouples these two effects.

To this end, we introduce two additional variables $A$ and $\Bbf$. We enforce the constraint $A = D \varphi$, and hence $A$ is defined on the grid $\Gtimec \times V$. On the other hand, the variable $\Bbf$ stores the values of $G \varphi$ but with some redundancy. Each $(G\varphi)_f^{t}$ appears in more than one inequality constraint in \eqref{equation_fully_discrete_problem}, and $\Bbf$ is chosen so that each component of $\Bbf$ appears in only one inequality constraint. In detail, $\Bbf$ is defined on the grid $\Gtimec \times \{\pm 1\} \times T \times V$ with the constraint that $(f,v) \in T \times V$ is such that $v \in f$. We will impose the constraint that $\Bbf^{t,i}_{f,v} = (G\varphi)_f^{t + i \tau / 2}$ for all $(t,i,f,v) \in \Gtimec \times \{\pm 1\} \times T \times V$.

We introduce the notation $q = (A,\Bbf)$ and write $q = \Lambda \varphi$ if $A,\Bbf$ satisfy the relations written above. %, roughly $\Lambda = (D, G)$, but remember that the values of $\Bbf$ store redundant values of $\varphi$. % <--- this was kind of confusing
Define
\begin{equation}
\label{equation_definition_F}
F(\varphi) = \sum_{v \in V} |v| \varphi^1_v \bar{\mu}^1_v - \sum_{v \in V} |v| \varphi^0_v \bar{\mu}^0_v,
\end{equation}
and $C$ to be the function such that $C(A,\Bbf) = C(q) = 0$ if
\begin{align}
\forall (t,v) &\in \Gtimec \times V,\nonumber\\&\ A^t_v + \displaystyle{ \frac{1}{2} \sum_{i \in \{-1,1\}} \frac{1}{2} \frac{\sum_{f \in T_v } |f| \| \Bbf^{t,i}_{v,f} \|^2 }{  3 |v|}}   \leqslant 0 \label{equation_constraint_AB_definition}
\end{align}
and $- \infty$ otherwise.
The discrete problem \eqref{equation_fully_discrete_problem} can be written
\begin{equation}
\max_{q = \Lambda \varphi} F(\varphi) + C(q).
\end{equation}
The idea is to introduce a Lagrange multiplier $\sigma = (\mu, \mbf)$ associated to the constraint $q = \Lambda \varphi$ and to build the augmented Lagrangian
\begin{equation}
\label{equation_discrete_lagrangian}
L(\varphi,q,\sigma) = F(\varphi) + C(q) + \langle \sigma, q - \Lambda \varphi \rangle - \frac{r}{2} \| q - \Lambda \varphi \|^2.
\end{equation}
In this equation, $\langle \sigma, q - \Lambda \varphi \rangle = \langle \mu, A - D \varphi \rangle_V + \langle \mbf, \review{\Bbf} - G \varphi \rangle_T$, where the scalar product $\langle \ , \ \rangle_V$ (resp.\ $\langle \ , \ \rangle_T$) is weighted by the \review{areas} of the vertices (resp.\ the triangles) and the time step $\tau$.

The saddle points of the Lagrangian \eqref{equation_discrete_lagrangian}---which do not depend on the parameter $r$---are precisely the solutions to the problem \eqref{equation_fully_discrete_problem}, and $\mu$, the first component of $\sigma$ associated to the constraint $A = D \varphi$, is an approximation of the time-continuous geodesic \eqref{equation_BB_discrete_primal}. On the other hand, the second component $\mbf$ is an approximation of the momentum $\mu \vbf$.  

To compute a saddle point, we use ADMM, which consists in iterations of the following form \cite{Boyd2011}:
\begin{enumerate}
\item Given $q$ and $\sigma$, find $\varphi$ that maximizes $L$.
\item Given $\varphi$ and $\sigma$, find $q$ that maximizes $L$.
\item Do a gradient descent step (with step $r$) to update $\sigma$.
\end{enumerate}
The parameter \review{$r>0$} is arbitrary and  tuned to speed up the convergence; see \cite{Boyd2011} for discussion. In our case, \review{details of} the iterations are briefly presented below and summarized in Algorithm \ref{algorithm_geodesics}.

\begin{algorithm}
\caption{{\sc Geodesic Computation}}
\label{algorithm_geodesics}
\begin{algorithmic}

\Function{Geodesic}{$\bar{\mu}^0, \bar{\mu}^1$}

\State {Initialize $\varphi,A,\Bbf,\mu, \mbf \leftarrow 0$}
\While{PrimalResidual \textbf{and} DualResidual $> \varepsilon$}

\State $\varphi \leftarrow$ solution of \eqref{equation_discrete_Poisson}

\For{$s,v \in \Gtimec \times V$}
\State update $A$ and $\Bbf$ by solving \eqref{equation_minimization_AB}
\EndFor

\State Update $\mu$ and $\mbf$ through \eqref{equation_update_dual}

\EndWhile

\State\Return $\mu$

\EndFunction
\end{algorithmic}
\end{algorithm}

\paragraph{Maximization w.r.t.\ $\varphi$}

The Lagrangian $L$ is simply a quadratic function of $\varphi$, so its \review{maximization} amounts to inverting the matrix $\Lambda^\top \Lambda$ which, in our case, behaves like a space-time Laplacian.

More precisely, writing $\varphi \in \R^{\Gtimes \times V}$ as a $(N+1) \times |V|$ matrix (with rows indexed by time and columns by space), the equation satisfied by a \review{maximizer} of $L$ \review{over $\varphi$} reads
\begin{multline}
\label{equation_discrete_Poisson}
r \left[ D^\top M_V D \varphi  + 3  (E^\top E) \varphi ( G^\top M_T G )   \right]  \\ = N(\bar{\mu}^1 I_{t = 1} -  \bar{\mu}^0 I_{t = 0}) - D^\top M_V( \mu - r A ) - (\mbf - r \Bbf) M_T \tilde{G}^\top  .
\end{multline}
Again recall that the unknown here is $\varphi$; the remaining symbols are fixed matrices.
In this equation, $E \in \R^{\Gtimec \times \Gtimes}$ stands for the averaging in time defined by $(E \varphi)^t = \frac{\varphi^ {t - \tau / 2} + \varphi^{t + \tau / 2}}{2}$. The matrices $I_{t= 0}$ and $I_{t=1} \in \R^{\Gtimes \times V}$ stand for the indicator of $t=0$ (resp.\ $t=1$), namely they contain zeros except on the first (resp.\ last) row which is full of ones. The factor $3$ comes from the fact that each value of $(G \varphi)_f$ is duplicated $3$ times in $\Bbf$, one for each vertex which belongs to $f$. The operator $\tilde{G}$ is almost the same as $G$ but takes in account the fact that the values of $G \varphi$ are duplicated in $\Bbf$ (hence in $\mbf$): $\tilde{G}$ corresponds to the adjoint of the second component of the operator $\Lambda$.

$(D^\top M_V D )$ is the discrete Laplacian in time, and $ G^\top M_T G$ is the discrete Laplacian on $S$.  In fact, \eqref{equation_discrete_Poisson} is a Poisson equation with a space-time Laplacian. Equation \eqref{equation_discrete_Poisson} admits more than one solution but they only differ by a constant \review{whose value} does not modify the value of $L$.

%\justin{\eqref{equation_discrete_Poisson} is messy.  Can you write it more elegantly as a matrix equation (e.g.\ $A\phi+\phi B=C$)?  Computer graphics people know that this is linear in $\phi$ even if you don't write it out, and plus my bet is faster solvers exist for this case.}

The linear operator to invert is the same at each iteration, \review{and} hence standard precomputation techniques can be used to speed up \review{the application of its inverse}.%its resolution

%and hence a preconditioner for the matrix which appears in the left hand side of \eqref{equation_discrete_Poisson} can be computed before the iterative process. \justin{Somewhere you need to explicitly outline how you solve for $\varphi$.}

\paragraph{Maximization w.r.t.\ $A,\Bbf$}

The Lagragian $L$ is also quadratic w.r.t.\ $q$, but there is a quadratic constraint on these two variables due to the presence of $C(q)$. Because of the redundancy in $\Bbf$, each component of $A$ or $\Bbf$ is subject to only one constraint. More precisely, we can check that one needs, for each $(t,v) \in \Gtimec \times V$, to minimize
\begin{multline}
\label{equation_minimization_AB}
|v| \left( A^t_v - (\review{D} \varphi)^t_v - \frac{1}{r} \mu^t_v \right)^2 \\ +  \frac{|f|}{2} \sum_{i \in \{ \pm 1 \}} \sum_{v \in V_f} \left\| \Bbf^{t,i}_{f,v} - (G \varphi)^{t + i \tau / 2}_f - \frac{1}{r} \mbf^{t,i}_{f,v}  \right\|^2
\end{multline}
under the constraint \eqref{equation_constraint_AB_definition}. \review{This minimization} amounts to a Euclidean projection on the set of $A,\Bbf$ satisfying \eqref{equation_constraint_AB_definition}, which can be carried out by solving a cubic equation in one variable, independently on each point of $\Gtimec \times V$. These equations are solved using Newton's method.

\paragraph{Dual update}

This gradient descent corresponds to the following operations:
\begin{equation}
\label{equation_update_dual}
\begin{array}{rl}
\mu^t_v & \leftarrow \mu^t_v - r  \left( A_v^t - (\review{D} \varphi)^t_v \right) \\
\mbf^{t,i}_{f,v} & \leftarrow \mbf^{t,i}_{f,v} - r \left(  \Bbf^{t,i}_{f,v} - (G\varphi)_f^{t + i \tau / 2} \right),
\end{array}
\end{equation}
for any $(t,v) \in \Gtimec \times V$ and any $(t,i,f,v) \in \Gtimec \times \{ \pm 1 \} \times T \times V$.

\section{Experiments}

Recall that our main practical contribution is to be able interpolate between probability distributions using an optimal transport model that preserves structure from the non-discretized case. We will illustrate the robustness of our method: It can handle peaked distributions, and \review{it lifts} the intrinsic geometry of the discrete surface while being insensitive to the \review{choice of mesh topology}.

%\justin{Be more exciting here.  Remind them of your contribution and why robustness tests are needed/interesting/useful.}

The typical computation is the following: We enter the data $\bar{\mu}^0, \bar{\mu}^1$ and compute a solution of the discrete problem~\eqref{equation_fully_discrete_problem}. Then, we plot the evolution over time of $\mu$, which approximates the geodesic in the Riemannian metric described in Subsection \ref{subsection_Riemannian_discrete}. As a byproduct of the optimization process, we also obtain the optimal momentum $\mbf = \mu \vbf$, which can be also plotted, see Figure \ref{figure_momentum}. \review{The code used to conduct all our experiments is available at \url{https://github.com/HugoLav/DynamicalOTSurfaces}.}

As the color map is sometimes normalized independently for different time instants on the same interpolation curve, let us underscore this fact: For every example, we have checked numerically that the densities are always nonnegative and that mass is always preserved over time.

\subsection{Convergence of the ADMM iterations}

\begin{figure}
\begin{tikzpicture}
\node[inner sep=0pt] at (2.2,4.5)
    {\includegraphics[width=.11\textwidth]{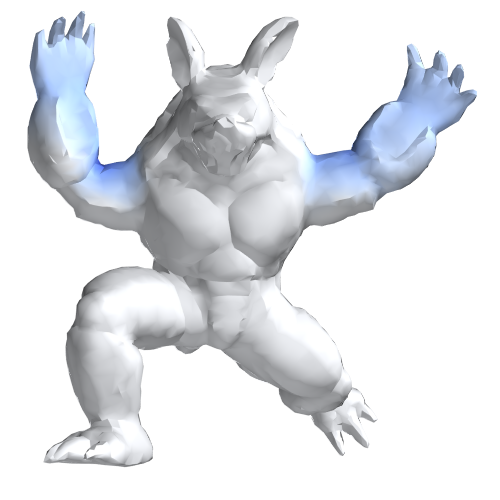}};

\node[inner sep=0pt, opacity = 1] at (3.9,3)
    {\includegraphics[width=.11\textwidth]{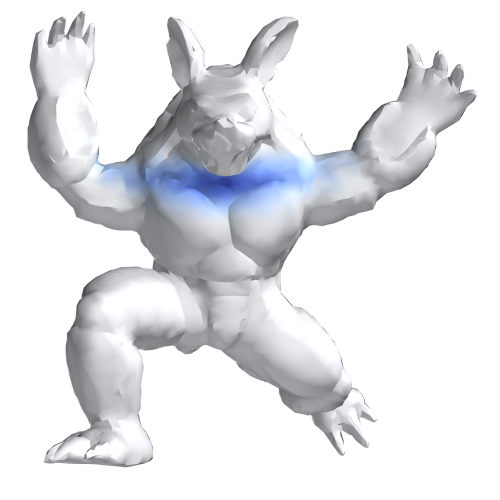}};

\node[inner sep=0pt, opacity = 1] at (5.9,2.5)
    {\includegraphics[width=.11\textwidth]{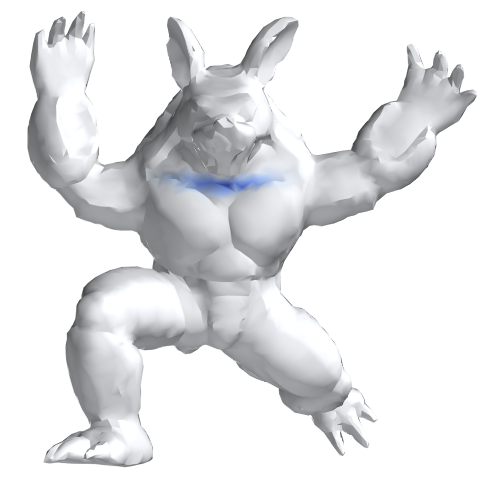}};

\begin{semilogyaxis}[
  xlabel=Number of iterations,
  ylabel=$L^2$ norm, opacity = 1]
\addplot [blue, thick = 1pt] table [x=iteration, y=primal]{data_convergence_ADMM.txt};
\addlegendentry{Primal residual}

\addplot [red, thick = 1pt] table[x=iteration, y=dual]{data_convergence_ADMM.txt};
\addlegendentry{Dual residual}

\addplot [blue, only marks] table[x=iterationMark, y=primalMark]{data_convergence_ADMM.txt};
\end{semilogyaxis}

\draw [->, blue, dashed] (0.58,3.93) -- (1.1,4.2) ;

\draw [->, blue, dashed] (0.63,3.73) -- (2.8,2.7) ;

\draw [->, blue, dashed] (6.3,0.5) -- (6,1.5) ;

\end{tikzpicture}
\caption{Amplitude of the primal and dual residual \cite[Section 3.3]{Boyd2011} in $L^2$ norm. The \review{distributions} $\bar{\mu}^0$ and $\bar{\mu}^1$ are delta functions located on respectively the right and left hand of the armadillo. We also show the midpoint $\mu^{1/2}$ for different numbers of iterations (10,50 and 5000). After a few hundred iterations, there is no visible difference in $\mu^{1/2}$. There is a jump in the value of the dual residual at around 4600 iterations.  It is due to a change in the value of the parameter $r$, which is updated according to the heuristic rule presented in Section 3.4.1 of \cite{Boyd2011}.}
\label{figure_convergence_ADMM}
\end{figure}
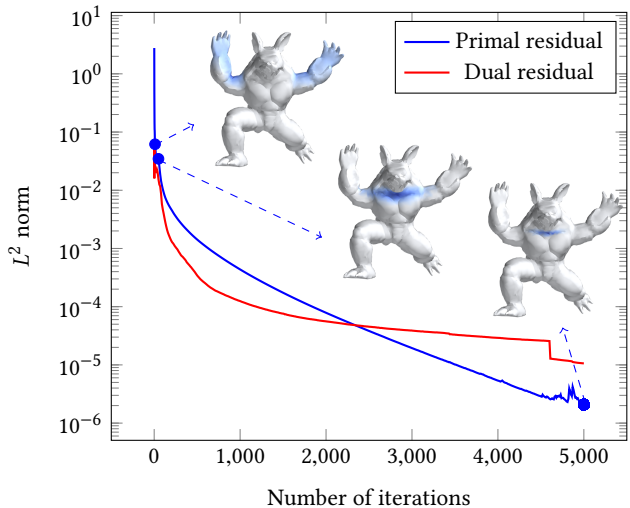

For fixed boundary data $\bar{\mu}^0$ and $\bar{\mu}^1$, we \review{plot} in Figure \ref{figure_convergence_ADMM} the primal and dual error \review{defined} by Boyd et al.~\shortcite{Boyd2011}, as \review{a} function of the number of iterations of the ADMM scheme. We usually need on the order of a few thousand iterations to satisfy our convergence criteria, this number being dependent of the boundary data $\bar{\mu}^0, \bar{\mu}^1$ (the more diffuse, the fewer iterations are needed).

Because our objective functional is scaled \review{according} to the geometry of the mesh (i.e.\ scalar products are weighted by the areas of the triangles and the number of time steps), the number of iterations needed does not depend on the size of the resolution of the mesh nor the number of discretization points in time, but the computation time needed per iteration does. \review{Typical values of the timings are given in  Table \ref{table_timing}, they are of the order of $1$ second per ADMM iterations for meshes with a few thousand vertices.}  

%For instance, on a 2.40GHz Intel i5 processor with 8GB RAM, in a Python implementation, for a mesh of the order of a few thousand vertices (like the Armadillo:  5002 vertices and 10000 faces) and 31 discretization points in time (the parameters of \ref{figure_convergence_ADMM}), each iteration took 1s.  

%\hugo{I need you guys to insert something about the CPU time}.
%Thanks to a rather small cost per iteration, the overall running time stays in the order of few minutes\hugo{True ??} \sebastian{Probably not. Depends on the mesh: takes on the order of days on my fairly old macbook even with the new SOCP code.}.

\begin{table*}
\begin{small}
\begin{tabular}{|c|c||c|c|c|c||c|c||c|c|}
\hline
Mesh & Figure & $N$ & $|V|$ & $|T|$ & $\alpha$ & ADMM Iters. & ADMM Time (s) & Mosek Time (s) & CVX Total (s)\\
\hline
Punctured sphere & 10 & 13 & 1020 & 2024 & 0.02 & 546 & 16 & 23 & 27 \\
Punctured sphere & 10 & 31 & 1020 & 2024 & 0.02 & 547 & 47 & 114 & 122 \\
Hand & 8 & 13 & 1515 & 3026 & 0.02 & 846 & 47 & 37 & 47 \\
Hand & 8 & 31 & 1515 & 3026 & 0.02 & 858 & 97 & 174 & 191 \\
Armadillo & 7 & 31 & 5002 & 10000 & 0 & 929 & 332 & 766 & 882 \\
Armadillo & 7 & 63 & 5002 & 10000 & 0 & 808 & 649 & 3719 & 3970 \\
Armadillo & 7 & 31 & 5002 & 10000 & 1 & 308 & 116 & 938$^\ast$ & 1054 \\
Face & 2 & 31  & 5002 & 10000 & 0.1 & 415 & 155 & 1829 & 1944\\
Airplane & 9 & 31 & 3772 & 7540 & 0.1 & 535 & 144 & 764 & 831\\
Planar square & 3 & 31 & 11838 & 23242 & 0 & 565 & 473 & 10270 & 11082\\
\hline
\end{tabular}
\end{small}
\caption{\review{Timing data for various meshes and boundary data from the figures (numbers listed in the table). $N$ denotes the number of time discretization points and $\alpha$ is the value of the congestion regularization parameter (see Section \ref{subsection_regularization}). For the ADMM method, the number of iterations and timing are given. Iterations were stopped once an error of $10^{-4}$ was reached for the $L^2$ norm of both the primal and dual residual. One can see that the time per iteration depends on the size of the mesh and the temporal grid, but the number of iteration is quite insensitive to these parameters and rather depends on the boundary conditions and the regularization parameter. For the CVX implementation of the optimization problem, the solver time and the total time (includes CVX pre-processing) are given. Standard precision settings were used, but are hard to interpret absolutely due to unknown algebraic rearrangement of the problem. $^\ast$ denotes that CVX reported a failure in this case. Results obtained on an 8-core 3.60GHz Intel i7 processor with 32GB RAM.}}
\label{table_timing}
\end{table*}

\subsection{\review{CVX implementation}}
\label{subsection_cvx}

\review{Since} the optimization problem in Equation~\eqref{equation_regularized_optimization} is a convex cone problem, we have also used a straightforward implementation in CVX \cite{cvx,gb08}, with Mosek as a solver~\cite{mosek}. \review{This approach is provided as a simpler alternative to the ADMM implementation, and has comparable performance on small meshes with standard precision settings (fewer than 1000 vertices). In general, it is difficult to compare the error thresholds across the two implementations due to algebraic rearrangements performed by CVX.}  \review{See Table \ref{table_timing}.}

\review{
\subsection{Convergence with discretization in space and time}}

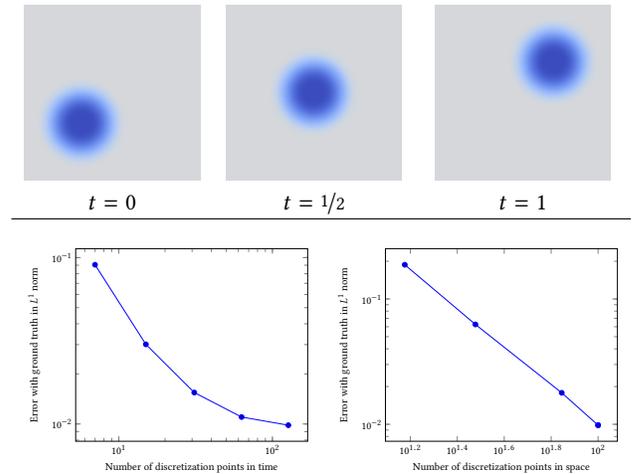
\begin{figure}
\begin{tabular}{cccccc}

\multicolumn{2}{c}{\includegraphics[width=.13\textwidth]{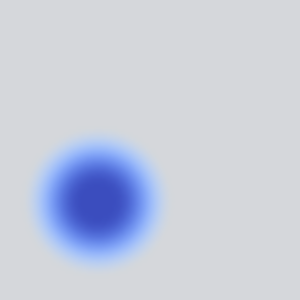}} &
\multicolumn{2}{c}{\includegraphics[width=.13\textwidth]{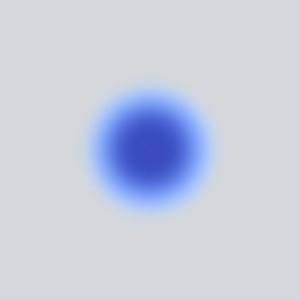}} &
\multicolumn{2}{c}{\includegraphics[width=.13\textwidth]{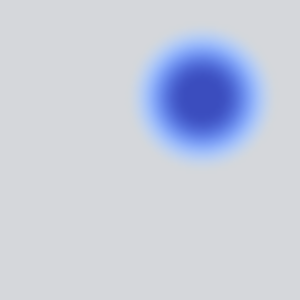}} \\ 
 
\multicolumn{2}{c}{$t=0$} &
\multicolumn{2}{c}{$t=\nicefrac{1}{2}$} &
\multicolumn{2}{c}{$t=1$} \\

\hline 
 
& & & & & \\

\multicolumn{3}{c}{
\begin{tikzpicture}[scale = 0.45]
\begin{loglogaxis}[
  xlabel=Number of discretization points in time,
  ylabel=Error with ground truth in $L^1$ norm]
\addplot table [x=TimeNumber, y=TimeError]{data_convergence_plane.txt};
\end{loglogaxis}
\end{tikzpicture}} 
&
\multicolumn{3}{c}{
\begin{tikzpicture}[scale = 0.45]
\begin{loglogaxis}[
  xlabel=Number of discretization points in space,
  ylabel=Error with ground truth in $L^1$ norm]
\addplot table [x=SpaceNumber, y=SpaceError]{data_convergence_plane.txt};
\end{loglogaxis}
\end{tikzpicture}}

\end{tabular}
\caption{\review{Top row: the test case. The inputs, i.e.\ probability distributions at times $t=0$ and $t=1$, correspond to the same density translated two different ways. Optimal transport predicts that at time $t=\nicefrac{1}{2}$ we should observe the same density again, but translated to the midpoint between the two inputs; this gives us ground truth we can use to verify our algorithm's output. %w.r.t.\ which we compare the output of our algorithm. 
Bottom row: convergence plots. On the left: error, measured in $L^1$ norm, where the mesh is fixed (regular triangle mesh with $100$ points per side of the square) and the number $N$ of discretization points in time varies. On the right: error, measured in $L^1$ norm, where the number of discretization points in time is fixed ($127$ points) and the mesh is a regular triangle mesh whose number of points per side varies and is plotted on the $x$-axis.}}
\label{figure_plane_convergence}
\end{figure}

\review{
As indicated in Section \ref{section_discrete_OT}, it is not known theoretically whether our discrete distance converges to the true Wasserstein distance when the mesh is refined. This is also the case as far as the time discretization is concerned; one could likely adapt the method of proof of Erbar et al.~\shortcite{Erbar2017}, but doing so is out of the scope of this article.

In Figure \ref{figure_plane_convergence}, however, we present some experiments indicating that convergence under space and time refinement is likely to be true. These were conducted in the simplest case: translation of a given density on a flat space. For this problem, the ground truth is known, and for a flat space it is clear what it means to refine the mesh:  We have use a regular triangle mesh with an increasing number of points per side. The error was evaluated at time $t= \nicefrac{1}{2}$ between the computed geodesic and the ground truth. As a measure of error, as the distributions are compactly supported, we use a total variation norm (in other words the $L^1$ norm between the densities) rather than the Kullback--Leibler divergence.  
As expected, we observe a decrease in error as the temporal and spatial meshes are refined.}

\subsection{Congestion and regularization}
\label{subsection_regularization}

\begin{figure*}
\begin{center}
\begin{tabular}{l|ccccc}
	 $\alpha = 0$ &
    \raisebox{-.5\height}{\includegraphics[width=.15\textwidth]{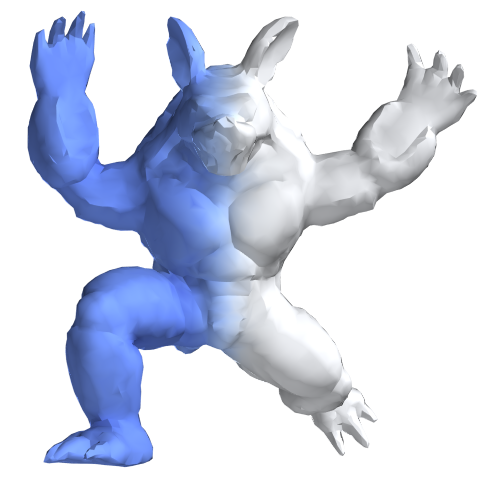}} &
   \raisebox{-.5\height}{ \includegraphics[width=.15\textwidth]{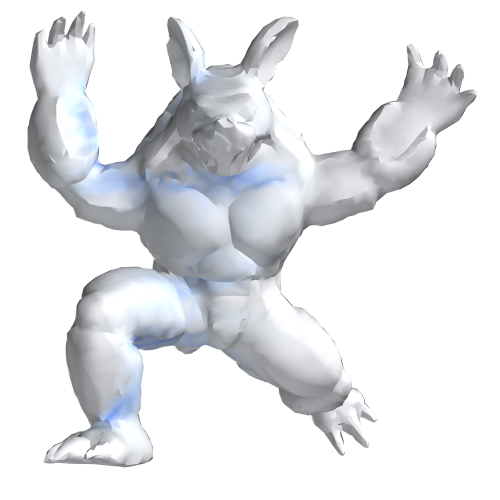}} &
    \raisebox{-.5\height}{\includegraphics[width=.15\textwidth]{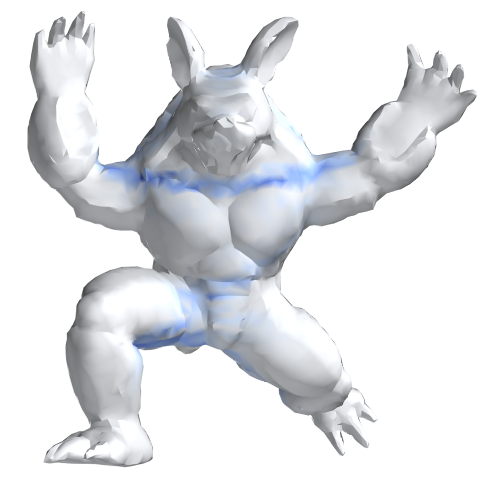}} &
    \raisebox{-.5\height}{\includegraphics[width=.15\textwidth]{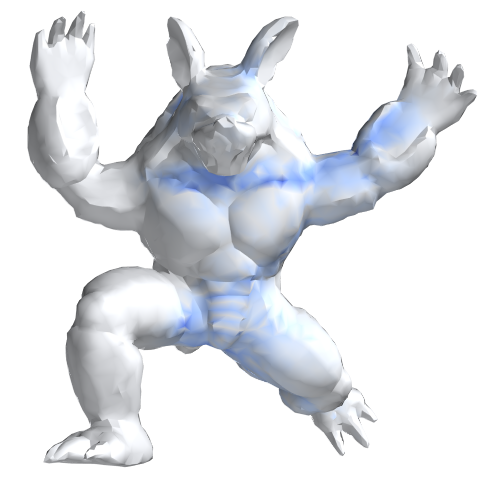}} &
    \raisebox{-.5\height}{\includegraphics[width=.15\textwidth]{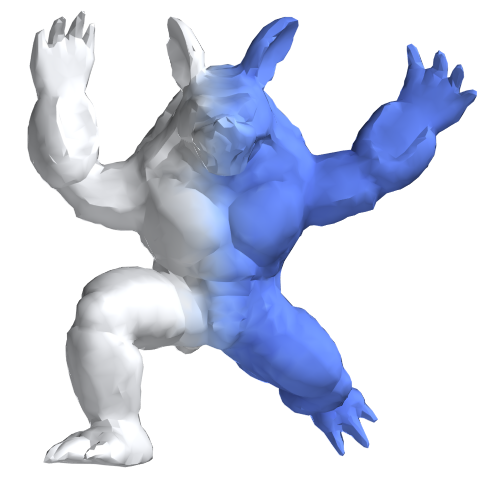}} \\

	 $\alpha = 10^{-2}$ &
    \raisebox{-.5\height}{\includegraphics[width=.15\textwidth]{fig8_3_0.png}} &
    \raisebox{-.5\height}{\includegraphics[width=.15\textwidth]{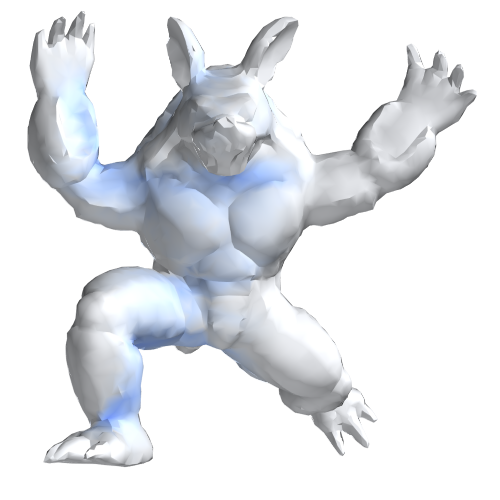}} &
    \raisebox{-.5\height}{\includegraphics[width=.15\textwidth]{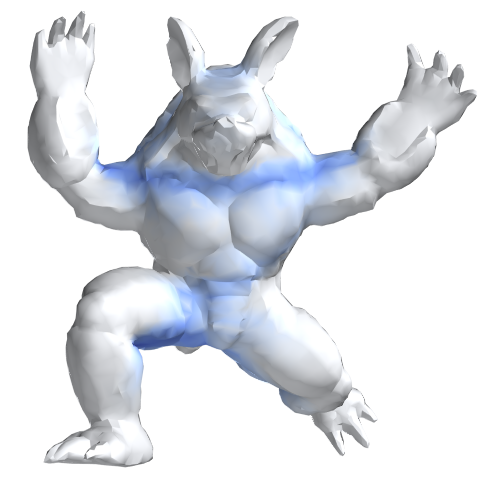}} &
    \raisebox{-.5\height}{\includegraphics[width=.15\textwidth]{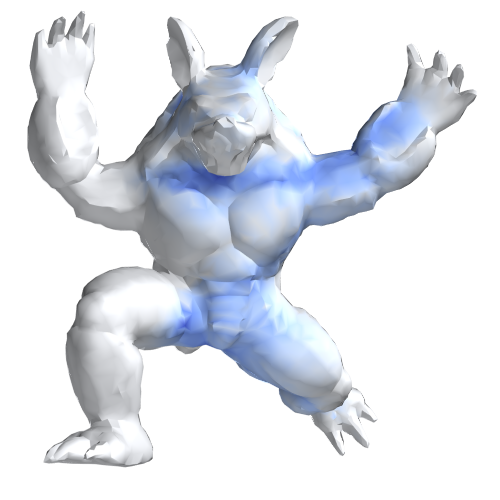}} &
    \raisebox{-.5\height}{\includegraphics[width=.15\textwidth]{fig8_3_30.png}} \\

    $\alpha = 10^{-1}$ &
    \raisebox{-.5\height}{\includegraphics[width=.15\textwidth]{fig8_3_0.png}} &
    \raisebox{-.5\height}{\includegraphics[width=.15\textwidth]{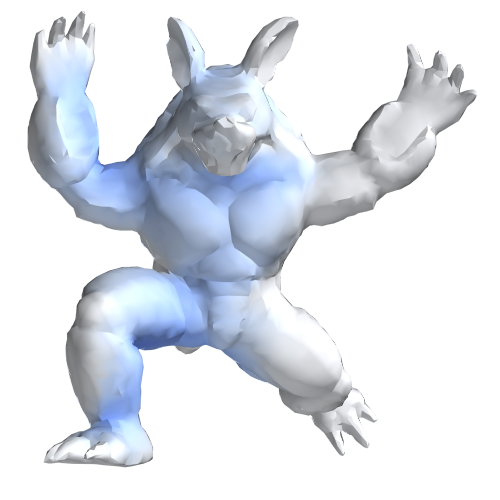}} &
    \raisebox{-.5\height}{\includegraphics[width=.15\textwidth]{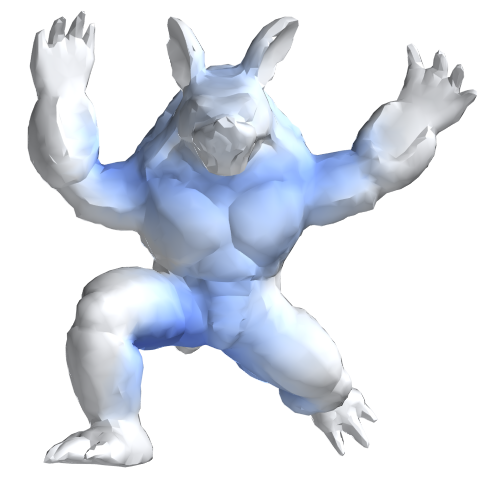}} &
    \raisebox{-.5\height}{\includegraphics[width=.15\textwidth]{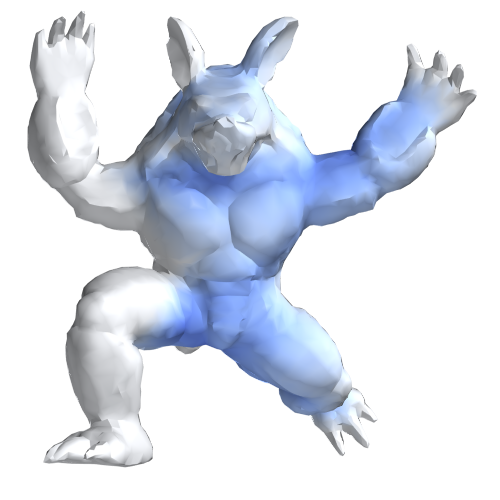}} &
    \raisebox{-.5\height}{\includegraphics[width=.15\textwidth]{fig8_3_30.png}} \\

    $\alpha = 1$ &
    \raisebox{-.5\height}{\includegraphics[width=.15\textwidth]{fig8_3_0.png}} &
    \raisebox{-.5\height}{\includegraphics[width=.15\textwidth]{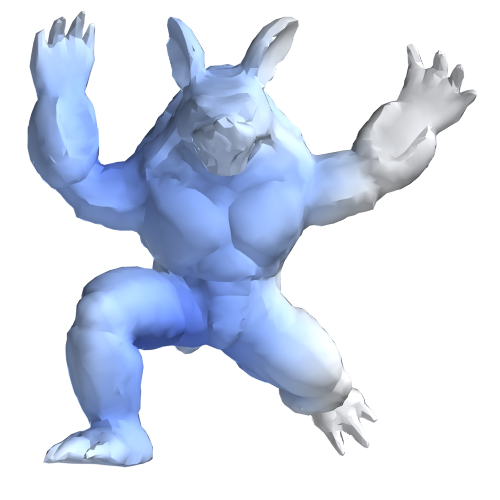}} &
    \raisebox{-.5\height}{\includegraphics[width=.15\textwidth]{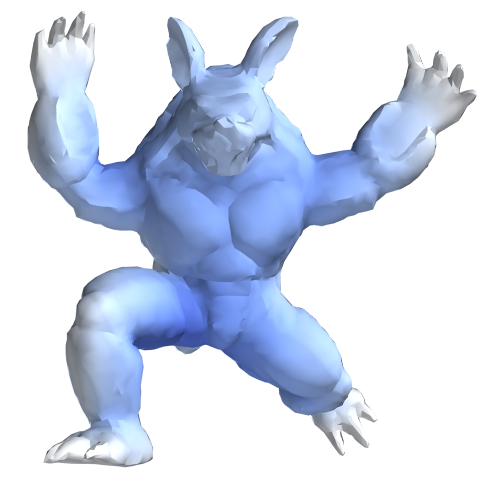}} &
    \raisebox{-.5\height}{\includegraphics[width=.15\textwidth]{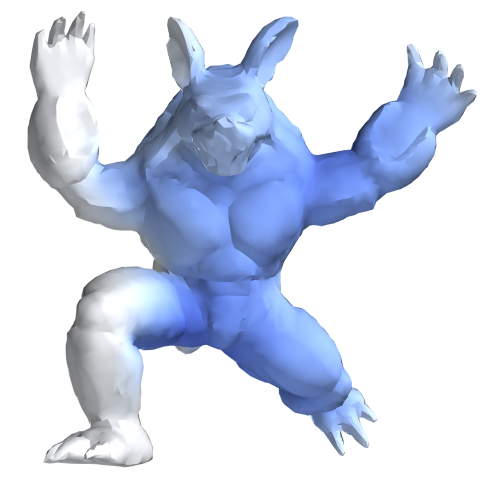}} &
    \raisebox{-.5\height}{\includegraphics[width=.15\textwidth]{fig8_3_30.png}} \\

	\hline

    & $t=0$ & $t=\nicefrac{1}{4}$ & $t = \nicefrac{1}{2}$ & $t=\nicefrac{3}{4}$ & $t=1$

\end{tabular}
\end{center}
\caption{Effect of the regularizing parameter $\alpha$ penalizing congestion. On each row, the interpolation between the same boundary data (distributions located on the right and on the left of the armadillo) is shown. Different rows correspond to different to different values of $\alpha$. The color in each image is normalized independently from the others, explaining the change in intensity. Mass is always preserved along the interpolation.}
\label{figure_regularization}
\end{figure*}

In optimal transport there is no price paid for highly-congested densities. \review{Imagine} the probability distributions as an assembly of particles moving along the surface. %evolving.
Along a geodesic in Wasserstein space, each particle evolves in time by following a geodesic on the surface---but does not feel the presence of its neighbors.

Now imagine, due to the particular structure of the triangle mesh, there is a small shortcut \review{in terms of geodesic distance through which} all geodesics tend to concentrate. This is likely to appear near a hyperbolic vertex \cite{Polthier2006}. Then all the particles have the incentive to take this shortcut, resulting in densely-populated zones, as they are not prone to congestion.  As an example, see the first row of Figure \ref{figure_regularization} in which, to go from the left to the right of the armadillo, all the particles go through only two paths, leaving the rest of the mesh without any \review{mass.}

This effect, although visually unpleasant, would be observed on a smooth surface $\M$ as soon as geodesics concentrate in some regions. A way to remove this artifact is to penalize congestion; we can do so with little modification to the algorithm.

We penalize the densities by their $L^2$ \review{norms}: The choice of the exponent $2$ is important, as it \review{preserves} the quadratic structure of the optimization problem. Namely, we add to the Lagrangian \eqref{equation_discrete_lagrangian} the term
\begin{equation}
\review{\frac{\alpha \tau}{2} \sum_{t \in \Gtimec} \sum_{v \in V} |v| |\mu^t_v|^2 = \sup_{\lambda} \sum_{t \in \Gtimec} \sum_{v \in V} \tau |v| \left( \lambda^t_v \mu_v^t - \frac{1}{2 \alpha} (\lambda_v^t)^2  \right),}
\end{equation}
where the parameter $\alpha$ tunes the scale of the congestion effect and $\lambda \in \R^{\Gtimec \times V}$ corresponds to the dual variable associated to the congestion constraint.

Using the notation from Section \ref{subsection_algorithm}, one can write the problem as maximizing
\begin{equation} \label{equation_regularized_optimization}
\max_{\hat{\Lambda}(\varphi,\lambda)=q }F(\varphi,\lambda) + C(q),
\end{equation}
but this time
\begin{equation}
F(\varphi,\lambda) = \eqref{equation_definition_F} - \review{\frac{1}{2\alpha} \sum_{t \in \Gtimec} \sum_{v \in V} \tau |v|(\lambda_v^t)^2}
\end{equation}
and $\hat{\Lambda}(\varphi, \lambda) = (-\lambda,0) - \Lambda(\varphi)$. Then one runs exactly the same algorithm, with a straightforward adaptation of the update formulas.

After regularization, the interpolation is no longer a geodesic.  For instance, the interpolation between two instances of the same probability distribution is not constant in time, because the $L^2$ norm potentially can be reduced by diffusing outward in the intermediate time steps. % because it tends to minimize its $L^2$ norm. 
On the other hand, undesirable sharp features and oscillations can be removed, as seen in Figure \ref{figure_regularization}. Note that regardless of the level of regularization, the interpolating curves are still valued in \review{$\P(S)$,} i.e.\ mass is still preserved along the interpolation.

The tuning of the parameter $\alpha$ allows our method to be robust to noisy mesh inputs, as shown in Figure \ref{figure_noisy}. Noisy meshes have more local variation in curvature, leading to a higher tendency for congested trajectories, but this can be tamed via greater regularization.

\review{Recall that the dynamical formulation of optimal transport can be interpreted as the least-action principle for a pressureless gas. The effect of the penalization of congested densities can be seen, from the modeling point of view, as adding a pressure force: the trajectories of the moving particles are no longer geodesics, they are bent by the pressure forces. %translates, at the level of the trajectories of the particles, by the appearance of pressure forces. 
The congestion term can also be see as an instance of variational mean field games, for which the augmented Lagrangian approach has been applied for flat spaces with grid discretization \cite{Benamou2017}.}

\review{Rather than a drawback, we see the regularization as an added feature of our method. Without regularization, one has a faithful discrete Benamou--Brenier formula on discrete surfaces with Riemannian structure. For applications in physics or gradient flows (Subsection \ref{subsection_gradient_flows}), this is likely the preferable formulation. For graphics, where blurriness might be sharpened \emph{a posteriori}, penalizing concentration of mass is reasonable. Either can be achieved thanks to our regularization term without additional computational cost and only by adding a few lines of code.}

\begin{figure*}
\begin{center}
\begin{tabular}{cccccc}
	\includegraphics[width=.14\textwidth]{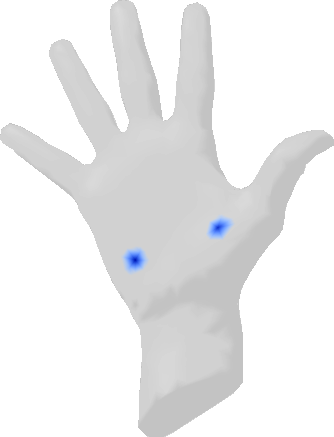} &
	\includegraphics[width=.14\textwidth]{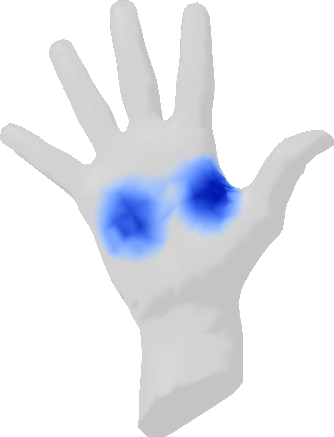} &
	\includegraphics[width=.14\textwidth]{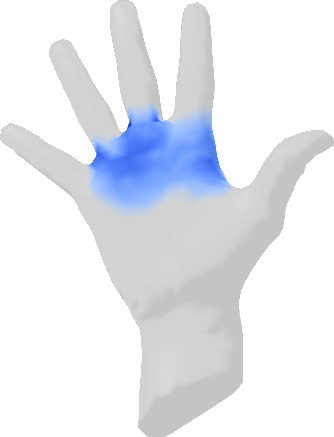} &
	\includegraphics[width=.14\textwidth]{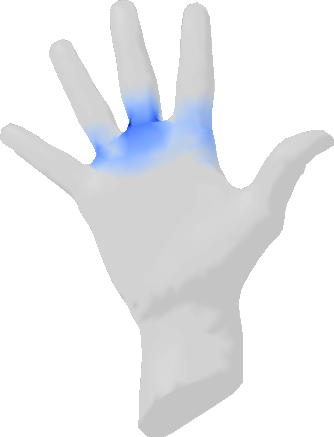} &
	\includegraphics[width=.14\textwidth]{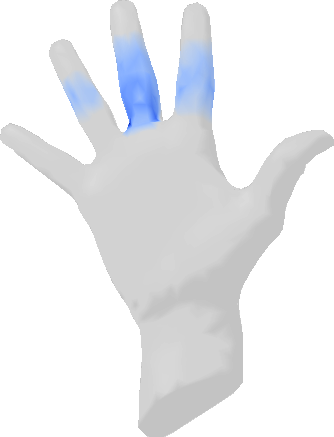} &
	\includegraphics[width=.14\textwidth]{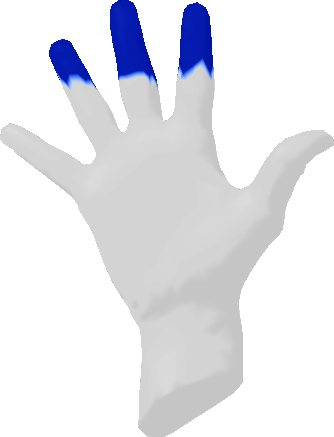} \\

	\includegraphics[width=.14\textwidth]{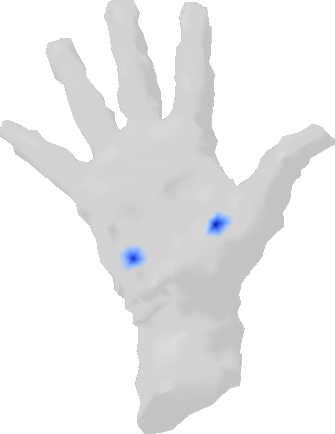} &
	\includegraphics[width=.14\textwidth]{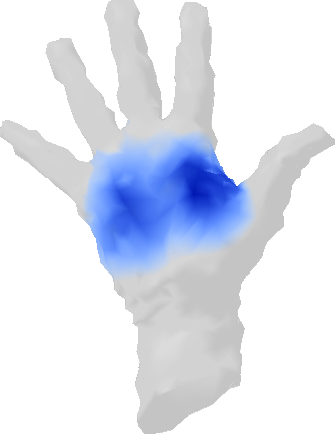} &
	\includegraphics[width=.14\textwidth]{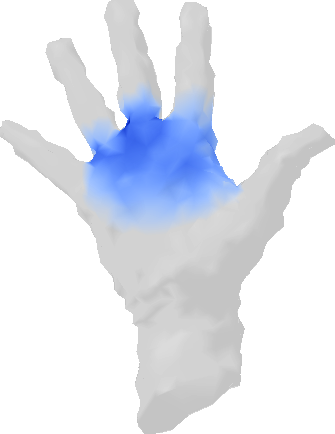} &
	\includegraphics[width=.14\textwidth]{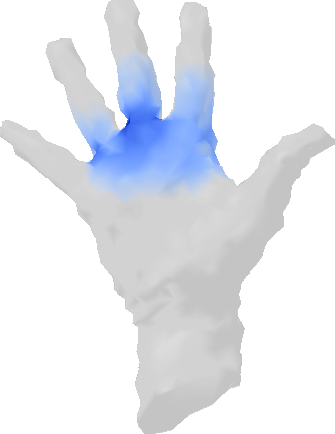} &
	\includegraphics[width=.14\textwidth]{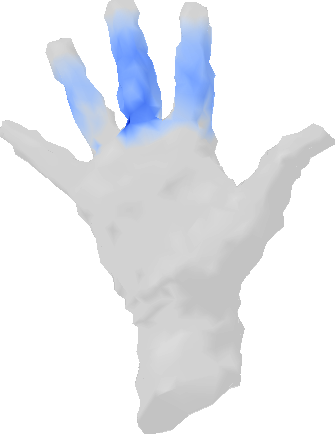} &
	\includegraphics[width=.14\textwidth]{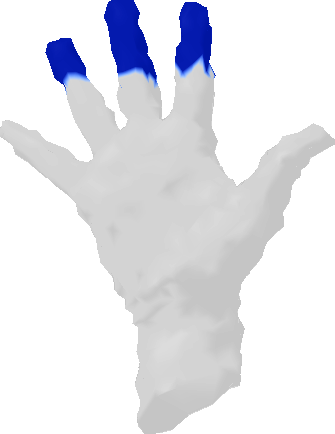} \\

	\includegraphics[width=.14\textwidth]{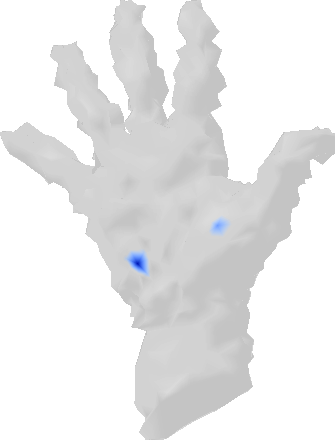} &
	\includegraphics[width=.14\textwidth]{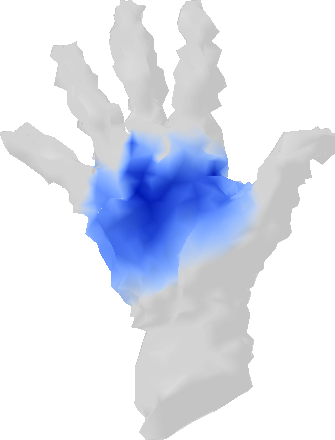} &
	\includegraphics[width=.14\textwidth]{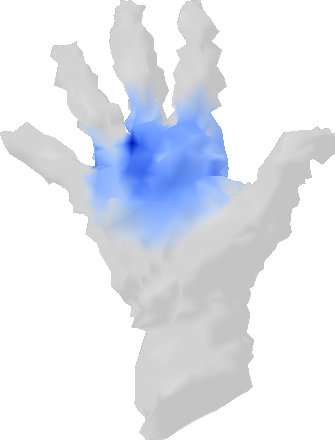} &
	\includegraphics[width=.14\textwidth]{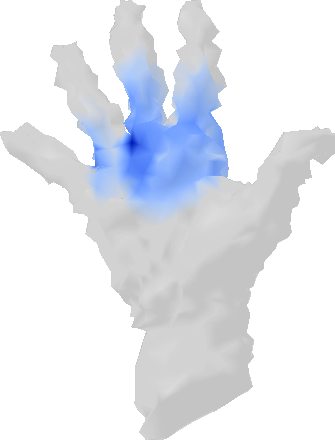} &
	\includegraphics[width=.14\textwidth]{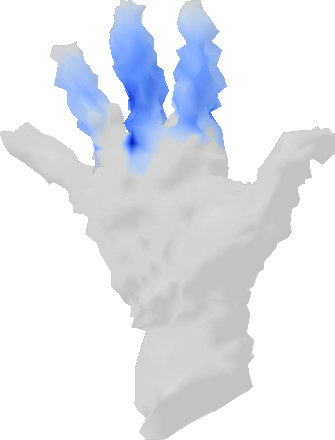} &
	\includegraphics[width=.14\textwidth]{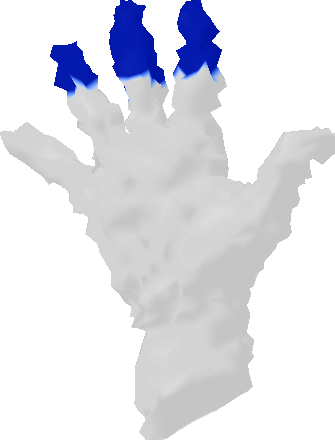} \\

	\hline

    $t=0$ & $t=\nicefrac{1}{5}$ & $t = \nicefrac{2}{5}$ & $t=\nicefrac{3}{5}$ & $t=\nicefrac{4}{5}$ & $t=1$
\end{tabular}
\end{center}
\caption{Robustness to noisy meshes, after adjusting the parameter $\alpha$. Top row: original mesh, $\alpha = 0.02$; middle row: noisy mesh, $\alpha = 0.1$; bottom row: very noisy mesh, $\alpha = 0.2$. The bounding boxes of the meshes were of side length \textasciitilde 1.5. Noisy mesh vertices were obtained by uniformly random perturbation, in the normal direction, of magnitudes up to 0.02 and 0.04, for the middle and bottom row, respectively.}
\label{figure_noisy}
\end{figure*}

\subsection{Intrinsic geometry}

\begin{figure}
\begin{center}
\begin{tabular}{ccc}

\includegraphics[width=.14\textwidth]{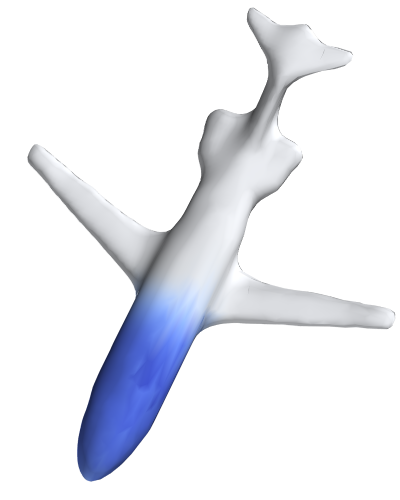} &
\includegraphics[width=.14\textwidth]{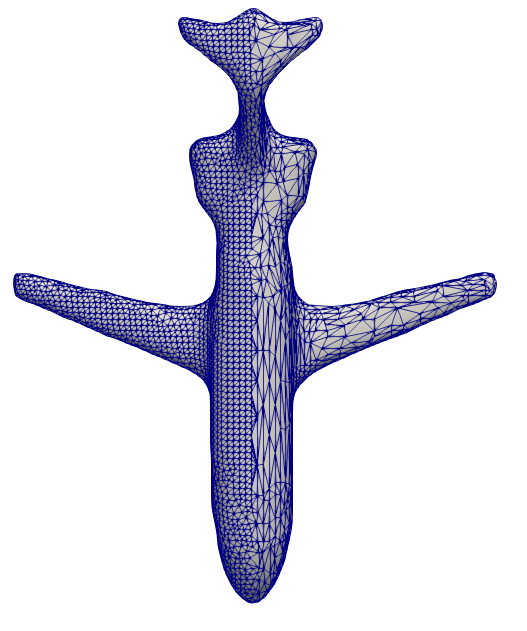} &
\includegraphics[width=.14\textwidth]{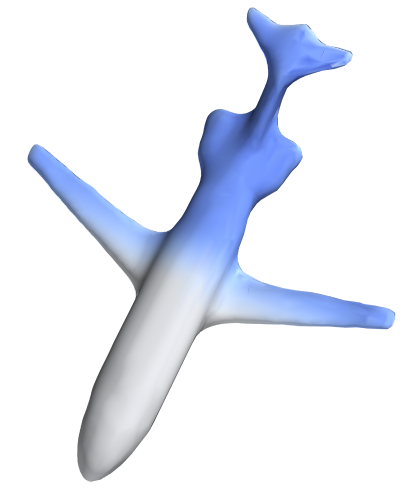} \\

\hline

$t=0$ & mesh & $t=1$ \\

\includegraphics[width=.14\textwidth]{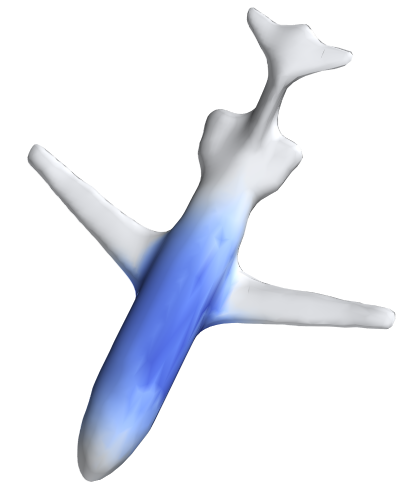} &
\includegraphics[width=.14\textwidth]{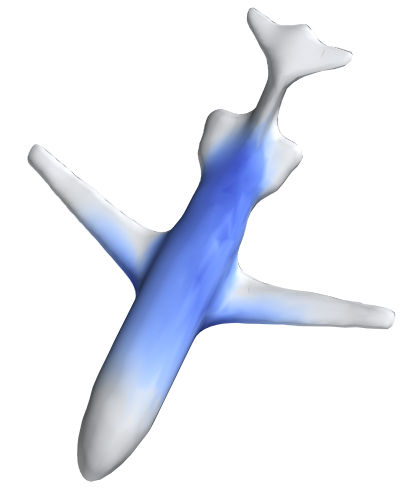} &
\includegraphics[width=.14\textwidth]{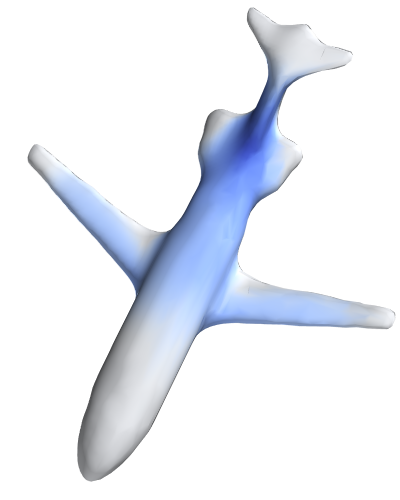} \\

\hline 

$t=\nicefrac{1}{4}$ & $t=\nicefrac{1}{2}$ & $t=\nicefrac{3}{4}$

\end{tabular}
\end{center}
\caption{Top row: mesh and initial/final probability distributions. Notice the difference of coarseness in the mesh. Bottom row: interpolation shown at different times where no effect of the difference in coarseness is seen. We have used the regularization described in Subsection \ref{subsection_regularization} with $\alpha = 0.1$.}
\label{figure_airplane_coarse_fine}

\end{figure}

%\begin{figure}
%\begin{tikzpicture}
%% Upper
%\draw (-3,-2) node{$t=0$} ;
%\node[inner sep=0pt] at (-3,0)
%    {\includegraphics[width=.14\textwidth]{figs/airplane_mu_0.png}};
%
%\draw (0,-2) node{mesh} ;
%\node[inner sep=0pt] at (0,0)
%    {\includegraphics[width=.14\textwidth]{figs/airplane_mesh.png}};
%
%\draw (3,-2) node{$t=1$} ;
%\node[inner sep=0pt] at (3,0)
%    {\includegraphics[width=.14\textwidth]{figs/airplane_mu_30.png}};
%
%
%% Lower
%\node[inner sep=0pt] at (-3,-4.1)
%    {\includegraphics[width=.14\textwidth]{figs/airplane_mu_7.png}};
%\draw (-3,-6.1) node{$t=\nicefrac{1}{4}$} ;
%\node[inner sep=0pt] at (0,-4.1)
%    {\includegraphics[width=.14\textwidth]{figs/airplane_mu_15.png}};
%\draw (0,-6.1) node{$t=\nicefrac{1}{2}$} ;
%\node[inner sep=0pt] at (3,-4.1)
%    {\includegraphics[width=.14\textwidth]{figs/airplane_mu_23.png}};
%\draw (3,-6.1) node{$t=\nicefrac{3}{4}$} ;
%\end{tikzpicture}
%
%\end{figure}

To illustrate the fact that the discrete Wasserstein metric is really associated to the geometric structure of the mesh, we perform the following experiment. We design a mesh where the right part is much coarser than the left one, and we let the density evolve. As one can see in Figure \ref{figure_airplane_coarse_fine}, %the picture is qualitatively the same \justin{as what?} and
the jump in coarseness does not affect the density and does not produce any numerical artifact.

\subsection{Arbitrary topologies}

\begin{figure}
\begin{tabular}{cccc}
	\includegraphics[width=.095\textwidth]{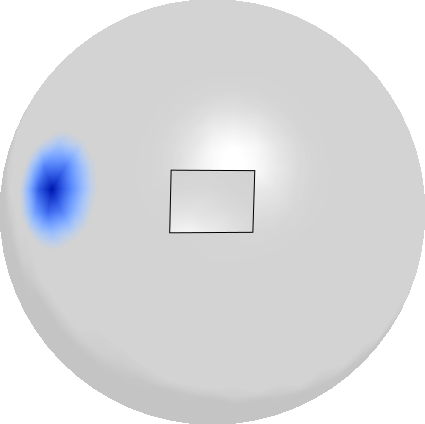} &
	\includegraphics[width=.095\textwidth]{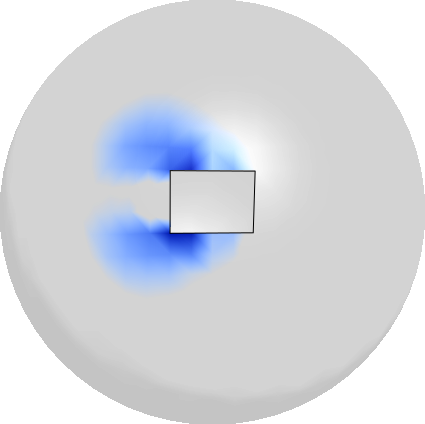} &
	\includegraphics[width=.095\textwidth]{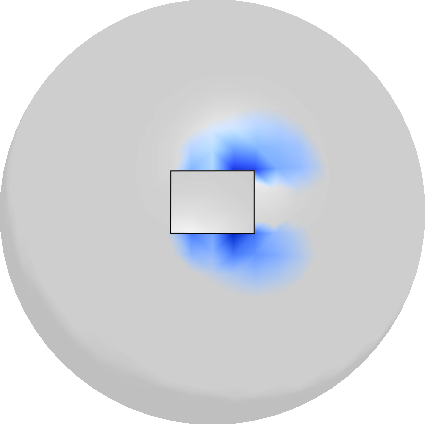} &
	\includegraphics[width=.095\textwidth]{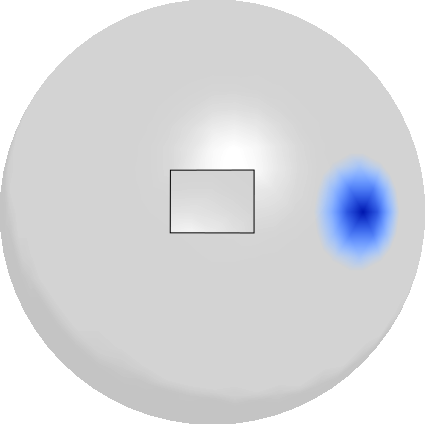} \\

	\includegraphics[width=.095\textwidth]{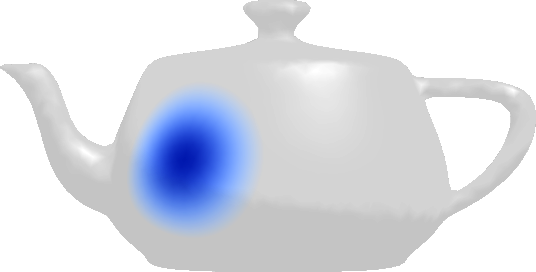} &
	\includegraphics[width=.095\textwidth]{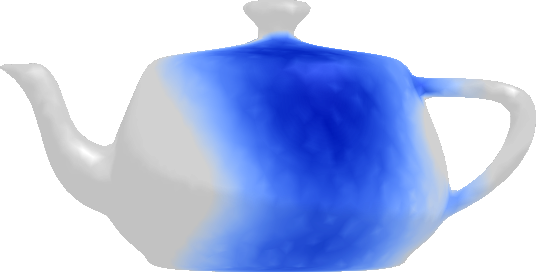} &
	\includegraphics[width=.095\textwidth]{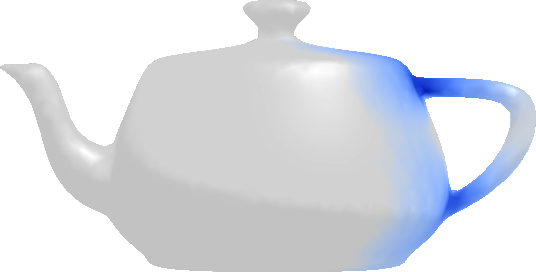} &
	\includegraphics[width=.095\textwidth]{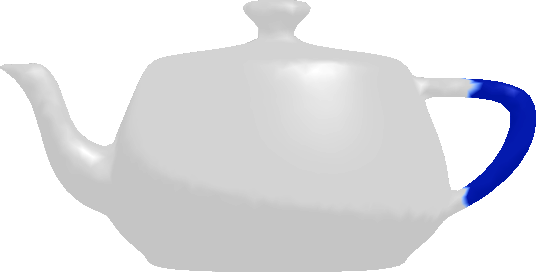} \\

	\hline

    $t=0$ & $t=\nicefrac{1}{3}$ & $t = \nicefrac{2}{3}$ & $t=1$
\end{tabular}
\caption{Our formulation easily handles non-spherical topologies. In the top row is a punctured sphere, and in the bottom row is a genus-1 teapot mesh. These interpolations were generated with $\alpha = 0.02$ and $\alpha = 0.2$, respectively.}
\label{figure_topology}
\end{figure}

The discrete formulation that we have chosen applies \review{without change} %immediately generalizes 
to meshes with boundary and those of non-spherical topology. This is illustrated in Figure \ref{figure_topology} with two meshes topologically equivalent to a disc and a torus. 

In the first example, the interpolating distribution stays near the boundary, approximately following the geodesic between the means of the endpoint distributions.  In the second example, one can see the initial distribution splitting to travel both ways to the other side of a handle, before merging again to achieve the final distribution. 

%\subsection{Comparison to other methods}

\subsection{Comparison to convolutional method}

Solomon et al.~\shortcite{Solomon2015} provide a convolutional method for approximating the Wasserstein geodesic between two distributions supported on triangle meshes. Their approach solves a regularized optimal transport barycenter problem using a modified Sinkhorn algorithm, with a heat kernel taking the place of explicitly-calculated pairwise distances between vertices. As a result, their method blurs the input distributions, and the interpolated distributions are typically of higher entropy than the endpoints. This is combated with a nonconvex projection method that attempts to lower the entropy of intermediate distributions to an approximated bound.

In comparing our methods, we found that \cite{Solomon2015} also tends to produce interpolating distributions that do not travel with constant speed. This effect can be seen in Figures \ref{figure_starzip} and \ref{figure_horsezip}, where their interpolating distributions remain mostly stationary for times near $t=0$ and $t=1$, but move \review{with} high speed for times near $t = \nicefrac{1}{2}$. Loosening the entropy bound in the nonconvex step \review{helps somewhat,} but the problem persists regardless. Most likely this effect is due to the fact that the entropy reduction step of their algorithm is not geometry-aware but rather simply sharpens the regularized interpolant.

Our method does not suffer from this issue, and the spread of our interpolating distributions is comparable or better in both cases. Furthermore, unless the regularizer $\alpha$ is large, our interpolating distributions tend to diffuse only in the direction of the geodesics along which particles are traveling, which better mimics the behavior of Wasserstein geodesics; this diffusion is reduced by adding more time steps to our interpolation problem.

%\ed{Will try to show decreasing diffusion for increasing time steps, in an overnight run.}

Our formulation also has comparable runtimes to the convolutional method of Solomon et al.~\shortcite{Solomon2015}. For instance, the implementation of the convolutional method provided by the authors of that paper took 57 and 141 seconds to converge, on the punctured sphere (1020 vertices) and teapot (3900 vertices), respectively, for 13 time steps. This is to be put in comparision with the timings provided in Subsection \ref{subsection_cvx}.

%If we go back to the case of a smooth manifold $M$, it is known that the optimal transport between two delta functions $\bar{\mu}^0 = \delta_{x_0}$ and $\bar{\mu}^1 = \delta_{x_1}$ stays a delta function, namely $\mu^t = \delta_{x_t}$ where $(x_t)_{t \in [0,1]}$ is the geodesic on $\M$ between $x_0$ and $x_1$, provided it is unique. Numerical methods interpolating between delta functions have a tendency to introduce diffusion, and the mass tends to spread out rather than stay on the geodesic. Our method, however, tends to diffuse only in the direction of the geodesic, and this issue can be reduced by increasing the number of time steps.

%\hugo{
%In particular, we do better than:
%\begin{itemize}
%\item Entropic regularization
%\item Compare with non OT method (in a case where it's bad).
%\end{itemize}
%Also some plot of: diffusion (e.g.\ variance) of the midpoint in function of the number of time steps.
%}

%To validate our method against recent results, w

% We compare our interpolant with the methods of Solomon et al.~\shortcite{Solomon2015} and Azencot et al.~\shortcite{Azencot2016} in Figure~\ref{figure_delta_comp}. In contrast to both of these methods, we recover sharper estimates for the interpolated distribution. Our algorithm exhibits comparable performance to that of \cite{Solomon2015} and is significantly more efficient than that of \cite{Azencot2016}.

The comparisons in this section were computed on a 3.60GHz Intel i7-7700 processor with 32GB of RAM. For the convolutional method, the heat kernel was used to diffuse to $t=0.0015$ with 10 implicit Euler steps.

\begin{figure*}
\begin{tabular}{cccccccc}

	\includegraphics[width=.095\textwidth]{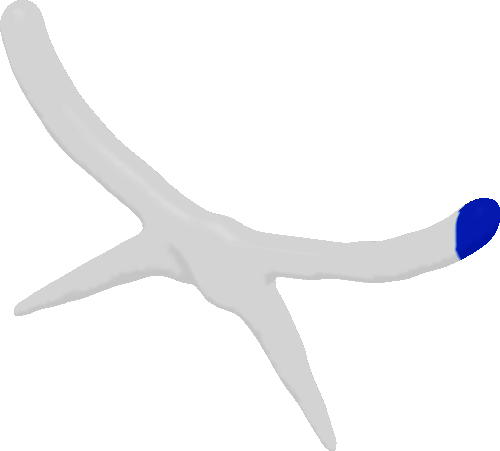} &
	\includegraphics[width=.095\textwidth]{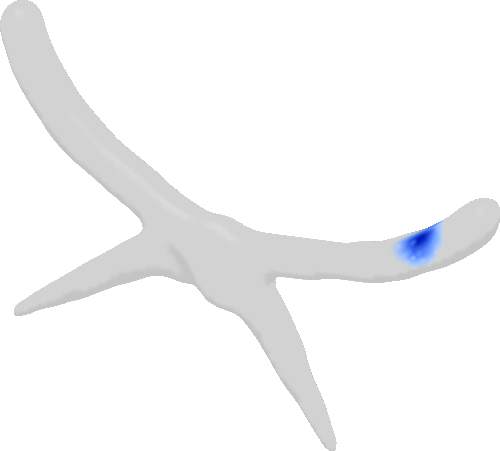} &
	\includegraphics[width=.095\textwidth]{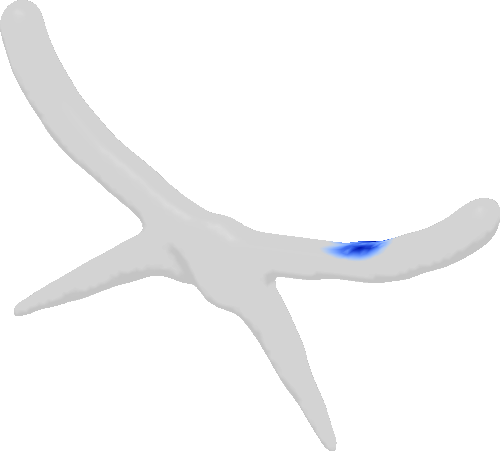} &
	\includegraphics[width=.095\textwidth]{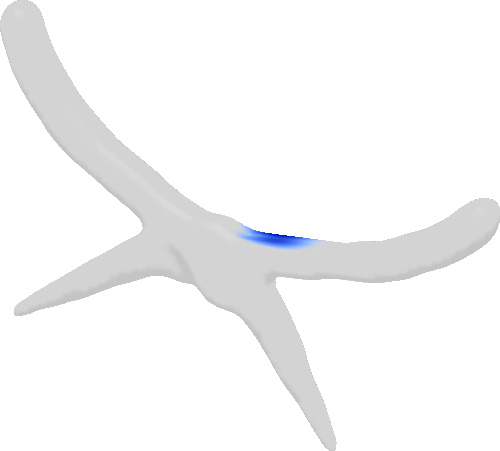} &
	\includegraphics[width=.095\textwidth]{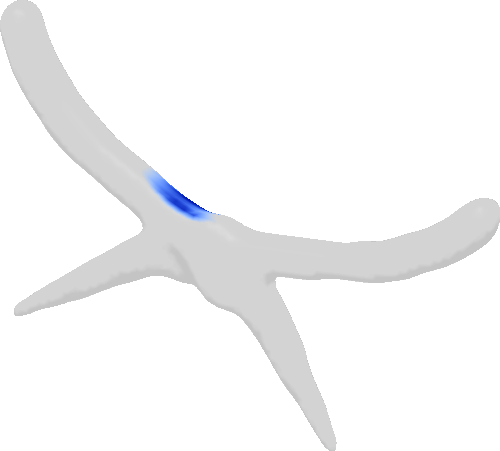} &
	\includegraphics[width=.095\textwidth]{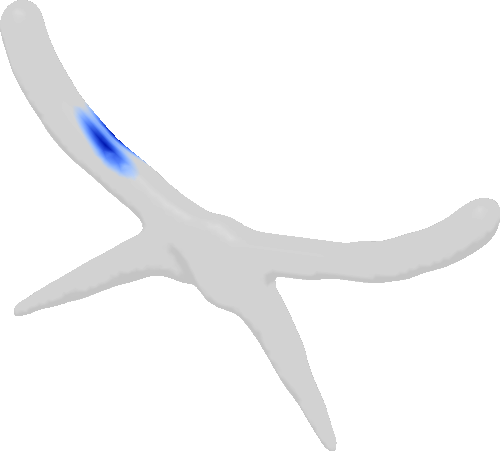} &
	\includegraphics[width=.095\textwidth]{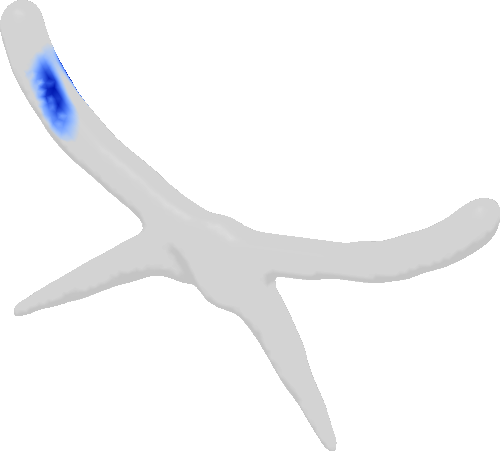} &
	\includegraphics[width=.095\textwidth]{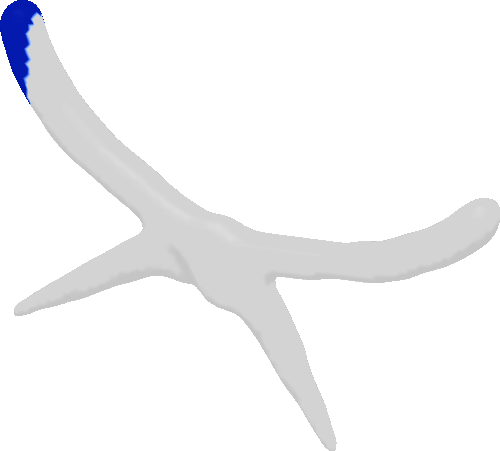} \\

	\includegraphics[width=.095\textwidth]{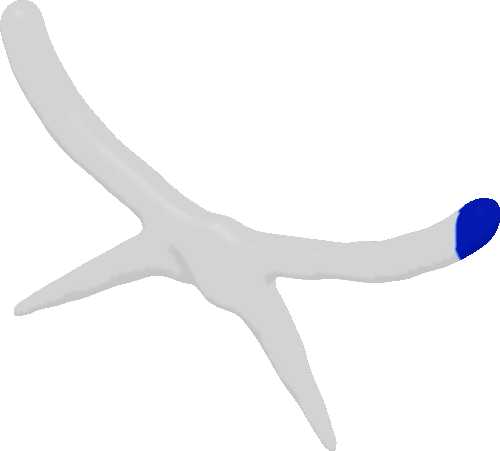} &
	\includegraphics[width=.095\textwidth]{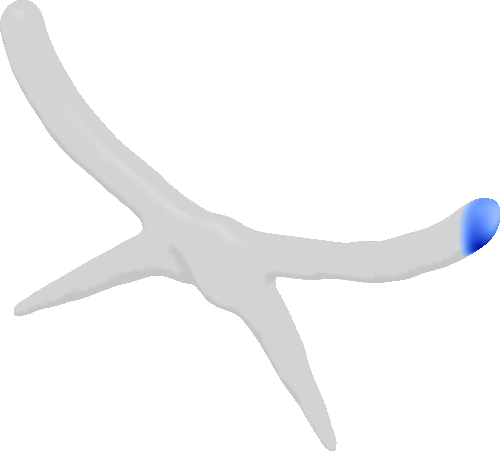} &
	\includegraphics[width=.095\textwidth]{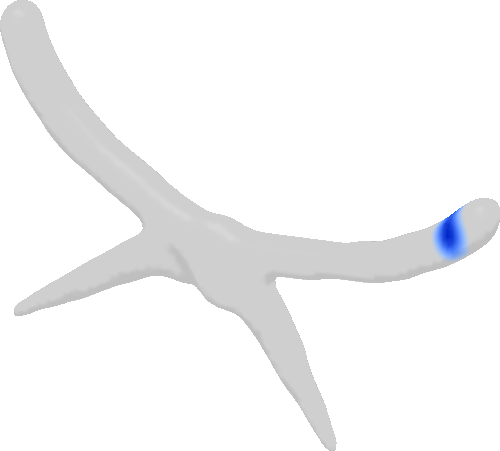} &
	\includegraphics[width=.095\textwidth]{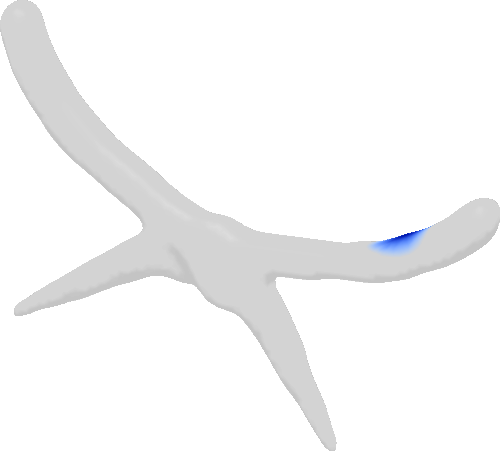} &
	\includegraphics[width=.095\textwidth]{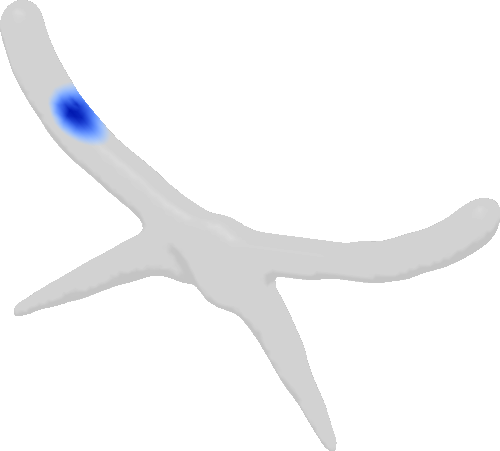} &
	\includegraphics[width=.095\textwidth]{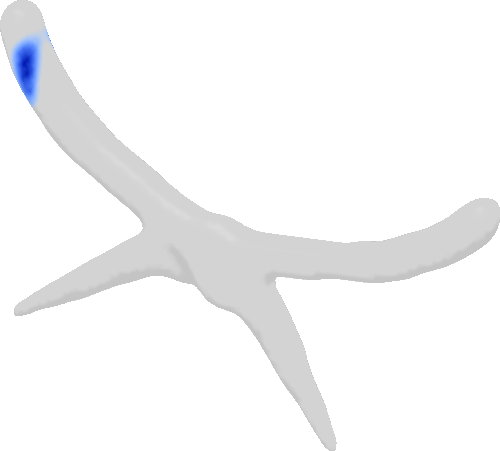} &
	\includegraphics[width=.095\textwidth]{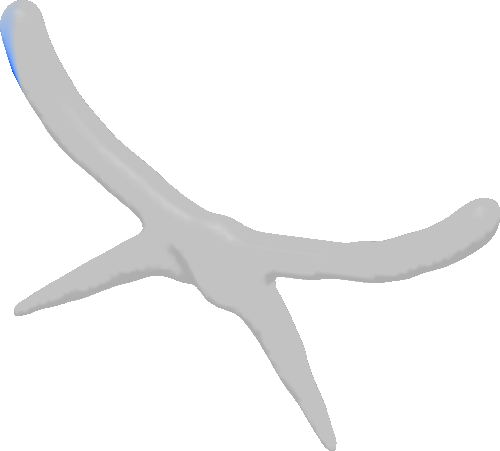} &
	\includegraphics[width=.095\textwidth]{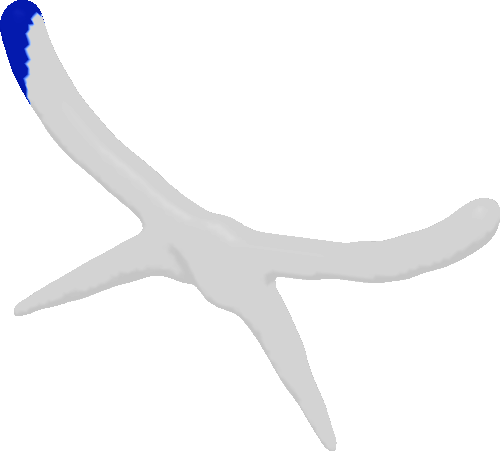} \\

	\includegraphics[width=.095\textwidth]{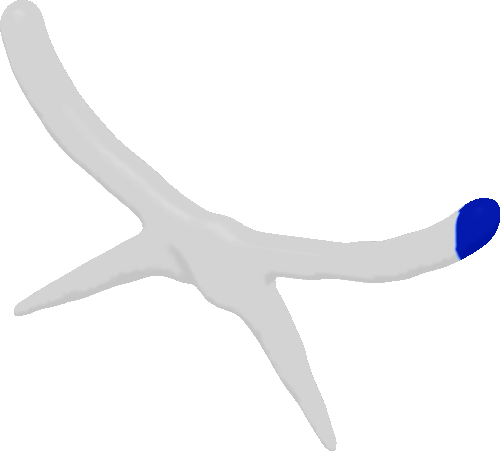} &
	\includegraphics[width=.095\textwidth]{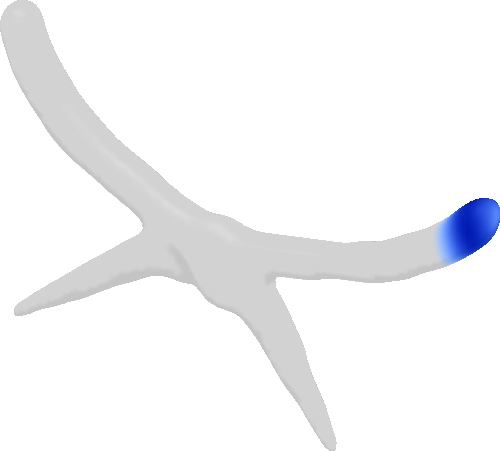} &
	\includegraphics[width=.095\textwidth]{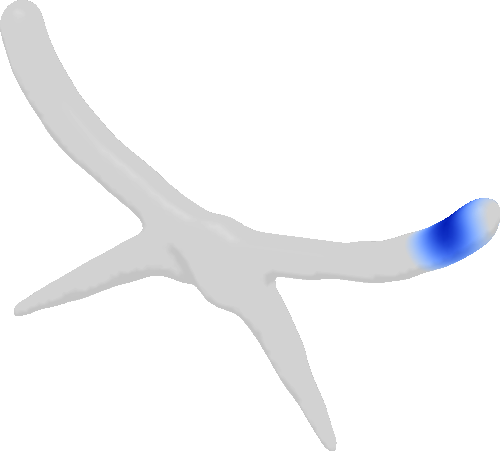} &
	\includegraphics[width=.095\textwidth]{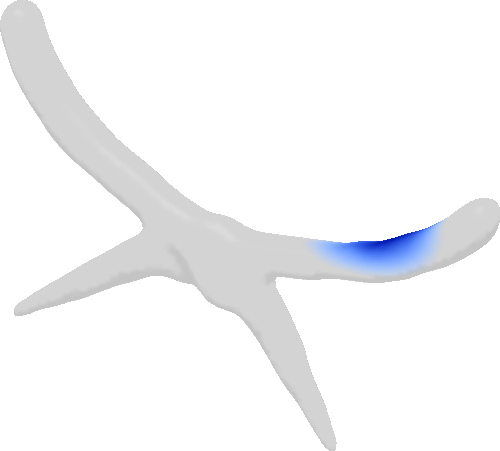} &
	\includegraphics[width=.095\textwidth]{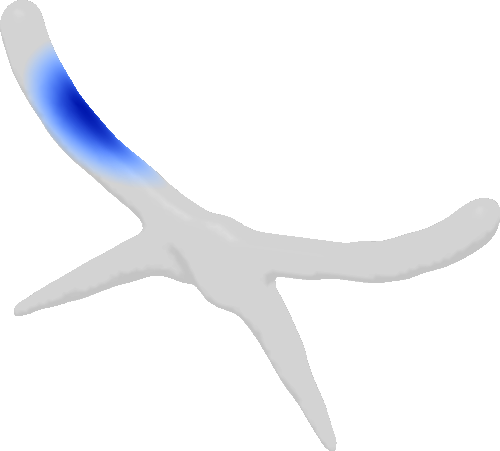} &
	\includegraphics[width=.095\textwidth]{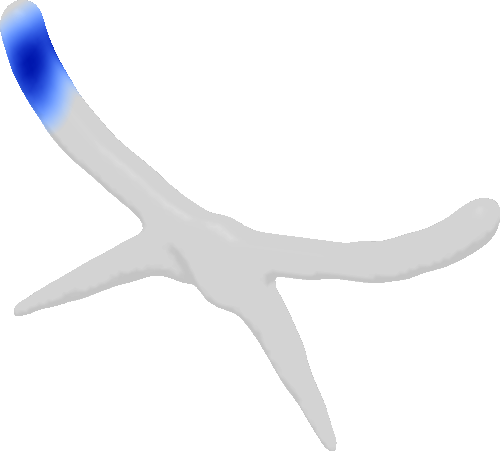} &
	\includegraphics[width=.095\textwidth]{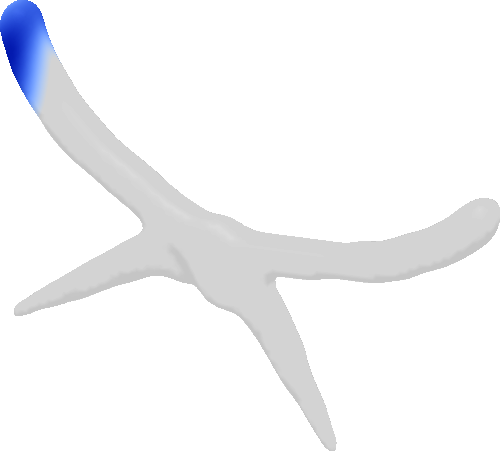} &
	\includegraphics[width=.095\textwidth]{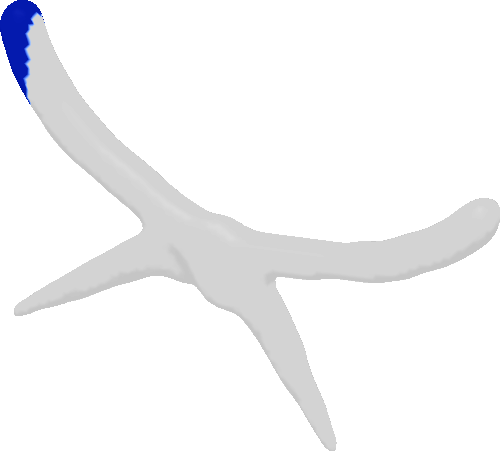} \\
	
	\hline 
	
	$t=0$ & $t =\nicefrac{1}{7}$ & $t =\nicefrac{2}{7}$ & $t =\nicefrac{3}{7}$ & $t =\nicefrac{4}{7}$ & $t =\nicefrac{5}{7}$ & $t =\nicefrac{6}{7}$ & $t=1$	
\end{tabular}
\caption{Constant-speed interpolation. Indicator distributions on handle ends of a pliers mesh are interpolated. Top row: our method, calculated with $\alpha = 0.001$; middle row: method of Solomon et al.~\protect\shortcite{Solomon2015}, calculated with entropy bounded by that of the endpoint distributions; bottom row: method of Solomon et al.~\protect\shortcite{Solomon2015}, calculated with no entropy bound. As can be seen, the method of Solomon et al.~\protect\shortcite{Solomon2015} stays mostly stationary except for the middle frames.}
\label{figure_starzip}
\end{figure*}

\begin{figure*}
\begin{tabular}{cccccccc}
	\includegraphics[width=.095\textwidth]{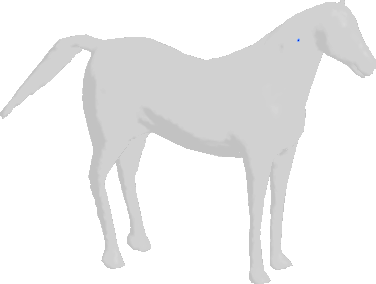} &
	\includegraphics[width=.095\textwidth]{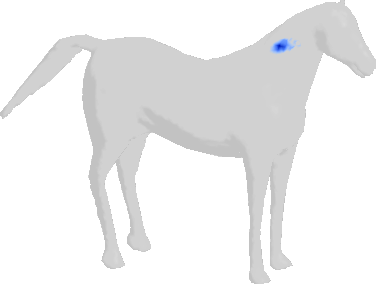} &
	\includegraphics[width=.095\textwidth]{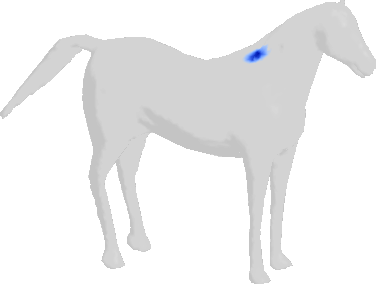} &
	\includegraphics[width=.095\textwidth]{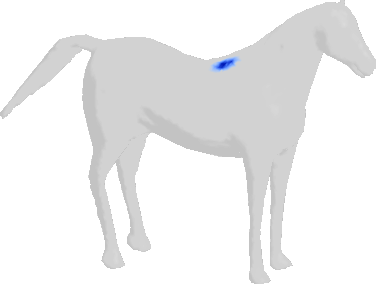} &
	\includegraphics[width=.095\textwidth]{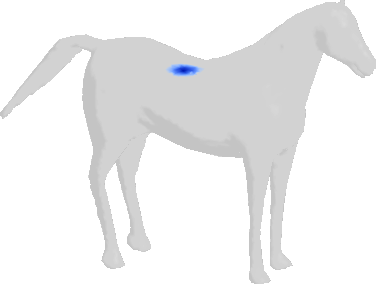} &
	\includegraphics[width=.095\textwidth]{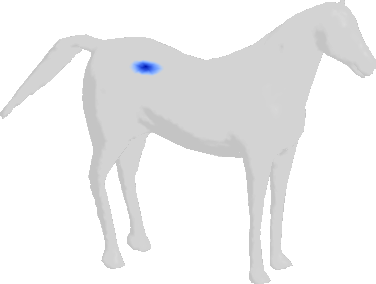} &
	\includegraphics[width=.095\textwidth]{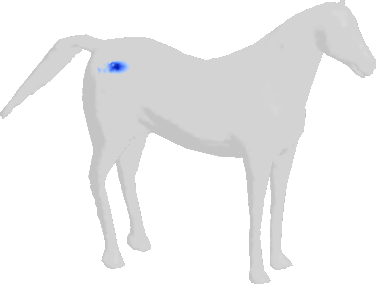} &
	\includegraphics[width=.095\textwidth]{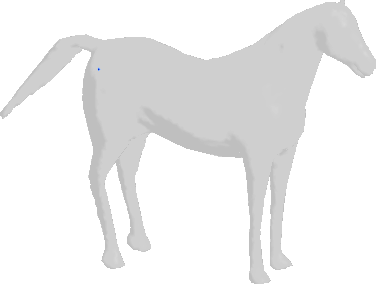} \\

	\includegraphics[width=.095\textwidth]{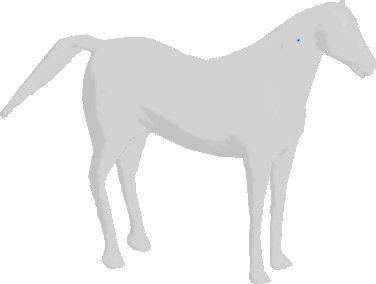} &
	\includegraphics[width=.095\textwidth]{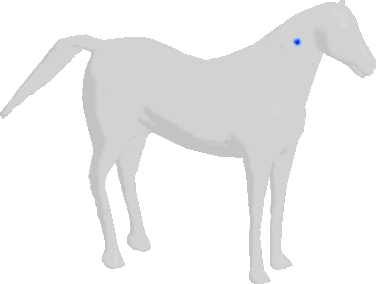} &
	\includegraphics[width=.095\textwidth]{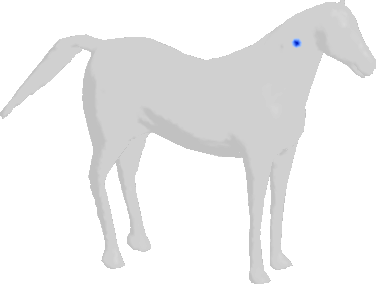} &
	\includegraphics[width=.095\textwidth]{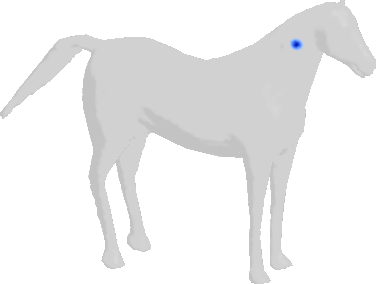} &
	\includegraphics[width=.095\textwidth]{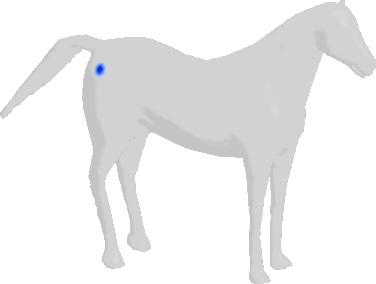} &
	\includegraphics[width=.095\textwidth]{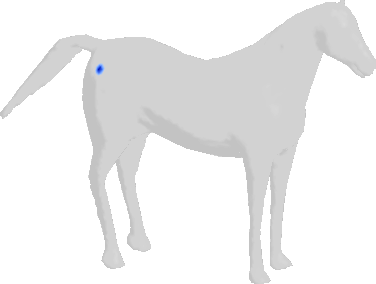} &
	\includegraphics[width=.095\textwidth]{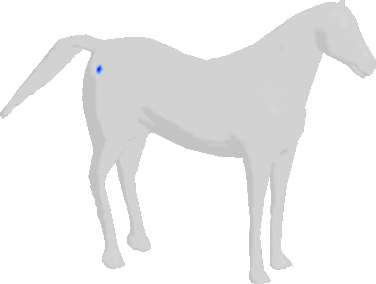} &
	\includegraphics[width=.095\textwidth]{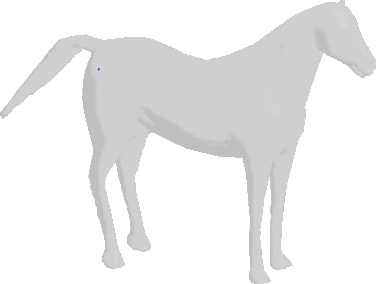} \\

	\includegraphics[width=.095\textwidth]{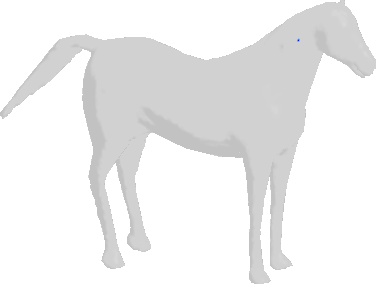} &
	\includegraphics[width=.095\textwidth]{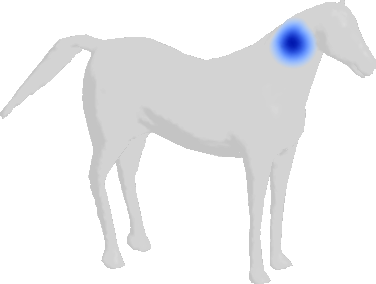} &
	\includegraphics[width=.095\textwidth]{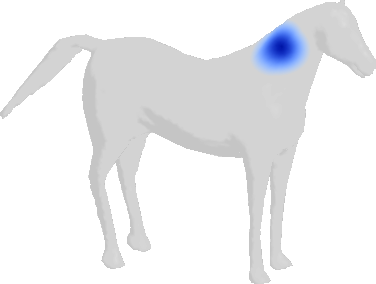} &
	\includegraphics[width=.095\textwidth]{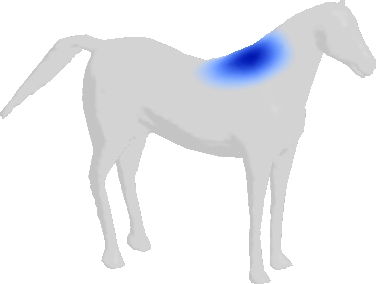} &
	\includegraphics[width=.095\textwidth]{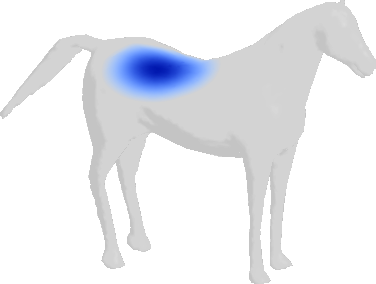} &
	\includegraphics[width=.095\textwidth]{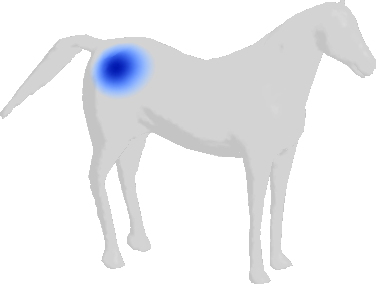} &
	\includegraphics[width=.095\textwidth]{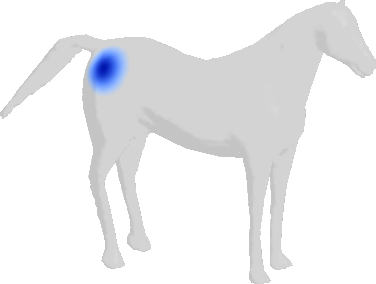} &
	\includegraphics[width=.095\textwidth]{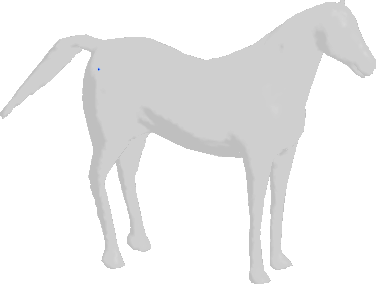} \\
	
	\hline 
	
	$t=0$ & $t =\nicefrac{1}{7}$ & $t =\nicefrac{2}{7}$ & $t =\nicefrac{3}{7}$ & $t =\nicefrac{4}{7}$ & $t =\nicefrac{5}{7}$ & $t =\nicefrac{6}{7}$ & $t=1$
\end{tabular}
\caption{Constant-speed interpolation. Delta distributions on a horse mesh are interpolated. Top row: our method, calculated with $\alpha = 0.01$; middle row: method of Solomon et al.~\protect\shortcite{Solomon2015}, calculated with entropy bounded by that of endpoint distributions; bottom row: method of Solomon et al.~\protect\shortcite{Solomon2015}, calculated with no entropy bound. For the middle row, the motion is even more concentrated in the middle frames. As seen in the bottom row, exclusion of the entropy bound helps somewhat, but the result still is mostly stationary, save for the middle frames.}
\label{figure_horsezip}
\end{figure*}

\section{Applications and extensions}

\subsection{Harmonic mappings}

\begin{figure}
\begin{center}
\begin{tikzpicture}[scale = 1]

% \node[inner sep=0pt] at (0,0) {\includegraphics[width=.2\textwidth]{harmonic_mesh.png}};

% \draw (0,-1.5) node{(a)} ;

\node[inner sep=0pt] at (-2.5,-3.5) {\includegraphics[width=.18\textwidth]{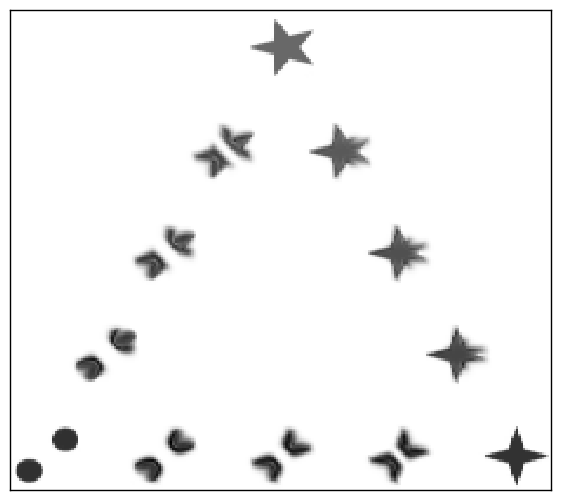}};

\draw (-2.5,-5.25) node{(a)} ;

\node[inner sep=0pt] at (2.5,-3.5) {\includegraphics[width=.18\textwidth]{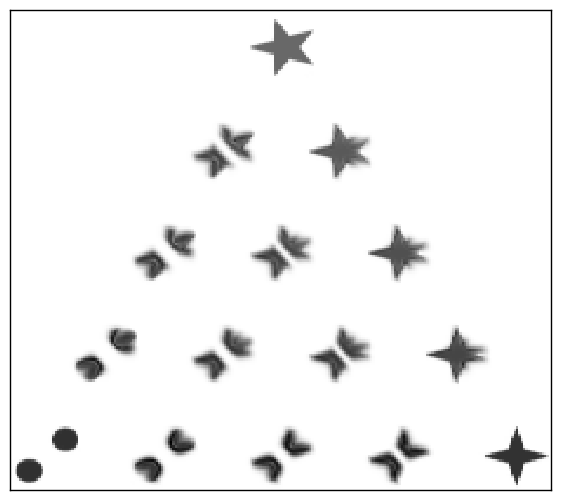}};

\draw (2.5,-5.25) node{(b)} ;

% Arrow
\draw[->, thick = 2pt] (-0.7,-3.5) -- (0.7,-3.5) ;
\draw (0,-3.4) node[above]{Harmonic};
\draw (0,-3.5) node[below]{interpolation};

\end{tikzpicture}
\caption{(a) Data: boundary conditions, i.e. value of the probability distributions at the boundary vertexes of a regular triangulation of an equilateral triangle. (b) Result: interpolation over the interior of the triangular mesh obtained by minimizing the Dirichlet energy with the boundary constraints. The mesh of the source domain $\Omega$ is a regular triangulation of a triangle with $15$ vertexes per side but for clarity reasons we display the value of the probability distributions only a subset of the set of vertexes. The target domain $\M$ is a flat square triangulated with $30$ vertexes by side.}
\label{figure_harmonic}
\end{center}
\end{figure}

As $\P(\M)$ can be viewed as a Riemannian manifold of infinite dimension, one can consider not only geodesics valued in this space, but also harmonic mappings. That is, we consider a domain $\Omega$ and a function $\mu : x \in \Omega \to \mu^x$ which takes fixed values on $\dr \Omega$ the boundary of $\Omega$ \review{and minimizes} the Dirichlet energy
\begin{equation}
\label{equation_Dirichlet_energy_continuous}
\mathrm{Dir}(\mu) \eqdef \frac{1}{2} \int_\Omega \| \nabla_\Omega \mu^x \|^2_{T_{\mu^x}\P(\M)} \dint x,
\end{equation}
where the norm $\| \ \|_{\mu^x}$ is defined in \eqref{equation_norm_continuous}. Such harmonic mappings have been introduced under the name \emph{soft maps} by \review{Solomon et al.~\shortcite{solomon2012soft,Solomon2013}} for the purpose of surface mapping; one can also find a formal definition and theoretical analysis in \cite{Lavenant2017,Lu2017}.

As explained by Lavenant~\shortcite{Lavenant2017}, if some boundary conditions $\bar{\mu} : \dr \Omega \to \P(\M)$ are given, the Dirichlet problem consists in minimizing the Dirichlet energy \eqref{equation_Dirichlet_energy_continuous} of $\mu$ under the constraint that $\mu = \bar{\mu}$ on $\dr \Omega$.  Specifically, it is a convex problem whose dual reads
\begin{equation}
\label{equation_harmonic_dual_continuous}
\begin{cases}
\max_\varphi & \int_{\dr \Omega} \left( \int_\M \varphi(x,\cdot) \cdot n_\Omega(x) \dint \bar{\mu}^x \right) \dint x \\
& \text{s.t. }
\nabla_\Omega \cdot \varphi + \frac{1}{2} \| \nabla_\M \varphi \|^2 \leqslant 0 \text{ on } \Omega \times \M,
\end{cases}
\end{equation}
where %the unknown
$\varphi$ is defined on $\Omega \times M$ and valued in $T \Omega$ (i.e.\ for a point $(x,y) \in \Omega \times M$, one has $\varphi(x,y) \in T_x \Omega$), $n_\Omega(x)$ is the outward normal to $\Omega$, and $\dint x$ is the integration on $\dr \Omega$ w.r.t.\ the surface measure. In the case where $\Omega$ is a segment, the dual problem~\eqref{equation_BB_continuous_dual} for the geodesics is recovered.

To discretize \eqref{equation_harmonic_dual_continuous}, we use the same strategy as for the geodesic problem. We assume that \review{we have} $S_\Omega = (V_\Omega,E_\Omega,T_\Omega)$ a triangulation of the surface $\Omega$. The discrete unknown $\varphi$ maps every element of $T_\Omega \times V$ onto a vector in $\R^3$ (thought as the tangent space of $S_\Omega$). The divergence $\nabla_\Omega \cdot \varphi$ is replaced by its discrete counterpart which lives on $V_\Omega \times V$. On the other hand, $\nabla_\M \varphi$ is naturally seen as a vector in $\R^3$ for each pair of triangles in $T_\Omega \times T$. We apply the same idea: For the term $\| \nabla_\M \varphi \|^2$, first square and then average (weighting by the area of the triangles) to put it on the grid $V_\Omega \times V$. 
Once we have a fully-discrete problem, we build an augmented Lagragian and use ADMM: The solution $\mu$ is the Lagrange multiplier associated to the constraint $A = \nabla_\Omega \cdot \varphi$.

In Figure \ref{figure_harmonic}, we show an example where $S_\Omega$ is a triangulation of an equilateral triangle and $S$ the triangulation of a flat square. On the corners of the triangle we put some distributions, and on the side, as part of the boundary data, we have chosen to put the geodesic in the Wasserstein space between the distributions on the corners. This choice is arbitrary, we could have chosen other configurations on the edges of the boundary of the triangle $\Omega$. This picture resembles the barycentric interpolation \cite{cuturi2014fast,benamou2015iterative,Solomon2015}, though no theoretical evidence indicates that harmonic and barycentric interpolation coincide.

%\subsection{Color interpolation with the Lab geometry}
%
%Optimal transport has be used to provide color interpolation: the idea is to transport the density of the color histogram, and then to lift this transformation in the picture. The Lab color space is known to be encompass the geometry of the color vision. We have performed the same color interpolation experiments, but where the space of color is endowed with this natural geometry, and where we have use the optimal transport in this geometry. \hugo{Have to choose a space of dimension 2 if we want to use discrete surfaces}.

\subsection{Gradient flows in the Wasserstein space}
\label{subsection_gradient_flows}

\begin{figure*}
\begin{center}

\begin{tabular}{cccccccc}

\includegraphics[width=.05\textwidth]{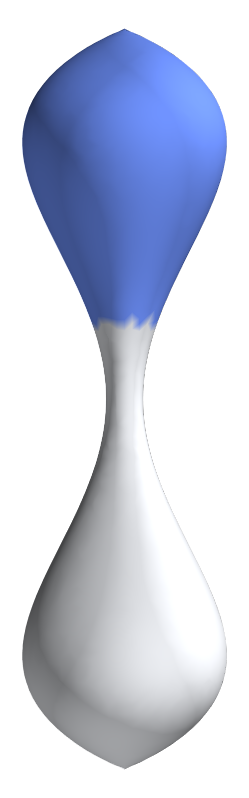} &
\includegraphics[width=.05\textwidth]{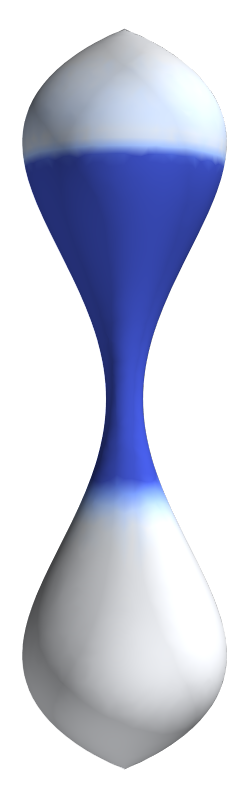} &
\includegraphics[width=.05\textwidth]{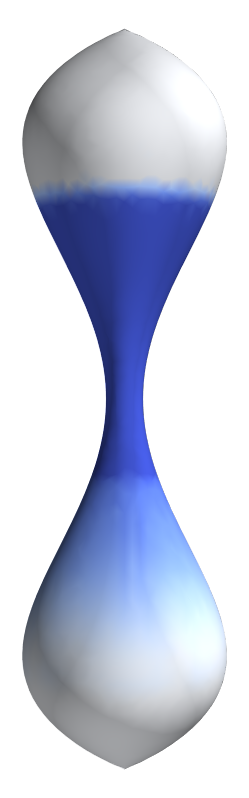} &
\includegraphics[width=.05\textwidth]{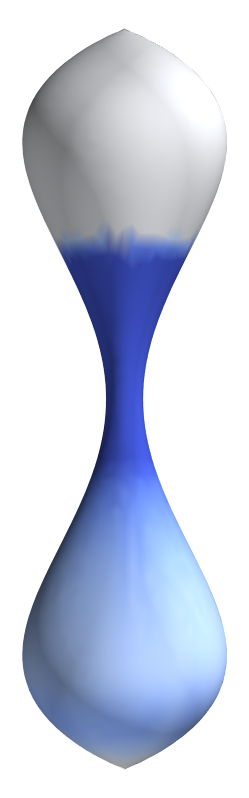} &
\includegraphics[width=.05\textwidth]{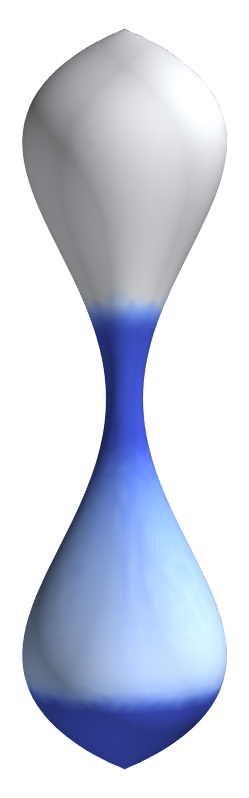} &
\includegraphics[width=.05\textwidth]{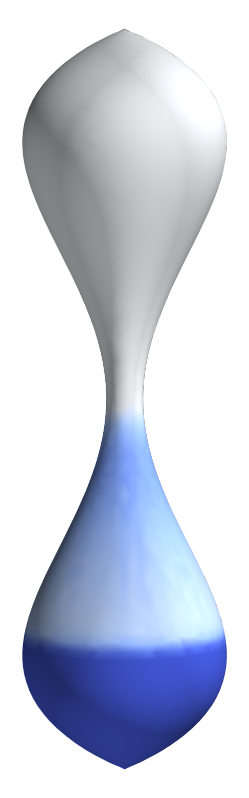} &
\includegraphics[width=.05\textwidth]{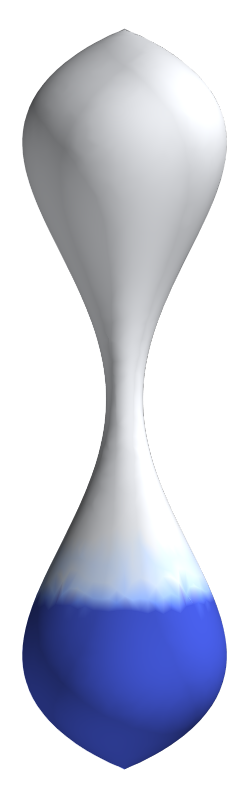} &
\includegraphics[width=.05\textwidth]{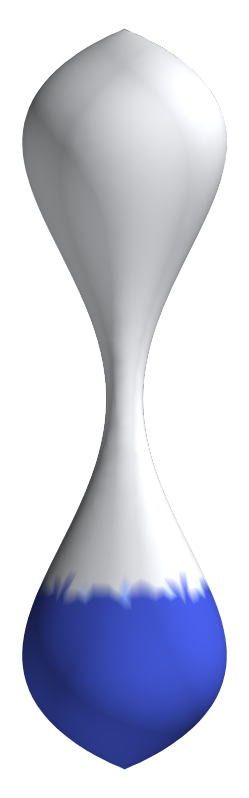} \\ 

\includegraphics[width=.1\textwidth]{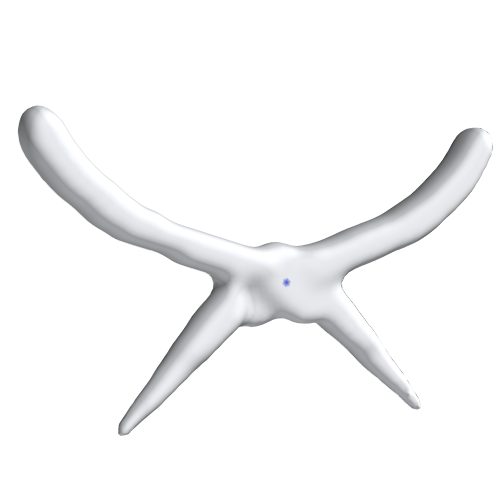} &
\includegraphics[width=.1\textwidth]{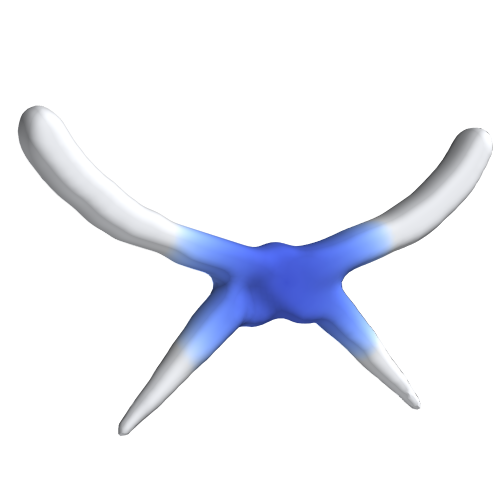} &
\includegraphics[width=.1\textwidth]{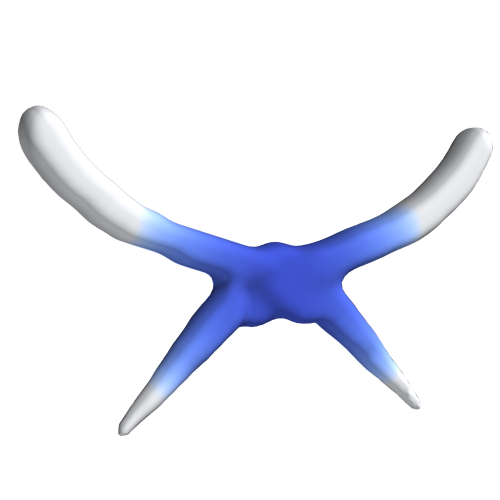} &
\includegraphics[width=.1\textwidth]{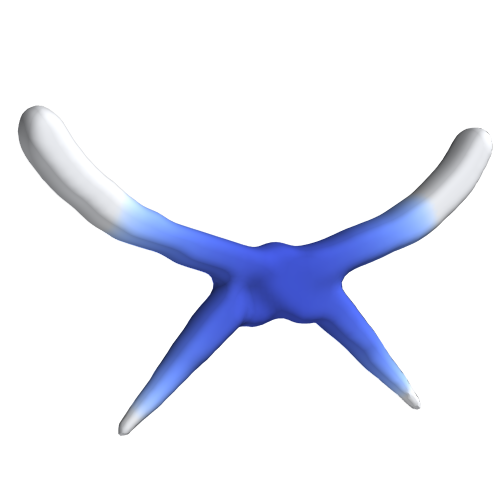} &
\includegraphics[width=.1\textwidth]{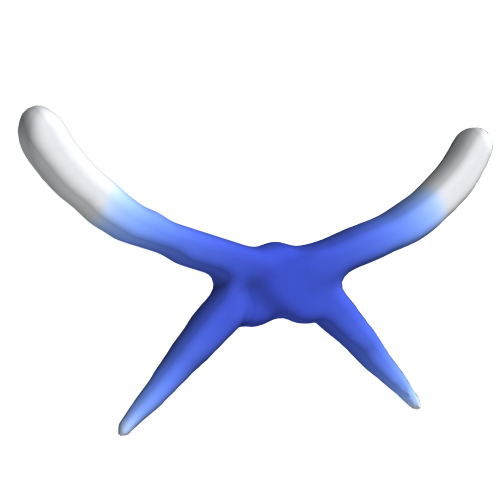} &
\includegraphics[width=.1\textwidth]{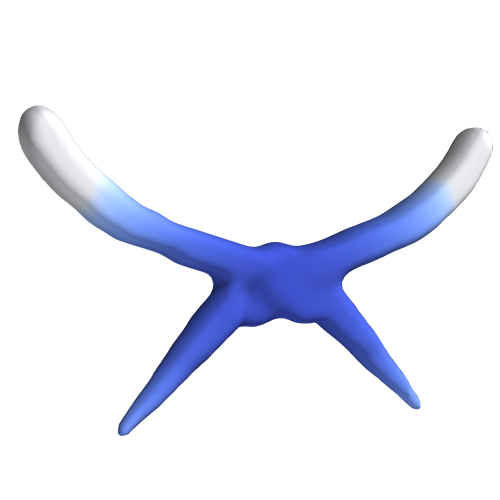} &
\includegraphics[width=.1\textwidth]{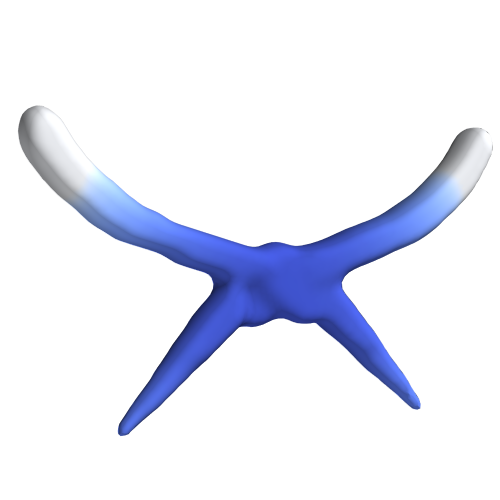} &
\includegraphics[width=.1\textwidth]{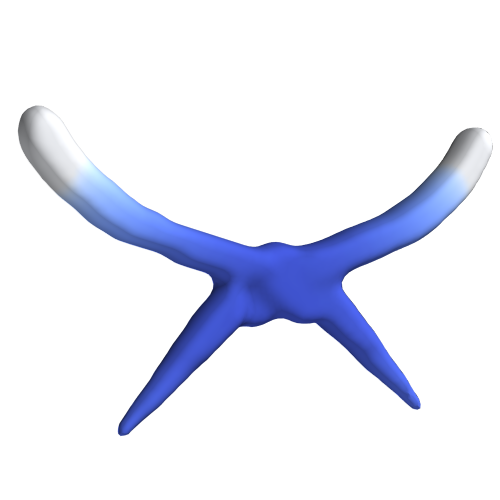} \\

\hline 

$t=0$ & $t=1$ & $t=2$ & $t=3$ & $t=4$ & $t=5$ & $t=6$ & $t=7$ 

\end{tabular}
\end{center}
\caption{First row: gradient flow in \review{discrete} Wasserstein space of the functional \eqref{equation_gradient_flow_crowd}.The potential is chosen to be $W(x,y,z) = z$: the mass \review{flows} down and then \review{saturates} because of the density constraint. \review{Congestion can be observed} as the mass goes through the thinner part. Second row: gradient flow of the internal energy \eqref{equation_gradient_flow_porous_medium} starting from a delta function on the \review{surface} of the plier. The colors do not correspond to the same scale for different times; otherwise nothing would be seen for $t>0$ \review{thanks to} the \review{peaked} distribution at $t=0$. Notice the finite speed of propagation of the density.}
\label{figure_jko}
\end{figure*}

The seminal work of Jordan et al.~\shortcite{Jordan1998}, demonstrates that by considering the gradient flow
\begin{equation}
\label{equation_gradient_flow_continuous}
\dr_t \mu^t = - \nabla_{\P(M)} F(\mu^t),
\end{equation}
where $\nabla_{\P(M)}$ is the gradient w.r.t.\ the scalar product defined by \eqref{equation_elliptic} and \eqref{equation_norm_continuous}, one can recover several well-known PDE (\review{Fokker--Planck}, porous medium equations, aggregation diffusion equations) by an appropriate choice of the functional $F : \P(M) \to \R$ \cite{Ambrosio2008, Santambrogio2015}. Moreover, a natural implicit discretization comes with this equation: The idea is to take a time step $s > 0$ and to define $\mu^{k s}$ for $k = 0,1, \ldots$ recursively in the following way:
\begin{equation}
\label{equation_JKO}
\mu^{(k+1) s} \eqdef \argmin_\mu \left[\frac{W_2^2(\mu, \mu^{k s})}{2 s} + F(\mu)\right].
\end{equation}
The scheme above, defined in arbitrary metric spaces, is refered to as \emph{minimizing movement scheme}, in the framework of optimal transport it is sometimes known as a \emph{JKO integrator}, after the work of Jordan, Kinderlehrer and Otto \shortcite{Jordan1998}.

In the case of a discrete surface $S$ with the Riemannian structure on $\P(S)$, \eqref{equation_gradient_flow_continuous} makes sense and can be written
\begin{equation}
\label{equation_gradient_flow_discrete}
\dr_t \mu^t = - (P_{\mu^t})^{-1} \nabla F (\mu^t)
\end{equation}
where $P_{\mu}$ is the metric tensor defined in \eqref{equation_metric_tensor_discrete} and $\nabla F$ is the usual gradient of $F$ as a function from $\R^{|V|}$ to $\R$. We can still define an iterative scheme to compute solutions of \eqref{equation_gradient_flow_discrete} by solving \eqref{equation_JKO}, where $W_2$ has been replaced by the discrete Wasserstein distance $W_d$. As long as $F : \P(S) \to \R$ is convex, the discrete version of \eqref{equation_JKO} can be tackled by the same augmented Lagrangian and ADMM iterations at the price of introducing an additional variable \cite{Benamou2016}.

We cannot recover cotangent Laplacian heat flow as a gradient flow for the Wasserstein distance $W_d$; to get such a result, one would need to define $\hat{\mu}$ by some nonlinear averaging process rather than \eqref{equation_definition_averaging_mu}. Such a choice would increase the complexity of computing geodesics, and it would be likely to introduce more diffusion as in \cite{Maas2011,Erbar2017}.

\review{The advantages of our numerical method are that positivity is automatic and that mass is preserved.} Moreover, as we expect the difference between two solutions $\mu^{k s}$ and $\mu^{(k+1) s}$ to be very small, we do not need a large number of discretization points in time $N$ (we chose $N=5$ in practice) for the computation of the discrete Wasserstein distance.

We \review{apply} our model to two different cases, \review{illustrated in} Figure \ref{figure_jko}. The \review{first} corresponds to
\begin{equation}
\label{equation_gradient_flow_crowd}
F(\mu) = \begin{cases}
\sum_{v \in V} |v| W_v \mu_v & \text{if } \mu_v \leqslant \mu^\star \ \forall v \in V, \\
+ \infty & \text{else},
\end{cases}
\end{equation}
where $W$ contains one value per vertex. This choice of $F$ \review{yields} a crowd motion model \cite{Maury2010,Santambrogio2018}: The probability distribution wants to flow to the areas where $W$ is low, but at the same time its density is constrained to stay below a threshold $\mu^\star$. This model can only be formulated in terms of a gradient flow \review{in} Wasserstein space and not as an evolution PDE, \review{justifying} the use of \review{a} JKO integrator. %\justin{be sure to define JKO acronym somewhere}.
The \review{second} corresponds to
\begin{equation}
\label{equation_gradient_flow_porous_medium}
F(\mu) = \frac{1}{m-1}
\sum_{v \in V} |v|  (\mu_v)^{m}.
\end{equation}
\review{In} the case of a continuous manifold $\M$, \review{this choice of $F$ yields} the porous medium equation $\dr_t \mu = \Delta(\mu^m)$ \cite{Vazquez2007}. With the scheme \eqref{equation_JKO}, we do not recover a discrete cotangent Laplacian, but the computed solution still exhibits features of the porous medium, like a finite speed of propagation and a convergence to a uniform probability distribution.

% \begin{equation}
% \begin{cases}
% \min_{\bar{\mu}} &  \max_{\varphi} F(\bar{\mu}) + \frac{1}{2\tau} \left(  \sum_{v \in V} |v| \varphi^1_v \bar{\mu}_v - \sum_{v \in V} |v| \varphi^0_v \mu^{k\tau}_v \right)  \\
% &\text{s.t. }  (\Delta \varphi)^s_v + \displaystyle{ \frac{1}{2} \sum_{i \in \{ -1,1 \}} \frac{1}{2} \frac{\sum_{t \ : \  v \in t  } |t| \| (G \varphi)^{s + i \Delta s/2}_t \|^2 }{  3 |v|}}   \leqslant 0 \\ & \text{ for all } (s,v) \in \Gtimec \times V.
% \end{cases}
% \end{equation}
% As long as $F$ is a convex function, the equation above can be still reformulated as a convex problem and the same ADMM algorithm can be used at the price of introducing an additional variable \cite{Benamou2016},

% !TeX root = main.tex

\section{Discussion and conclusion}

Although techniques using entropic regularization or semi-discrete optimal transport can interpolate between distributions on a discrete surface, they do not provide a Riemannian structure and are subject to practical limitations that restrict the scenarios to which they can be applied.  Using an intrinsic formulation of dynamical transport, we can realize the theoretical and practical potential of optimal transport on discrete domains enabled by the Riemannian structure on the space of probability distributions, the so-called Otto calculus.  Our technique can be phrased in familiar language from discrete differential geometry and is implementable using standard tools in that domain. The key ingredients, namely first- and second-order operators in geometry processing (gradient, divergence, Laplacian) as well as SOCP optimization, remain in the realm of what is already widely used. 

We have demonstrated the power of our \review{model} by showing how it can handle a variety of geometries and peaked distributions, while introducing little diffusion. Mass may concentrate to yield a visually inelegant result, but this behavior is at the core of  optimal transport theory and expected: No price is paid for mass congestion, and hence any concentration of geodesics will result in a concentration of mass. Nevertheless, as we have shown, one can easily modify the optimization problem to penalize congested densities, leading to smoother interpolants with a controllable level of diffusion. Unlike entropically-regularized transport, however, our optimization problem does not degenerate as the coefficient in front of the regularizer vanishes.

Beyond evaluation of transport distances, \review{our} framework \review{extends} to support other tasks involving transport \review{terms}.  We can reliably compute harmonic mappings valued in this discrete Wasserstein space, and the JKO integrator based on our discrete Wasserstein distance exhibits expected qualitative behavior.  % the ``Dirichlet energy'' paper and the entropy/heat paper both compute harmonic mappings into discrete Wasserstein space, so it's not really the first

\review{The main drawback of our approach remains its scalability. The bottleneck of the computations is the solution of a linear system whose number of unknowns is the product of the number of discretization points in time and the size of the mesh.  This is an extremely structured linear system on a product manifold, for which specialized matrix inversion techniques may exist. In any event, with the current bottleneck our method can handle meshes with few thousand vertices but is not currently practical for larger meshes.}

As one of the first structure-preserving discretizations of transport on meshes, our work also suggests several exciting avenues for future research.  Many theoretical properties of our discrete Wasserstein distance remain to be explored.  For instance, while we have shown that our formulation is a true Riemannian distance, one could verify the extent to which a wealth of other theoretical properties of transport are preserved.  Convergence of our transport over meshes to the true transport in the limit of mesh refinement also remains an open problem for our techniques and others in a similar class. From a practical perspective, a natural next step is to accelerate the optimization procedure as much as possible; a faster solver for the convex optimization problem would clearly benefit our method.

\review{\section*{Acknowledgements}

J.\ Solomon acknowledges the generous support of Army Research Office grant W911NF-12-R-0011 ("Smooth Modeling of Flows on Graphs"), of National Science Foundation grant IIS-1838071 (``BIGDATA:F: Statistical and Computational Optimal Transport for Geometric Data Analysis''), from the MIT Research Support Committee, from an Amazon Research Award, from the MIT--IBM Watson AI Laboratory, and from the Skoltech--MIT Next Generation Program. Most of this work was done during a visit of H.\ Lavenant to MIT; the hospitality of CSAIL and MIT is warmly acknowledged. The authors thank the reviewers for their helpful feedback.} 

\bibliographystyle{ACM-Reference-Format}
\bibliography{bibliography}

%%% -*-BibTeX-*-
%%% Do NOT edit. File created by BibTeX with style
%%% ACM-Reference-Format-Journals [18-Jan-2012].

\begin{thebibliography}{66}

%%% ====================================================================
%%% NOTE TO THE USER: you can override these defaults by providing
%%% customized versions of any of these macros before the \bibliography
%%% command.  Each of them MUST provide its own final punctuation,
%%% except for \shownote{}, \showDOI{}, and \showURL{}.  The latter two
%%% do not use final punctuation, in order to avoid confusing it with
%%% the Web address.
%%%
%%% To suppress output of a particular field, define its macro to expand
%%% to an empty string, or better, \unskip, like this:
%%%
%%% \newcommand{\showDOI}[1]{\unskip}   % LaTeX syntax
%%%
%%% \def \showDOI #1{\unskip}           % plain TeX syntax
%%%
%%% ====================================================================

\ifx \showCODEN    \undefined \def \showCODEN     #1{\unskip}     \fi
\ifx \showDOI      \undefined \def \showDOI       #1{#1}\fi
\ifx \showISBNx    \undefined \def \showISBNx     #1{\unskip}     \fi
\ifx \showISBNxiii \undefined \def \showISBNxiii  #1{\unskip}     \fi
\ifx \showISSN     \undefined \def \showISSN      #1{\unskip}     \fi
\ifx \showLCCN     \undefined \def \showLCCN      #1{\unskip}     \fi
\ifx \shownote     \undefined \def \shownote      #1{#1}          \fi
\ifx \showarticletitle \undefined \def \showarticletitle #1{#1}   \fi
\ifx \showURL      \undefined \def \showURL       {\relax}        \fi
% The following commands are used for tagged output and should be
% invisible to TeX
\providecommand\bibfield[2]{#2}
\providecommand\bibinfo[2]{#2}
\providecommand\natexlab[1]{#1}
\providecommand\showeprint[2][]{arXiv:#2}

\bibitem[\protect\citeauthoryear{Ambrosio, Gigli, and Savar{\'e}}{Ambrosio
  et~al\mbox{.}}{2008}]%
        {Ambrosio2008}
\bibfield{author}{\bibinfo{person}{Luigi Ambrosio}, \bibinfo{person}{Nicola
  Gigli}, {and} \bibinfo{person}{Giuseppe Savar{\'e}}.}
  \bibinfo{year}{2008}\natexlab{}.
\newblock \bibinfo{booktitle}{\emph{Gradient Flows: In Metric Spaces and in the
  Space of Probability Measures}}.
\newblock \bibinfo{publisher}{Springer Science \& Business Media}.
\newblock


\bibitem[\protect\citeauthoryear{Aurenhammer, Hoffmann, and Aronov}{Aurenhammer
  et~al\mbox{.}}{1998}]%
        {Aurenhammer1998}
\bibfield{author}{\bibinfo{person}{Franz Aurenhammer},
  \bibinfo{person}{Friedrich Hoffmann}, {and} \bibinfo{person}{Boris Aronov}.}
  \bibinfo{year}{1998}\natexlab{}.
\newblock \showarticletitle{Minkowski-type theorems and least-squares
  clustering}.
\newblock \bibinfo{journal}{\emph{Algorithmica}} \bibinfo{volume}{20},
  \bibinfo{number}{1} (\bibinfo{year}{1998}), \bibinfo{pages}{61--76}.
\newblock


\bibitem[\protect\citeauthoryear{Azencot, Vantzos, and Ben-Chen}{Azencot
  et~al\mbox{.}}{2016}]%
        {Azencot2016}
\bibfield{author}{\bibinfo{person}{Omri Azencot}, \bibinfo{person}{Orestis
  Vantzos}, {and} \bibinfo{person}{Mirela Ben-Chen}.}
  \bibinfo{year}{2016}\natexlab{}.
\newblock \showarticletitle{Advection-based function matching on surfaces}. In
  \bibinfo{booktitle}{\emph{Computer Graphics Forum}},
  Vol.~\bibinfo{volume}{35}. Wiley Online Library, \bibinfo{pages}{55--64}.
\newblock


\bibitem[\protect\citeauthoryear{Benamou and Brenier}{Benamou and
  Brenier}{2000}]%
        {Benamou2000}
\bibfield{author}{\bibinfo{person}{Jean-David Benamou} {and}
  \bibinfo{person}{Yann Brenier}.} \bibinfo{year}{2000}\natexlab{}.
\newblock \showarticletitle{A computational fluid mechanics solution to the
  {M}onge-{K}antorovich mass transfer problem}.
\newblock \bibinfo{journal}{\emph{Numer. Math.}} \bibinfo{volume}{84},
  \bibinfo{number}{3} (\bibinfo{year}{2000}), \bibinfo{pages}{375--393}.
\newblock


\bibitem[\protect\citeauthoryear{Benamou and Carlier}{Benamou and
  Carlier}{2015}]%
        {Benamou2015}
\bibfield{author}{\bibinfo{person}{Jean-David Benamou} {and}
  \bibinfo{person}{Guillaume Carlier}.} \bibinfo{year}{2015}\natexlab{}.
\newblock \showarticletitle{Augmented {L}agrangian methods for transport
  optimization, mean field games and degenerate elliptic equations}.
\newblock \bibinfo{journal}{\emph{Journal of Optimization Theory and
  Applications}} \bibinfo{volume}{167}, \bibinfo{number}{1}
  (\bibinfo{year}{2015}), \bibinfo{pages}{1--26}.
\newblock


\bibitem[\protect\citeauthoryear{Benamou, Carlier, Cuturi, Nenna, and
  Peyr{\'e}}{Benamou et~al\mbox{.}}{2015}]%
        {benamou2015iterative}
\bibfield{author}{\bibinfo{person}{Jean-David Benamou},
  \bibinfo{person}{Guillaume Carlier}, \bibinfo{person}{Marco Cuturi},
  \bibinfo{person}{Luca Nenna}, {and} \bibinfo{person}{Gabriel Peyr{\'e}}.}
  \bibinfo{year}{2015}\natexlab{}.
\newblock \showarticletitle{Iterative {B}regman projections for regularized
  transportation problems}.
\newblock \bibinfo{journal}{\emph{SIAM Journal on Scientific Computing}}
  \bibinfo{volume}{37}, \bibinfo{number}{2} (\bibinfo{year}{2015}),
  \bibinfo{pages}{A1111--A1138}.
\newblock


\bibitem[\protect\citeauthoryear{Benamou, Carlier, and Laborde}{Benamou
  et~al\mbox{.}}{2016}]%
        {Benamou2016}
\bibfield{author}{\bibinfo{person}{Jean-David Benamou},
  \bibinfo{person}{Guillaume Carlier}, {and} \bibinfo{person}{Maxime Laborde}.}
  \bibinfo{year}{2016}\natexlab{}.
\newblock \showarticletitle{An augmented {L}agrangian approach to {W}asserstein
  gradient flows and applications}.
\newblock \bibinfo{journal}{\emph{ESAIM: Proceedings and Surveys}}
  \bibinfo{volume}{54} (\bibinfo{year}{2016}), \bibinfo{pages}{1--17}.
\newblock


\bibitem[\protect\citeauthoryear{Benamou, Carlier, and Santambrogio}{Benamou
  et~al\mbox{.}}{2017}]%
        {Benamou2017}
\bibfield{author}{\bibinfo{person}{Jean-David Benamou},
  \bibinfo{person}{Guillaume Carlier}, {and} \bibinfo{person}{Filippo
  Santambrogio}.} \bibinfo{year}{2017}\natexlab{}.
\newblock \showarticletitle{Variational mean field games}.
\newblock In \bibinfo{booktitle}{\emph{Active Particles, Volume 1}}.
  \bibinfo{publisher}{Springer}, \bibinfo{pages}{141--171}.
\newblock


\bibitem[\protect\citeauthoryear{Bonneel, Van De~Panne, Paris, and
  Heidrich}{Bonneel et~al\mbox{.}}{2011}]%
        {bonneel2011displacement}
\bibfield{author}{\bibinfo{person}{Nicolas Bonneel}, \bibinfo{person}{Michiel
  Van De~Panne}, \bibinfo{person}{Sylvain Paris}, {and}
  \bibinfo{person}{Wolfgang Heidrich}.} \bibinfo{year}{2011}\natexlab{}.
\newblock \showarticletitle{Displacement interpolation using {L}agrangian mass
  transport}. In \bibinfo{booktitle}{\emph{ACM Transactions on Graphics
  (TOG)}}, Vol.~\bibinfo{volume}{30}. ACM, \bibinfo{pages}{158}.
\newblock


\bibitem[\protect\citeauthoryear{Boyd, Parikh, Chu, Peleato, Eckstein,
  et~al\mbox{.}}{Boyd et~al\mbox{.}}{2011}]%
        {Boyd2011}
\bibfield{author}{\bibinfo{person}{Stephen Boyd}, \bibinfo{person}{Neal
  Parikh}, \bibinfo{person}{Eric Chu}, \bibinfo{person}{Borja Peleato},
  \bibinfo{person}{Jonathan Eckstein}, {et~al\mbox{.}}}
  \bibinfo{year}{2011}\natexlab{}.
\newblock \showarticletitle{Distributed optimization and statistical learning
  via the alternating direction method of multipliers}.
\newblock \bibinfo{journal}{\emph{Foundations and Trends in Machine Learning}}
  \bibinfo{volume}{3}, \bibinfo{number}{1} (\bibinfo{year}{2011}),
  \bibinfo{pages}{1--122}.
\newblock


\bibitem[\protect\citeauthoryear{Boyd and Vandenberghe}{Boyd and
  Vandenberghe}{2004}]%
        {Boyd2004}
\bibfield{author}{\bibinfo{person}{Stephen Boyd} {and} \bibinfo{person}{Lieven
  Vandenberghe}.} \bibinfo{year}{2004}\natexlab{}.
\newblock \bibinfo{booktitle}{\emph{Convex Optimization}}.
\newblock \bibinfo{publisher}{Cambridge University Press}.
\newblock


\bibitem[\protect\citeauthoryear{Brenier}{Brenier}{1991}]%
        {Brenier1991}
\bibfield{author}{\bibinfo{person}{Yann Brenier}.}
  \bibinfo{year}{1991}\natexlab{}.
\newblock \showarticletitle{Polar factorization and monotone rearrangement of
  vector-valued functions}.
\newblock \bibinfo{journal}{\emph{Communications on Pure and Applied
  Mathematics}} \bibinfo{volume}{44}, \bibinfo{number}{4}
  (\bibinfo{year}{1991}), \bibinfo{pages}{375--417}.
\newblock


\bibitem[\protect\citeauthoryear{Brenner and Scott}{Brenner and Scott}{2007}]%
        {brenner2007mathematical}
\bibfield{author}{\bibinfo{person}{Susanne Brenner} {and}
  \bibinfo{person}{Ridgway Scott}.} \bibinfo{year}{2007}\natexlab{}.
\newblock \bibinfo{booktitle}{\emph{The Mathematical Theory of Finite Element
  Methods}}. Vol.~\bibinfo{volume}{15}.
\newblock \bibinfo{publisher}{Springer Science \& Business Media}.
\newblock


\bibitem[\protect\citeauthoryear{Chow, Dieci, Li, and Zhou}{Chow
  et~al\mbox{.}}{2016}]%
        {Chow2016}
\bibfield{author}{\bibinfo{person}{Shui-Nee Chow}, \bibinfo{person}{Luca
  Dieci}, \bibinfo{person}{Wuchen Li}, {and} \bibinfo{person}{Haomin Zhou}.}
  \bibinfo{year}{2016}\natexlab{}.
\newblock \showarticletitle{Entropy dissipation semi-discretization schemes for
  {F}okker-{P}lanck equations}.
\newblock \bibinfo{journal}{\emph{arXiv:1608.02628}} (\bibinfo{year}{2016}).
\newblock


\bibitem[\protect\citeauthoryear{Claici, Chien, and Solomon}{Claici
  et~al\mbox{.}}{2018}]%
        {Claici2018}
\bibfield{author}{\bibinfo{person}{Sebastian Claici}, \bibinfo{person}{Edward
  Chien}, {and} \bibinfo{person}{Justin Solomon}.}
  \bibinfo{year}{2018}\natexlab{}.
\newblock \showarticletitle{Stochastic {W}asserstein barycenters}. In
  \bibinfo{booktitle}{\emph{Proceedings of the 35th International Conference on
  Machine Learning, ICML 2018 (to appear)}}.
\newblock


\bibitem[\protect\citeauthoryear{Cuturi}{Cuturi}{2013}]%
        {Cuturi2013}
\bibfield{author}{\bibinfo{person}{Marco Cuturi}.}
  \bibinfo{year}{2013}\natexlab{}.
\newblock \showarticletitle{Sinkhorn distances: Lightspeed computation of
  optimal transport}. In \bibinfo{booktitle}{\emph{Advances in Neural
  Information Processing Systems, NIPS 2013}}. \bibinfo{pages}{2292--2300}.
\newblock


\bibitem[\protect\citeauthoryear{Cuturi and Doucet}{Cuturi and Doucet}{2014}]%
        {cuturi2014fast}
\bibfield{author}{\bibinfo{person}{Marco Cuturi} {and} \bibinfo{person}{Arnaud
  Doucet}.} \bibinfo{year}{2014}\natexlab{}.
\newblock \showarticletitle{Fast computation of {W}asserstein barycenters}. In
  \bibinfo{booktitle}{\emph{International Conference on Machine Learning}}.
  \bibinfo{pages}{685--693}.
\newblock


\bibitem[\protect\citeauthoryear{De~Goes, Breeden, Ostromoukhov, and
  Desbrun}{De~Goes et~al\mbox{.}}{2012}]%
        {de2012blue}
\bibfield{author}{\bibinfo{person}{Fernando De~Goes},
  \bibinfo{person}{Katherine Breeden}, \bibinfo{person}{Victor Ostromoukhov},
  {and} \bibinfo{person}{Mathieu Desbrun}.} \bibinfo{year}{2012}\natexlab{}.
\newblock \showarticletitle{Blue noise through optimal transport}.
\newblock \bibinfo{journal}{\emph{ACM Transactions on Graphics (TOG)}}
  \bibinfo{volume}{31}, \bibinfo{number}{6} (\bibinfo{year}{2012}),
  \bibinfo{pages}{171}.
\newblock


\bibitem[\protect\citeauthoryear{de~Goes, Cohen-Steiner, Alliez, and
  Desbrun}{de~Goes et~al\mbox{.}}{2011}]%
        {DeGoes2011}
\bibfield{author}{\bibinfo{person}{Fernando de Goes}, \bibinfo{person}{David
  Cohen-Steiner}, \bibinfo{person}{Pierre Alliez}, {and}
  \bibinfo{person}{Mathieu Desbrun}.} \bibinfo{year}{2011}\natexlab{}.
\newblock \showarticletitle{An optimal transport approach to robust
  reconstruction and simplification of 2d shapes}. In
  \bibinfo{booktitle}{\emph{Computer Graphics Forum}},
  Vol.~\bibinfo{volume}{30}. Wiley Online Library, \bibinfo{pages}{1593--1602}.
\newblock


\bibitem[\protect\citeauthoryear{de~Goes, Desbrun, and Tong}{de~Goes
  et~al\mbox{.}}{2015a}]%
        {DeGoes2015vector}
\bibfield{author}{\bibinfo{person}{Fernando de Goes}, \bibinfo{person}{Mathieu
  Desbrun}, {and} \bibinfo{person}{Yiying Tong}.}
  \bibinfo{year}{2015}\natexlab{a}.
\newblock \showarticletitle{Vector field processing on triangle meshes}. In
  \bibinfo{booktitle}{\emph{SIGGRAPH Asia 2015 Courses}}. ACM,
  \bibinfo{pages}{17}.
\newblock


\bibitem[\protect\citeauthoryear{de~Goes, Memari, Mullen, and Desbrun}{de~Goes
  et~al\mbox{.}}{2014}]%
        {DeGoes2014}
\bibfield{author}{\bibinfo{person}{Fernando de Goes}, \bibinfo{person}{Pooran
  Memari}, \bibinfo{person}{Patrick Mullen}, {and} \bibinfo{person}{Mathieu
  Desbrun}.} \bibinfo{year}{2014}\natexlab{}.
\newblock \showarticletitle{Weighted Triangulations for Geometry Processing}.
\newblock \bibinfo{journal}{\emph{ACM Trans. Graph.}} \bibinfo{volume}{33},
  \bibinfo{number}{3} (\bibinfo{date}{June} \bibinfo{year}{2014}),
  \bibinfo{pages}{28:1--28:13}.
\newblock


\bibitem[\protect\citeauthoryear{de~Goes, Wallez, Huang, Pavlov, and
  Desbrun}{de~Goes et~al\mbox{.}}{2015b}]%
        {DeGoes2015fluid}
\bibfield{author}{\bibinfo{person}{Fernando de Goes}, \bibinfo{person}{Corentin
  Wallez}, \bibinfo{person}{Jin Huang}, \bibinfo{person}{Dmitry Pavlov}, {and}
  \bibinfo{person}{Mathieu Desbrun}.} \bibinfo{year}{2015}\natexlab{b}.
\newblock \showarticletitle{Power particles: an incompressible fluid solver
  based on power diagrams}.
\newblock \bibinfo{journal}{\emph{ACM Trans. Graph.}} \bibinfo{volume}{34},
  \bibinfo{number}{4} (\bibinfo{year}{2015}), \bibinfo{pages}{50--1}.
\newblock


\bibitem[\protect\citeauthoryear{Digne, Cohen-Steiner, Alliez, De~Goes, and
  Desbrun}{Digne et~al\mbox{.}}{2014}]%
        {Digne2014}
\bibfield{author}{\bibinfo{person}{Julie Digne}, \bibinfo{person}{David
  Cohen-Steiner}, \bibinfo{person}{Pierre Alliez}, \bibinfo{person}{Fernando
  De~Goes}, {and} \bibinfo{person}{Mathieu Desbrun}.}
  \bibinfo{year}{2014}\natexlab{}.
\newblock \showarticletitle{Feature-preserving surface reconstruction and
  simplification from defect-laden point sets}.
\newblock \bibinfo{journal}{\emph{Journal of Mathematical Imaging and Vision}}
  \bibinfo{volume}{48}, \bibinfo{number}{2} (\bibinfo{year}{2014}),
  \bibinfo{pages}{369--382}.
\newblock


\bibitem[\protect\citeauthoryear{Edmonds and Karp}{Edmonds and Karp}{1972}]%
        {edmonds1972theoretical}
\bibfield{author}{\bibinfo{person}{Jack Edmonds} {and}
  \bibinfo{person}{Richard~M Karp}.} \bibinfo{year}{1972}\natexlab{}.
\newblock \showarticletitle{Theoretical improvements in algorithmic efficiency
  for network flow problems}.
\newblock \bibinfo{journal}{\emph{Journal of the ACM (JACM)}}
  \bibinfo{volume}{19}, \bibinfo{number}{2} (\bibinfo{year}{1972}),
  \bibinfo{pages}{248--264}.
\newblock


\bibitem[\protect\citeauthoryear{Erbar, Rumpf, Schmitzer, and Simon}{Erbar
  et~al\mbox{.}}{2017}]%
        {Erbar2017}
\bibfield{author}{\bibinfo{person}{Matthias Erbar}, \bibinfo{person}{Martin
  Rumpf}, \bibinfo{person}{Bernhard Schmitzer}, {and} \bibinfo{person}{Stefan
  Simon}.} \bibinfo{year}{2017}\natexlab{}.
\newblock \showarticletitle{Computation of Optimal Transport on Discrete Metric
  Measure Spaces}.
\newblock \bibinfo{journal}{\emph{arXiv:1707.06859}} (\bibinfo{year}{2017}).
\newblock


\bibitem[\protect\citeauthoryear{Gallou{\"e}t and M{\'e}rigot}{Gallou{\"e}t and
  M{\'e}rigot}{2017}]%
        {Gallouet2017}
\bibfield{author}{\bibinfo{person}{Thomas~O Gallou{\"e}t} {and}
  \bibinfo{person}{Quentin M{\'e}rigot}.} \bibinfo{year}{2017}\natexlab{}.
\newblock \showarticletitle{A Lagrangian scheme {\`a} la Brenier for the
  incompressible Euler equations}.
\newblock \bibinfo{journal}{\emph{Foundations of Computational Mathematics}}
  (\bibinfo{year}{2017}), \bibinfo{pages}{1--31}.
\newblock


\bibitem[\protect\citeauthoryear{Gangbo and McCann}{Gangbo and McCann}{1996}]%
        {Gangbo1996}
\bibfield{author}{\bibinfo{person}{Wilfrid Gangbo} {and}
  \bibinfo{person}{Robert~J McCann}.} \bibinfo{year}{1996}\natexlab{}.
\newblock \showarticletitle{The geometry of optimal transportation}.
\newblock \bibinfo{journal}{\emph{Acta Mathematica}} \bibinfo{volume}{177},
  \bibinfo{number}{2} (\bibinfo{year}{1996}), \bibinfo{pages}{113--161}.
\newblock


\bibitem[\protect\citeauthoryear{Gigli and Maas}{Gigli and Maas}{2013}]%
        {Gigli2013}
\bibfield{author}{\bibinfo{person}{Nicola Gigli} {and} \bibinfo{person}{Jan
  Maas}.} \bibinfo{year}{2013}\natexlab{}.
\newblock \showarticletitle{Gromov--{H}ausdorff convergence of discrete
  transportation metrics}.
\newblock \bibinfo{journal}{\emph{SIAM Journal on Mathematical Analysis}}
  \bibinfo{volume}{45}, \bibinfo{number}{2} (\bibinfo{year}{2013}),
  \bibinfo{pages}{879--899}.
\newblock


\bibitem[\protect\citeauthoryear{Gladbach, Kopfer, and Maas}{Gladbach
  et~al\mbox{.}}{2018}]%
        {Gladbach2018}
\bibfield{author}{\bibinfo{person}{Peter Gladbach}, \bibinfo{person}{Eva
  Kopfer}, {and} \bibinfo{person}{Jan Maas}.} \bibinfo{year}{2018}\natexlab{}.
\newblock \showarticletitle{Scaling limits of discrete optimal transport}.
\newblock \bibinfo{journal}{\emph{arXiv:1809.01092}} (\bibinfo{year}{2018}).
\newblock


\bibitem[\protect\citeauthoryear{Grant and Boyd}{Grant and Boyd}{2008}]%
        {gb08}
\bibfield{author}{\bibinfo{person}{Michael Grant} {and}
  \bibinfo{person}{Stephen Boyd}.} \bibinfo{year}{2008}\natexlab{}.
\newblock \showarticletitle{Graph implementations for nonsmooth convex
  programs}.
\newblock In \bibinfo{booktitle}{\emph{Recent Advances in Learning and
  Control}}, \bibfield{editor}{\bibinfo{person}{V.~Blondel},
  \bibinfo{person}{S.~Boyd}, {and} \bibinfo{person}{H.~Kimura}} (Eds.).
  \bibinfo{publisher}{Springer-Verlag Limited}, \bibinfo{pages}{95--110}.
\newblock


\bibitem[\protect\citeauthoryear{Grant and Boyd}{Grant and Boyd}{2014}]%
        {cvx}
\bibfield{author}{\bibinfo{person}{Michael Grant} {and}
  \bibinfo{person}{Stephen Boyd}.} \bibinfo{year}{2014}\natexlab{}.
\newblock \bibinfo{title}{{CVX}: {M}atlab Software for Disciplined Convex
  Programming, version 2.1}.
\newblock \bibinfo{howpublished}{\url{http://cvxr.com/cvx}}.
\newblock


\bibitem[\protect\citeauthoryear{Guittet}{Guittet}{2003}]%
        {Guittet2003}
\bibfield{author}{\bibinfo{person}{Kevin Guittet}.}
  \bibinfo{year}{2003}\natexlab{}.
\newblock \showarticletitle{On the time-continuous mass transport problem and
  its approximation by augmented {L}agrangian techniques}.
\newblock \bibinfo{journal}{\emph{SIAM J. Numer. Anal.}} \bibinfo{volume}{41},
  \bibinfo{number}{1} (\bibinfo{year}{2003}), \bibinfo{pages}{382--399}.
\newblock


\bibitem[\protect\citeauthoryear{Heeren, Rumpf, Wardetzky, and Wirth}{Heeren
  et~al\mbox{.}}{2012}]%
        {Heeren2012}
\bibfield{author}{\bibinfo{person}{Behrend Heeren}, \bibinfo{person}{Martin
  Rumpf}, \bibinfo{person}{Max Wardetzky}, {and} \bibinfo{person}{Benedikt
  Wirth}.} \bibinfo{year}{2012}\natexlab{}.
\newblock \showarticletitle{Time-Discrete Geodesics in the Space of Shells}. In
  \bibinfo{booktitle}{\emph{Computer Graphics Forum}},
  Vol.~\bibinfo{volume}{31}. Wiley Online Library, \bibinfo{pages}{1755--1764}.
\newblock


\bibitem[\protect\citeauthoryear{Hug, Papadakis, and Maitre}{Hug
  et~al\mbox{.}}{2015}]%
        {Hug2015}
\bibfield{author}{\bibinfo{person}{Romain Hug}, \bibinfo{person}{Nicolas
  Papadakis}, {and} \bibinfo{person}{Emmanuel Maitre}.}
  \bibinfo{year}{2015}\natexlab{}.
\newblock \showarticletitle{On the convergence of augmented {L}agrangian method
  for optimal transport between nonnegative densities}.
\newblock  (\bibinfo{year}{2015}).
\newblock


\bibitem[\protect\citeauthoryear{Jordan, Kinderlehrer, and Otto}{Jordan
  et~al\mbox{.}}{1998}]%
        {Jordan1998}
\bibfield{author}{\bibinfo{person}{Richard Jordan}, \bibinfo{person}{David
  Kinderlehrer}, {and} \bibinfo{person}{Felix Otto}.}
  \bibinfo{year}{1998}\natexlab{}.
\newblock \showarticletitle{The variational formulation of the
  {F}okker--{P}lanck equation}.
\newblock \bibinfo{journal}{\emph{SIAM Journal on Mathematical Analysis}}
  \bibinfo{volume}{29}, \bibinfo{number}{1} (\bibinfo{year}{1998}),
  \bibinfo{pages}{1--17}.
\newblock


\bibitem[\protect\citeauthoryear{Jost}{Jost}{2008}]%
        {Jost2008}
\bibfield{author}{\bibinfo{person}{J{\"u}rgen Jost}.}
  \bibinfo{year}{2008}\natexlab{}.
\newblock \bibinfo{booktitle}{\emph{Riemannian geometry and geometric
  analysis}}. Vol.~\bibinfo{volume}{42005}.
\newblock \bibinfo{publisher}{Springer}.
\newblock


\bibitem[\protect\citeauthoryear{Kantorovich}{Kantorovich}{1942}]%
        {Kantorovich1942}
\bibfield{author}{\bibinfo{person}{Leonid~V Kantorovich}.}
  \bibinfo{year}{1942}\natexlab{}.
\newblock \showarticletitle{On the translocation of masses}. In
  \bibinfo{booktitle}{\emph{Dokl. Akad. Nauk. USSR (NS)}},
  Vol.~\bibinfo{volume}{37}. \bibinfo{pages}{199--201}.
\newblock


\bibitem[\protect\citeauthoryear{Kitagawa, M{\'e}rigot, and Thibert}{Kitagawa
  et~al\mbox{.}}{2016}]%
        {Kitagawa2016}
\bibfield{author}{\bibinfo{person}{Jun Kitagawa}, \bibinfo{person}{Quentin
  M{\'e}rigot}, {and} \bibinfo{person}{Boris Thibert}.}
  \bibinfo{year}{2016}\natexlab{}.
\newblock \showarticletitle{Convergence of a Newton algorithm for semi-discrete
  optimal transport}.
\newblock \bibinfo{journal}{\emph{arXiv preprint arXiv:1603.05579}}
  (\bibinfo{year}{2016}).
\newblock


\bibitem[\protect\citeauthoryear{Klein}{Klein}{1967}]%
        {klein1967primal}
\bibfield{author}{\bibinfo{person}{Morton Klein}.}
  \bibinfo{year}{1967}\natexlab{}.
\newblock \showarticletitle{A primal method for minimal cost flows with
  applications to the assignment and transportation problems}.
\newblock \bibinfo{journal}{\emph{Management Science}} \bibinfo{volume}{14},
  \bibinfo{number}{3} (\bibinfo{year}{1967}), \bibinfo{pages}{205--220}.
\newblock


\bibitem[\protect\citeauthoryear{Lavenant}{Lavenant}{2017}]%
        {Lavenant2017}
\bibfield{author}{\bibinfo{person}{Hugo Lavenant}.}
  \bibinfo{year}{2017}\natexlab{}.
\newblock \showarticletitle{Harmonic mappings valued in the {W}asserstein
  space}.
\newblock \bibinfo{journal}{\emph{arXiv:1712.07528}} (\bibinfo{year}{2017}).
\newblock


\bibitem[\protect\citeauthoryear{L{\'e}vy}{L{\'e}vy}{2015}]%
        {Levy2015}
\bibfield{author}{\bibinfo{person}{Bruno L{\'e}vy}.}
  \bibinfo{year}{2015}\natexlab{}.
\newblock \showarticletitle{A numerical algorithm for {$L_2$} semi-discrete
  optimal transport in 3D}.
\newblock \bibinfo{journal}{\emph{ESAIM: Mathematical Modelling and Numerical
  Analysis}} \bibinfo{volume}{49}, \bibinfo{number}{6} (\bibinfo{year}{2015}),
  \bibinfo{pages}{1693--1715}.
\newblock


\bibitem[\protect\citeauthoryear{Li, Ryu, Osher, Yin, and Gangbo}{Li
  et~al\mbox{.}}{2018}]%
        {Li2018}
\bibfield{author}{\bibinfo{person}{Wuchen Li}, \bibinfo{person}{Ernest~K Ryu},
  \bibinfo{person}{Stanley Osher}, \bibinfo{person}{Wotao Yin}, {and}
  \bibinfo{person}{Wilfrid Gangbo}.} \bibinfo{year}{2018}\natexlab{}.
\newblock \showarticletitle{A parallel method for earth mover's distance}.
\newblock \bibinfo{journal}{\emph{Journal of Scientific Computing}}
  \bibinfo{volume}{75}, \bibinfo{number}{1} (\bibinfo{year}{2018}),
  \bibinfo{pages}{182--197}.
\newblock


\bibitem[\protect\citeauthoryear{Lu}{Lu}{2017}]%
        {Lu2017}
\bibfield{author}{\bibinfo{person}{Zhuoran Lu}.}
  \bibinfo{year}{2017}\natexlab{}.
\newblock \emph{\bibinfo{title}{Properties of Soft Maps on {R}iemannian
  Manifolds}}.
\newblock \bibinfo{thesistype}{Ph.D. Dissertation}. \bibinfo{school}{New York
  University}.
\newblock


\bibitem[\protect\citeauthoryear{Maas}{Maas}{2011}]%
        {Maas2011}
\bibfield{author}{\bibinfo{person}{Jan Maas}.} \bibinfo{year}{2011}\natexlab{}.
\newblock \showarticletitle{Gradient flows of the entropy for finite {M}arkov
  chains}.
\newblock \bibinfo{journal}{\emph{Journal of Functional Analysis}}
  \bibinfo{volume}{261}, \bibinfo{number}{8} (\bibinfo{year}{2011}),
  \bibinfo{pages}{2250--2292}.
\newblock


\bibitem[\protect\citeauthoryear{Maury, Roudneff-Chupin, and
  Santambrogio}{Maury et~al\mbox{.}}{2010}]%
        {Maury2010}
\bibfield{author}{\bibinfo{person}{Bertrand Maury}, \bibinfo{person}{Aude
  Roudneff-Chupin}, {and} \bibinfo{person}{Filippo Santambrogio}.}
  \bibinfo{year}{2010}\natexlab{}.
\newblock \showarticletitle{A macroscopic crowd motion model of gradient flow
  type}.
\newblock \bibinfo{journal}{\emph{Mathematical Models and Methods in Applied
  Sciences}} \bibinfo{volume}{20}, \bibinfo{number}{10} (\bibinfo{year}{2010}),
  \bibinfo{pages}{1787--1821}.
\newblock


\bibitem[\protect\citeauthoryear{McCann}{McCann}{1997}]%
        {Mccann1997}
\bibfield{author}{\bibinfo{person}{Robert~J McCann}.}
  \bibinfo{year}{1997}\natexlab{}.
\newblock \showarticletitle{A convexity principle for interacting gases}.
\newblock \bibinfo{journal}{\emph{Advances in Mathematics}}
  \bibinfo{volume}{128}, \bibinfo{number}{1} (\bibinfo{year}{1997}),
  \bibinfo{pages}{153--179}.
\newblock


\bibitem[\protect\citeauthoryear{M{\'e}rigot}{M{\'e}rigot}{2011}]%
        {Merigot2011}
\bibfield{author}{\bibinfo{person}{Quentin M{\'e}rigot}.}
  \bibinfo{year}{2011}\natexlab{}.
\newblock \showarticletitle{A multiscale approach to optimal transport}. In
  \bibinfo{booktitle}{\emph{Computer Graphics Forum}},
  Vol.~\bibinfo{volume}{30}. Wiley Online Library, \bibinfo{pages}{1583--1592}.
\newblock


\bibitem[\protect\citeauthoryear{M{\'e}rigot, Meyron, and Thibert}{M{\'e}rigot
  et~al\mbox{.}}{2018}]%
        {Merigot2018}
\bibfield{author}{\bibinfo{person}{Quentin M{\'e}rigot},
  \bibinfo{person}{Jocelyn Meyron}, {and} \bibinfo{person}{Boris Thibert}.}
  \bibinfo{year}{2018}\natexlab{}.
\newblock \showarticletitle{An algorithm for optimal transport between a
  simplex soup and a point cloud}.
\newblock \bibinfo{journal}{\emph{SIAM Journal on Imaging Sciences}}
  \bibinfo{volume}{11}, \bibinfo{number}{2} (\bibinfo{year}{2018}),
  \bibinfo{pages}{1363--1389}.
\newblock


\bibitem[\protect\citeauthoryear{M{\'e}rigot and Mirebeau}{M{\'e}rigot and
  Mirebeau}{2016}]%
        {Merigot2016}
\bibfield{author}{\bibinfo{person}{Quentin M{\'e}rigot} {and}
  \bibinfo{person}{Jean-Marie Mirebeau}.} \bibinfo{year}{2016}\natexlab{}.
\newblock \showarticletitle{Minimal geodesics along volume-preserving maps,
  through semidiscrete optimal transport}.
\newblock \bibinfo{journal}{\emph{SIAM J. Numer. Anal.}} \bibinfo{volume}{54},
  \bibinfo{number}{6} (\bibinfo{year}{2016}), \bibinfo{pages}{3465--3492}.
\newblock


\bibitem[\protect\citeauthoryear{{MOSEK ApS}}{{MOSEK ApS}}{2017}]%
        {mosek}
\bibfield{author}{\bibinfo{person}{{MOSEK ApS}}.}
  \bibinfo{year}{2017}\natexlab{}.
\newblock \bibinfo{booktitle}{\emph{The {MOSEK} optimization toolbox for
  {MATLAB} manual}}.
\newblock


\bibitem[\protect\citeauthoryear{Orlin}{Orlin}{1997}]%
        {orlin1997polynomial}
\bibfield{author}{\bibinfo{person}{James~B Orlin}.}
  \bibinfo{year}{1997}\natexlab{}.
\newblock \showarticletitle{A polynomial time primal network simplex algorithm
  for minimum cost flows}.
\newblock \bibinfo{journal}{\emph{Mathematical Programming}}
  \bibinfo{volume}{78}, \bibinfo{number}{2} (\bibinfo{year}{1997}),
  \bibinfo{pages}{109--129}.
\newblock


\bibitem[\protect\citeauthoryear{Otto}{Otto}{2001}]%
        {Otto2001}
\bibfield{author}{\bibinfo{person}{Felix Otto}.}
  \bibinfo{year}{2001}\natexlab{}.
\newblock \showarticletitle{The geometry of dissipative evolution equations:
  the porous medium equation}.
\newblock  (\bibinfo{year}{2001}).
\newblock


\bibitem[\protect\citeauthoryear{Panozzo, Baran, Diamanti, and
  Sorkine-Hornung}{Panozzo et~al\mbox{.}}{2013}]%
        {panozzo2013weighted}
\bibfield{author}{\bibinfo{person}{Daniele Panozzo}, \bibinfo{person}{Ilya
  Baran}, \bibinfo{person}{Olga Diamanti}, {and} \bibinfo{person}{Olga
  Sorkine-Hornung}.} \bibinfo{year}{2013}\natexlab{}.
\newblock \showarticletitle{Weighted averages on surfaces}.
\newblock \bibinfo{journal}{\emph{ACM Transactions on Graphics (TOG)}}
  \bibinfo{volume}{32}, \bibinfo{number}{4} (\bibinfo{year}{2013}),
  \bibinfo{pages}{60}.
\newblock


\bibitem[\protect\citeauthoryear{Papadakis, Peyr{\'e}, and Oudet}{Papadakis
  et~al\mbox{.}}{2014}]%
        {Papadakis2014}
\bibfield{author}{\bibinfo{person}{Nicolas Papadakis}, \bibinfo{person}{Gabriel
  Peyr{\'e}}, {and} \bibinfo{person}{Edouard Oudet}.}
  \bibinfo{year}{2014}\natexlab{}.
\newblock \showarticletitle{Optimal transport with proximal splitting}.
\newblock \bibinfo{journal}{\emph{SIAM Journal on Imaging Sciences}}
  \bibinfo{volume}{7}, \bibinfo{number}{1} (\bibinfo{year}{2014}),
  \bibinfo{pages}{212--238}.
\newblock


\bibitem[\protect\citeauthoryear{Pinkall and Polthier}{Pinkall and
  Polthier}{1993}]%
        {Pinkall1993}
\bibfield{author}{\bibinfo{person}{Ulrich Pinkall} {and}
  \bibinfo{person}{Konrad Polthier}.} \bibinfo{year}{1993}\natexlab{}.
\newblock \showarticletitle{Computing discrete minimal surfaces and their
  conjugates}.
\newblock \bibinfo{journal}{\emph{Experimental Mathematics}}
  \bibinfo{volume}{2}, \bibinfo{number}{1} (\bibinfo{year}{1993}),
  \bibinfo{pages}{15--36}.
\newblock


\bibitem[\protect\citeauthoryear{Polthier and Schmies}{Polthier and
  Schmies}{2006}]%
        {Polthier2006}
\bibfield{author}{\bibinfo{person}{Konrad Polthier} {and}
  \bibinfo{person}{Markus Schmies}.} \bibinfo{year}{2006}\natexlab{}.
\newblock \bibinfo{booktitle}{\emph{Straightest geodesics on polyhedral
  surfaces}}.
\newblock \bibinfo{publisher}{ACM}.
\newblock


\bibitem[\protect\citeauthoryear{Santambrogio}{Santambrogio}{2015}]%
        {Santambrogio2015}
\bibfield{author}{\bibinfo{person}{Filippo Santambrogio}.}
  \bibinfo{year}{2015}\natexlab{}.
\newblock \showarticletitle{Optimal transport for applied mathematicians}.
\newblock \bibinfo{journal}{\emph{Birk{\"a}user, NY}} (\bibinfo{year}{2015}),
  \bibinfo{pages}{99--102}.
\newblock


\bibitem[\protect\citeauthoryear{Santambrogio}{Santambrogio}{2018}]%
        {Santambrogio2018}
\bibfield{author}{\bibinfo{person}{Filippo Santambrogio}.}
  \bibinfo{year}{2018}\natexlab{}.
\newblock \bibinfo{title}{Crowd motion and evolution {PDE}s under density
  constraints}.
\newblock
\newblock


\bibitem[\protect\citeauthoryear{Solomon, De~Goes, Peyr{\'e}, Cuturi, Butscher,
  Nguyen, Du, and Guibas}{Solomon et~al\mbox{.}}{2015}]%
        {Solomon2015}
\bibfield{author}{\bibinfo{person}{Justin Solomon}, \bibinfo{person}{Fernando
  De~Goes}, \bibinfo{person}{Gabriel Peyr{\'e}}, \bibinfo{person}{Marco
  Cuturi}, \bibinfo{person}{Adrian Butscher}, \bibinfo{person}{Andy Nguyen},
  \bibinfo{person}{Tao Du}, {and} \bibinfo{person}{Leonidas Guibas}.}
  \bibinfo{year}{2015}\natexlab{}.
\newblock \showarticletitle{Convolutional {W}asserstein distances: Efficient
  optimal transportation on geometric domains}.
\newblock \bibinfo{journal}{\emph{ACM Transactions on Graphics (TOG)}}
  \bibinfo{volume}{34}, \bibinfo{number}{4} (\bibinfo{year}{2015}),
  \bibinfo{pages}{66}.
\newblock


\bibitem[\protect\citeauthoryear{Solomon, Guibas, and Butscher}{Solomon
  et~al\mbox{.}}{2013}]%
        {Solomon2013}
\bibfield{author}{\bibinfo{person}{Justin Solomon}, \bibinfo{person}{Leonidas
  Guibas}, {and} \bibinfo{person}{Adrian Butscher}.}
  \bibinfo{year}{2013}\natexlab{}.
\newblock \showarticletitle{Dirichlet energy for analysis and synthesis of soft
  maps}. In \bibinfo{booktitle}{\emph{Computer Graphics Forum}},
  Vol.~\bibinfo{volume}{32}. Wiley Online Library, \bibinfo{pages}{197--206}.
\newblock


\bibitem[\protect\citeauthoryear{Solomon, Nguyen, Butscher, Ben-Chen, and
  Guibas}{Solomon et~al\mbox{.}}{2012}]%
        {solomon2012soft}
\bibfield{author}{\bibinfo{person}{Justin Solomon}, \bibinfo{person}{Andy
  Nguyen}, \bibinfo{person}{Adrian Butscher}, \bibinfo{person}{Mirela
  Ben-Chen}, {and} \bibinfo{person}{Leonidas Guibas}.}
  \bibinfo{year}{2012}\natexlab{}.
\newblock \showarticletitle{Soft maps between surfaces}. In
  \bibinfo{booktitle}{\emph{Computer Graphics Forum}},
  Vol.~\bibinfo{volume}{31}. Wiley Online Library, \bibinfo{pages}{1617--1626}.
\newblock


\bibitem[\protect\citeauthoryear{Solomon, Rustamov, Guibas, and
  Butscher}{Solomon et~al\mbox{.}}{2014}]%
        {Solomon2014}
\bibfield{author}{\bibinfo{person}{Justin Solomon}, \bibinfo{person}{Raif
  Rustamov}, \bibinfo{person}{Leonidas Guibas}, {and} \bibinfo{person}{Adrian
  Butscher}.} \bibinfo{year}{2014}\natexlab{}.
\newblock \showarticletitle{Earth mover's distances on discrete surfaces}.
\newblock \bibinfo{journal}{\emph{ACM Transactions on Graphics (TOG)}}
  \bibinfo{volume}{33}, \bibinfo{number}{4} (\bibinfo{year}{2014}),
  \bibinfo{pages}{67}.
\newblock


\bibitem[\protect\citeauthoryear{Trillos}{Trillos}{2017}]%
        {Trillos2017}
\bibfield{author}{\bibinfo{person}{Nicolas~Garcia Trillos}.}
  \bibinfo{year}{2017}\natexlab{}.
\newblock \showarticletitle{Gromov-{H}ausdorff limit of {W}asserstein spaces on
  point clouds}.
\newblock \bibinfo{journal}{\emph{arXiv:1702.03464}} (\bibinfo{year}{2017}).
\newblock


\bibitem[\protect\citeauthoryear{V{\'a}zquez}{V{\'a}zquez}{2007}]%
        {Vazquez2007}
\bibfield{author}{\bibinfo{person}{Juan~Luis V{\'a}zquez}.}
  \bibinfo{year}{2007}\natexlab{}.
\newblock \bibinfo{booktitle}{\emph{The Porous Medium Equation: Mathematical
  Theory}}.
\newblock \bibinfo{publisher}{Oxford University Press}.
\newblock


\bibitem[\protect\citeauthoryear{Villani}{Villani}{2003}]%
        {Villani2003}
\bibfield{author}{\bibinfo{person}{C{\'e}dric Villani}.}
  \bibinfo{year}{2003}\natexlab{}.
\newblock \bibinfo{booktitle}{\emph{Topics in Optimal Transportation}}.
\newblock Number~58. \bibinfo{publisher}{American Mathematical Soc.}
\newblock


\bibitem[\protect\citeauthoryear{Villani}{Villani}{2008}]%
        {Villani2008}
\bibfield{author}{\bibinfo{person}{C{\'e}dric Villani}.}
  \bibinfo{year}{2008}\natexlab{}.
\newblock \bibinfo{booktitle}{\emph{Optimal Transport: Old and New}}.
  Vol.~\bibinfo{volume}{338}.
\newblock \bibinfo{publisher}{Springer Science \& Business Media}.
\newblock


\end{thebibliography}

% \newpage

% \newpage

% \input{section_appendix}

\end{document}